\documentclass[twoside,11pt]{article}

\usepackage{thmtools}
\usepackage{thm-restate}
\usepackage[abbrvbib, preprint]{jmlr2e}
\usepackage[utf8]{inputenc} 
\usepackage[T1]{fontenc}    
\usepackage{hyperref}       
\usepackage{url}            
\usepackage{booktabs}       
\usepackage{amsfonts}       
\usepackage{nicefrac}       
\usepackage{microtype}      

\usepackage{enumitem}       
\makeatletter
\def\namedlabel#1#2{\begingroup
    #2%
    \def\@currentlabel{#2}%
    \phantomsection\label{#1}\endgroup
}
\makeatother
\usepackage{hyperref}       
\usepackage{url}            
\usepackage{booktabs}       
\usepackage{amsfonts}       
\usepackage{nicefrac}       
\usepackage{microtype}      

\usepackage{bbm} 

\usepackage{amsmath,amssymb}
\usepackage{graphicx}
\usepackage{comment}
\usepackage{ifthen, calc}
\usepackage{xcolor}
\usepackage{pstricks}
\usepackage{pst-grad} 
\usepackage{pst-plot} 
\usepackage{hyperref}
\usepackage{mathrsfs}
\usepackage{mathtools}
\usepackage{multicol}
\usepackage{relsize}
\usepackage{tikz}
\usepackage{caption}
\usepackage{subcaption}

\usepackage{accents}

\newcommand\Sc {\mathcal{S}}
\newcommand\Nc {\mathcal{N}}

\newcommand\Norm {\mathcal{N}}

\newcommand\E {\mathbb{E}}

\newcommand\Pbb {\mathbb{P}}
\newcommand\R {\mathbb{R}}

\newcommand\Var {\ensuremath{\operatorname{\mathrm{Var}}}}

\newcommand\eps {\varepsilon}

\newcommand\tr {\ensuremath{\operatorname{\mathrm{tr}}}}
\newcommand\diag {\ensuremath{\operatorname{\mathrm{diag}}}}

\newcommand\uptok {{0:k}}
\newcommand\ktoinf {{k:\infty}}

\usepackage{textcomp}

\usepackage{lipsum}

\DeclareMathOperator*{\argmin}{argmin}

\jmlrheading{}{2021}{}{}{}{}{Alexander Tsigler and Peter L. Bartlett}

\ShortHeadings{Benign Overfitting in Ridge Regression}{Tsigler and Bartlett}
\firstpageno{1}

\usepackage{times}

\newenvironment{proofsketch}{\par\noindent{\bf Proof sketch\ }}{\hfill\BlackBox\\[2mm]}

\author{\name Alexander Tsigler \email alexander\_tsigler@berkeley.edu\\
 \addr Department of Statistics\\
  University of California, Berkeley\\
  367 Evans Hall, Berkeley, CA 94720-3860\\
 \AND
 \name Peter L.~Bartlett  \email peter@berkeley.edu\\
 \addr Departments of Statistics and Computer Science\\
  University of California, Berkeley and Google Research\\ 367 Evans Hall, Berkeley, CA 94720-3860
}

\editor{}

\begin{document}
\title{Benign overfitting in ridge regression}
\maketitle
\begin{abstract}%
In many modern applications of deep learning the neural network has many more parameters than the data points used for its training. Motivated by those practices, a large body of recent theoretical research has been devoted to studying overparameterized models. One of the central phenomena in this regime is the ability of the model to interpolate noisy data, but still have test error lower than the amount of noise in that data. \cite{benign_overfitting} characterized for which covariance structure of the data such a phenomenon can happen in linear regression if one considers the interpolating solution with minimum $\ell_2$-norm and the data has independent components: they gave a sharp bound on the variance term and showed that it can be small if and only if the data covariance has high effective rank in a subspace of small co-dimension. We strengthen and complete their results by eliminating the independence assumption and providing sharp bounds for the bias term. Thus, our results  apply in a much more general setting than those of \cite{benign_overfitting}, e.g., kernel regression, and not only characterize how the noise is damped but also which part of the true signal is learned.  Moreover, we extend the result to the setting of ridge regression, which allows us to explain another interesting phenomenon: we give general sufficient conditions under which the optimal regularization is negative. 

\end{abstract}

\section{Introduction}
\label{sec::introduction}
\subsection{Motivation and our contribution}
The bias-variance tradeoff is well known in statistics and machine learning. The classical theory suggests that large models overfit the data and that one needs significant regularization to make them generalize. This intuition is, however, in contrast with the empirical study of modern machine learning techniques. It was repeatedly observed that even models with enough capacity to exactly interpolate the data can generalize with little regularization, or no regularization at all \citep{Belkin15849, ZhangBHRV16}. In some cases, the best value of the regularizer can be zero \citep{Liang2018JustIK} or even negative \citep{Kobak2020OptimalRP} for such models. 

The aim of this paper is to provide a theoretical understanding of these phenomena, and to do that we consider 
one of the simplest settings
in which they can be observed---ridge regression in dimension $p$ with $n < p$ i.i.d. noisy observations. Despite being a classical statistical 
methodology, ridge regression and its ridgeless limit are still not completely understood in such a regime: when $n < p$ classical theory suggests that the regularization parameter should be large enough to provide additional capacity control (see, e.g., \cite{Hsu_random_design_ridge} and references therein).  The basis of our work was set by  \cite{benign_overfitting}, who studied  the variance term for ridgeless regression with $n < p$ under the additional assumption that the data vectors have independent components. The main discovery of their work is that the variance term can be small if and only if there exists $k^* \ll n$ such that if one removes the first $k^*$ largest eigenvalues of the covariance operator, the remaining tail of the sequence of eigenvalues  has large effective rank compared to $n$. In our work we start afresh and use the same separation of eigendirections from the very beginning, which allows us to substitute the independence assumption by a weaker assumption on the condition number of the Gram matrix of the tails of the data vectors. Moreover, we show how the same separation of the eigenvalues gives tight bounds for the bias term too. Finally, by virtue of algebra, our argument extends very easily to the setting of ridge regression, which allows for comparison with the above mentioned classical results and investigation of the case when the regularization is even less than zero. We show that we extend  (with different constants) the results of \cite{Hsu_random_design_ridge} to a larger range of regularization parameters, and give general conditions under which negative regularization is optimal and can provide arbitrarily high multiplicative gain in excess risk. 

The structure of the paper is the following: in Section~\ref{sec::related work}, 
we provide an overview of the field of overparameterized ridge regression. We postpone a more technical overview to Section \ref{sec::detailed comparison}, where we also explain how our paper relates to other works. We start the presentation of our results with introducing the setting of ridge regression in Section \ref{sec::ridge setup main body}. After that, we use Section \ref{sec::separation story} to introduce the separation of eigendirections and define the relevant important objects: Subsection  \ref{sec::regimes of linear regression} shows two simple sketches aimed at building up intuition, Subsection \ref{sec::benign overfitting story} explains the results of \cite{benign_overfitting} in terms of that intuition and Subsection \ref{sec::our contribution intuition} explains how our work completes the story. 
The aim of this discussion is to elucidate
the meaning behind the rigorous assumptions and results that we show in Section \ref{sec::main result main body}. Then Section \ref{sec:assumption on A_k} provides a more technical discussion of the main assumption. Section \ref{sec::proof structure main body} provides an outline of the proof and explains where it uses the assumption that the data is sub-Gaussian. 
In Section~\ref{sec::alternative form of the bound} we note that as a side product of the proof an alternative form of the main bound arises, which makes it convenient to compare our bounds to the results of other papers.
In Section \ref{sec::negative regularization main body}, we derive the sufficient conditions for optimality of negative regularization. Finally, we conclude the paper with Section \ref{sec::conclusions}.

\subsection{Related work}
\label{sec::related work}

Motivated by the empirical success of overparametrized models, there has recently been a flurry of work aimed at  understanding theoretically whether the corresponding effects can be seen in overparametrized linear regression; see, e.g.,  \citep{Liang2019OnTR, Muthukumar2019HarmlessIO, Belkin2019TwoMO, Bibas,  Nakkiran2019MoreDC, NIPS2019_8753, zhou2021uniform, pmlr-v119-negrea20a} and other references in this section. 

The results that aim at characterizing the generalization performance of linear methods can be split roughly into three categories.
The first category is results that give exact expressions of the excess risk in the asymptotic setting with ambient dimension and the number of data points going to infinity, while their ratio goes to a constant, and the spectral density of the covariance operator converges weakly to some limiting distribution \citep{Dobriban2015HighDimensionalAO, Hastie2019SurprisesIH, wu2020optimal, richards2020asymptotics}. 

The second category is results that make strong assumptions on the distribution of  data (e.g., that data vectors have i.i.d.\ components or come from a uniform distribution on a sphere) and derive bounds on excess risk of  linear regression with some specific features, or kernel regression with a kernel that has some specific properties \citep{ montanari2020interpolation, ghorbani2020neural, mei2019generalization, ghorbani2020linearized,liang2020multiple}. Some of these results are also asymptotic, and some are non-asymptotic.  

The third category is results that prove non-asymptotic bounds depending on the arbitrary structure of the covariance of the data. This is the category to which this paper belongs.  We already mentioned the work of \cite{benign_overfitting}. The other works in this category are  \citep{Kobak2020OptimalRP}, \citep{chinot2021robustness}, \citep{Derezinski} and \citep{derezinski2020precise}.

 We provide more detailed comparison and discuss some technical aspects in Section \ref{sec::detailed comparison}.

There have been many related works since the arXiv version of this paper ~\citep{tb-borr-20} was posted  \citep{mei2021generalization, mei2021learning, ghosh2021stages, misiakiewicz2021learning, bmr-dlsp-21, celentano2021minimum, JMLR:v22:20-603, narang2021classification, mcrae2021harmless, shamir2022implicit, koehler2021uniform, bunea2022interpolating} etc. 
 \citet{hastie2020surprises} obtained a finite sample version of the asymptotic results of the old version of their paper \citep{Hastie2019SurprisesIH}. In Section~\ref{sec::comparison with Hsu and Hastie}  we provide an explicit comparison with our results. More recently, \citet{mei2021generalization} obtained generalization bounds for kernel ridge regression under similar assumptions to those we consider here (see their Assumption 1). \citet{koehler2021uniform} used the idea of separating the firs $k$ eigendirections of the covariance to study excess risk of minimum norm interpolators with arbitrary norms and Gaussian data. \citet{bmr-dlsp-21} obtained results which belong to the intersection of the first and the second categories which we described in Section \ref{sec::related work} (see their 
Theorem 4.1). \citet{shamir2022implicit} constructed an example of a misspecified setting (i.e., the noise is not independent from the data) in which our results don't hold even though the condition number of the matrix $A_k$ is a constant (see their Example 1).

\section{Ridge regression setup}
\label{sec::ridge setup main body}
The learning problem we consider is ridge regression. Its goal is to learn an unknown real-valued function on $\R^p$ given noisy observations of its values in  $n$ points. We operate in the overparameterized regime, i.e., $p > n$. 

\subsection{Covariate model}

We assume that the data set consists of $n$ i.i.d. vectors sampled from some distribution on $\R^p$, whose mean is zero. Throughout the paper $x$ denotes an independent draw from that distribution.
Denote $X \in \R^{n\times p}$ to be the matrix whose rows are the (transposed) data vectors.

Our results depend on the spectrum of the covariance matrix $\Sigma = \E[xx^\top]$. We fix an orthonormal basis in which $\Sigma$ is diagonal:
\begin{equation}
\label{eq::lambdas definition}
\Sigma = \diag(\lambda_1, \lambda_2, \dots, \lambda_p),
\end{equation}
where $\lambda_1 \geq \lambda_2 \geq \dots \geq \lambda_p$ is the non-increasing sequence of eigenvalues of $\Sigma$.

We assume sub-Gaussianity: denote $Z := X\Sigma^{-1/2}$ (whitened data matrix). Rows of $Z$ are isotropic centered i.i.d. random vectors. We assume that rows of $Z$ are sub-Gaussian with sub-Gaussian norm $\sigma_x$ as defined in Appendix \ref{sec::sub-Gaussianity appendix}. 

Sub-Gaussianity is a classical assumption, which provides a convenient framework for controlling deviations of various quantities of interest (see \cite{vershynin_hdp} for an introduction). We discuss whether it is actually needed in Section \ref{sec::role of sub-Gaussianity}.

\subsection{Response model}
Denote $y\in \R^n$ to be the vector whose coordinates are noisy measurements of the values of an unknown function in the corresponding data points. We assume that the true function is linear with coefficients $\theta^*\in \R^p$, i.e., 
\[
y = X \theta^* + \eps,
\]
where $\eps$ is the noise vector. We assume that  components of $\eps$ are i.i.d. centered random variables with variance $v_\eps^2$. 

\subsection{Learning procedure}
Ridge regression with regularization parameter $\lambda$ is a classical learning algorithm that estimates $\theta^*$ from $X, y$ according to the following formula:
\[
\hat{\theta}( y) := X^\top(X X^\top + \lambda I_n)^{-1}y.
\]
See Appendix \ref{sec::ridge appendix} for a discussion. The matrix $\lambda I_n + XX^\top$ will play an important role in our analysis, so we denote
\[
A := \lambda I_n + XX^\top.
\]
In the ridgeless case ($\lambda = 0$), $A$ is the Gram matrix of the data. Ridge regularization shifts all its eigenvalues by $\lambda$.

\subsection{Excess risk and its bias-variance decomposition}
The quantity of interest is excess risk
that we define in the following way: recall that $x$ is a new data point from the same distribution as rows of $X$. The error that our predictor incurs on this data point is $x^\top(\hat{\theta}(y) - \theta^*)$. We define excess risk as the average squared error over the population, i.e.,
\[
\E_x\left[(x^\top(\hat{\theta}(y) - \theta^*))^2\right] = \|\hat{\theta}(y) - \theta^*\|_{\Sigma}^2,
\]
where we define $\|x\|_M := \sqrt{x^\top M x}$ for any positive semi-definite (PSD) matrix $M$ and any vector $x$ of the corresponding dimension.

Note that $\hat{\theta}( y)$ is linear in $y$, which allows us to write
\begin{gather*}
\hat{\theta}( y) = \hat{\theta}(X\theta^*) +  \hat{\theta}( \eps),\\
\E_\eps\left[\|\hat{\theta}(y) - \theta^*\|_{\Sigma}^2\right] = \|\hat{\theta}(X\theta^*) - \theta^*\|_{\Sigma}^2 + \E_\eps\left[\|\hat{\theta}(\eps) \|_{\Sigma}^2\right],\\
\|\hat{\theta}(y) - \theta^*\|_{\Sigma}^2 \leq 2(\|\hat{\theta}(X\theta^*) - \theta^*\|_{\Sigma}^2 + \|\hat{\theta}(\eps) \|_{\Sigma}^2).
\end{gather*}

The term $\|\hat{\theta}(X\theta^*) - \theta^*\|_{\Sigma}^2$ is the error in the noiseless regime; it is caused by rows of $X$ not spanning the whole space and by regularization. The term $\|\hat{\theta}(\eps) \|_{\Sigma}^2$ is the error of learning the zero function from pure noise. One can see that these two terms nicely decouple from each other and can be studied separately. Moreover, note that $\|\hat{\theta}(\eps) \|_{\Sigma}^2$ is a quadratic form in $\eps$. Its expectation scales linearly with $v_\eps^2$ (variance of the noise):
\[
\E_\eps\left[\|\hat{\theta}(\eps) \|_{\Sigma}^2\right] = v_\eps^2\tr(A^{-1} X\Sigma X^\top A^{-1}).
\]
If the noise is sub-Gaussian with sub-Gaussian norm $\sigma_\eps$, then by Lemma \ref{lm::quadratic form concentration} from the appendix for some absolute constant $c$ and any $t > 1$, with probability at least $1-ce^{-n/c}$,
\begin{align*}
\|\hat{\theta}(\eps) \|_{\Sigma}^2 =& \eps^\top A^{-1} X\Sigma X^\top A^{-1}\eps\\
\leq& ct\sigma_\eps^2\tr(A^{-1} X\Sigma X^\top A^{-1}).
\end{align*}

Therefore, both expectation and deviations of the term $\|\hat{\theta}(\eps) \|_{\Sigma}^2$ are controlled by the quantity $\tr(A^{-1} X\Sigma X^\top A^{-1})$. Thus, we define:
\begin{equation}
\label{eq::bias-variance definition}
\begin{array}{rcccl}
B := & \|\hat{\theta}(X\theta^*) - \theta^*\|_{\Sigma}^2 & =&\|(X^\top A^{-1} X - I_p)\theta^*\|^2_{\Sigma}&\text{ --- bias,}\\[1mm]
V := & \E_\eps\left[\|\hat{\theta}(\eps) \|_{\Sigma}^2/v_\eps^2\right] &= &\tr(A^{-1} X\Sigma X^\top A^{-1})&\text{ --- variance.}
\end{array}
\end{equation}

These quantities don't depend on the distribution of the noise. The goal of this paper is to provide sharp non-asymptotic bounds for them.

\section{The story of separating the first $k$ eigendirections and our contribution}
\label{sec::separation story}
\subsection{Essentially high-dimensional linear regression vs. essentially low-dimensional}
\label{sec::regimes of linear regression}
Before we present our results, we develop some intuition by considering two easy scenarios: "essentially low-dimensional" and "essentially high-dimensional". For each scenario we will do an informal computation of the excess risk and give a geometric interpretation. 

\begin{itemize}
\item {\bf  Essentially low-dimensional linear regression.} Consider least squares regression in which  data lives in $\R^k$ and $k \ll n$: $X \in \R^{n \times k}$ with i.i.d. centered rows from a distribution with covariance $\Sigma\in \R^{k\times k}$ and $y = X\theta^* + \eps,$ where $\eps$ has i.i.d. centered components with variances $v_\eps^2$. Our estimator of choice in this regime is OLS:
\[
\hat{\theta} = \arg \min_{\theta} \|X\theta - y\|^2 = \arg \min_{\theta} \|X(\theta - \theta^*) - \eps\|^2
\]
As $\theta$ takes all possible values in $\R^k$, $X(\theta - \theta^*)$ takes all possible values in the span of columns of $X$, which means that
\[
X(\hat\theta - \theta^*) = \Pi_X \eps,
\]
where $\Pi_X$ is the projection on the span of columns of $X$.
This allows us to write the following informal computation, which leads to the classical $k/n$ rate:
\[
v_\eps^2k = \E_\eps\|\Pi_X \eps\|^2 = \E_\eps\|X(\hat\theta - \theta^*)\|^2 =  \E_\eps\left[(\hat\theta - \theta^*)^\top \underbrace{X^\top X}_{\approx n \Sigma}  (\hat\theta - \theta^*)\right],
\]
\[
v_\eps^2 \cdot k/n \approx (\hat\theta - \theta^*)^\top \Sigma (\hat\theta - \theta^*) = \E_{x \sim \Nc(0, \Sigma)} \langle x, \hat\theta - \theta^*\rangle^2.
\]

Here we used the informal transition $\|n^{-1}X^\top X -  \Sigma\|\approx 0$ --- the population covariance matrix is well-approximated by the sample covariance matrix uniformly in all directions.
If $k \ll n$ this results holds with very few additional assumptions (see \citep{10.1093/imrn/rnx067} and references therein). 

What we have obtained is an example of a classical argument:  the training error $\|X(\hat\theta - \theta^*)\|^2$ is a good proxy for the population error $\|\Sigma^{1/2}(\hat\theta - \theta^*)\|^2$ uniformly over all $\hat\theta \in \R^k$, and the model helps eliminate the noise because it gets projected on a subspace of low dimension. The larger the model, the more error comes from the noise. 

Such a result leads to a classical bias-variance trade-off: the larger the model is, the better it can approximate the true dependence, but also the more noise it picks up. A classical cartoon is shown in Figure \ref{fig:learning cartoon}: Figures \ref{fig:underfitting}--\ref{fig:overfitting} show the result of performing least squares regression with features $\{\cos(m\pi x)\}_{m=0}^p$. As the number of features grows, the ability of the model to approximate the signal grows too, but at the cost of increasing sensitivity to the noise. As the number of features approaches  the number of data points (the "interpolation threshold"), this leads to overfitting.

\item {\bf Essentially high-dimensional linear regression.} Now consider linear regression in which $p \gg n$, but with isotropic data: assume that the matrix $X$ has i.i.d. standard normal elements and $y = X\theta_*+ \eps$ where $\eps \sim \Norm(0_n, v_\eps^2I_n)$ --- independent from $X$.  We consider the minimum $\ell_2$-norm interpolating solution:
\[
\hat\theta= \argmin_{\theta \in \R^p: X\theta = y}\|\theta\| = X^\top(XX^\top)^{-1}y = X^\top(XX^\top)^{-1}(X\theta^* + \eps).
\]

According to our definitions of bias and variance from Equation \eqref{eq::bias-variance definition}
with $\lambda=0$,
\begin{align*}
B =& \|\bigl(I_p - X^\top (XX^\top)^{-1}X\bigr) \theta^*\|,\\
V =& \E_\eps\|X^\top (XX^\top)^{-1}\eps\|^2/v_\eps^2 = \tr\bigl(\underbrace{XX^\top}_{\approx p I_n})^{-1}.\\
\end{align*}

Here we see the following:  the matrix $ X^\top (XX^\top)^{-1}X$ is the projection on the span of the data. This is a random $n$-dimensional subspace in $p$-dimensional space. Thus, with high probability $\|X^\top (XX^\top)^{-1}X \theta^*\|^2/\|\theta^*\|^2 \approx n/p$, so the projection only preserves an $n/p$ fraction of the energy of the signal. When it comes to the variance term, we can use the same concentration result for the sample covariance as we did in the low-dimensional case, but for the transposed data matrix, meaning $XX^\top \approx pI_n$. Finishing the computation yields
\[
B \approx (1 - n/p) \|\theta^*\|^2, \quad \E_\eps V \approx  n/p.
\]
We see that the signal is almost not learned at all in this regime (the bias term is close to the full energy of the signal), but  the noise is also damped by the factor $p/n$.

The geometric interpretation is as follows:  if $p \gg n$, the span of $n$ data points is almost orthogonal to $\theta^*$ with high probability. The data just does not measure $\theta^*$ in most directions, so almost the whole signal is lost. On the other hand, despite the noise fully propagating into in-sample predictions, a new data point $x$ is also almost orthogonal to all the old data points with high probability, so those noisy predictions don't influence the prediction in $x$. Overall, despite interpolating the data, we effectively learn a zero estimate out of sample. The zero estimator can be a very good estimator, e.g., if the true signal is zero. This hints at the possibility of learning via high-dimensional interpolation:  the model can use the directions in which the signal is not learned to smear the noise over them. 

The learning cartoon for this regime is given in Figures \ref{fig:isotropic overparameterization}--\ref{fig:benign overfitting}: as the number of cosine features becomes large compared to the number of data points, the learning procedure predicts zero out of sample, despite interpolating the values in sample. However, if we add multiplicative weights to the cosine features, down-weighting higher frequencies, it causes the minimum norm solution to learn the low frequency signal and interpolate the noise using the high frequency components. 

\end{itemize}

\begin{figure}
    \def\figscale{0.48}
     \centering
     \begin{subfigure}[t]{\figscale\textwidth}
         \centering
         \includegraphics[width=\textwidth]{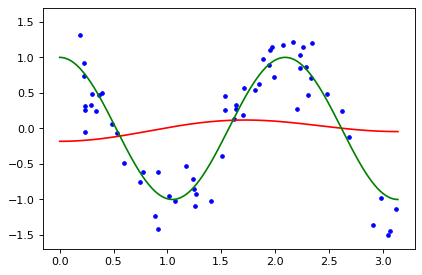}
         \caption{Features $\{\cos(mx)\}_{m=1}^2$: underfitting. A linear combination of features cannot approximate the true dependence.}
         \label{fig:underfitting}
     \end{subfigure}
     \hfill
     \begin{subfigure}[t]{\figscale\textwidth}
         \centering
         \includegraphics[width=\textwidth]{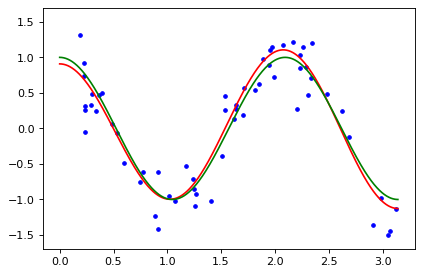}
         \caption{Features $\{\cos(mx)\}_{m=1}^3$: the best fit.  This is the minimum number of features that span the true dependence.}
         \label{fig:best fit}
     \end{subfigure}
     \newline
     \begin{subfigure}[t]{\figscale\textwidth}
         \centering
         \includegraphics[width=\textwidth]{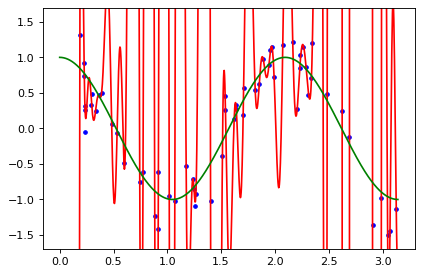}
         \caption{Features $\{\cos(mx)\}_{m=1}^{50}$: overfitting. As the number of features approaches the number of data points, the effect of the noise becomes stronger.}
         \label{fig:overfitting}
     \end{subfigure}
     \hfill
     \begin{subfigure}[t]{\figscale\textwidth}
         \centering
         \includegraphics[width=\textwidth]{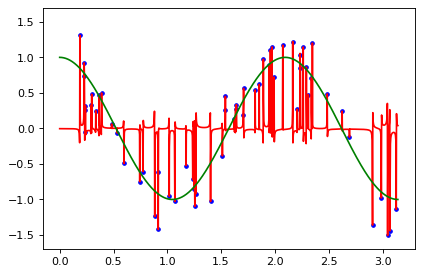}
        \caption{Features $\{\cos(mx)\}_{m=1}^{2000}$: isotropic overparameterization. As the number of cosine features grows above the interpolation threshold, the learned solution goes to zero out of sample. }
        \label{fig:isotropic overparameterization}
    \end{subfigure}
    \newline
    \begin{subfigure}[t]{\figscale\textwidth}
         \centering
         \includegraphics[width=\textwidth]{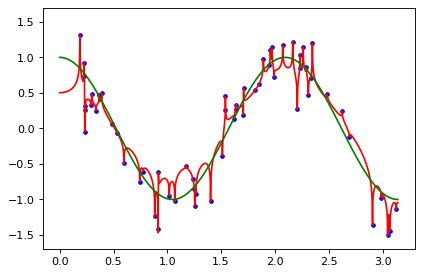}
         \caption{Features $\{\cos(mx)/m\}_{m=1}^{2000}$: benign overfitting. Adding weights to cosine features results in interpolating the noise with high frequency features and learning the signal with low frequency features.}
         \label{fig:benign overfitting}
     \end{subfigure}
     \hfill
     \begin{subfigure}[t]{\figscale\textwidth}
         \centering
         \includegraphics[width=\textwidth]{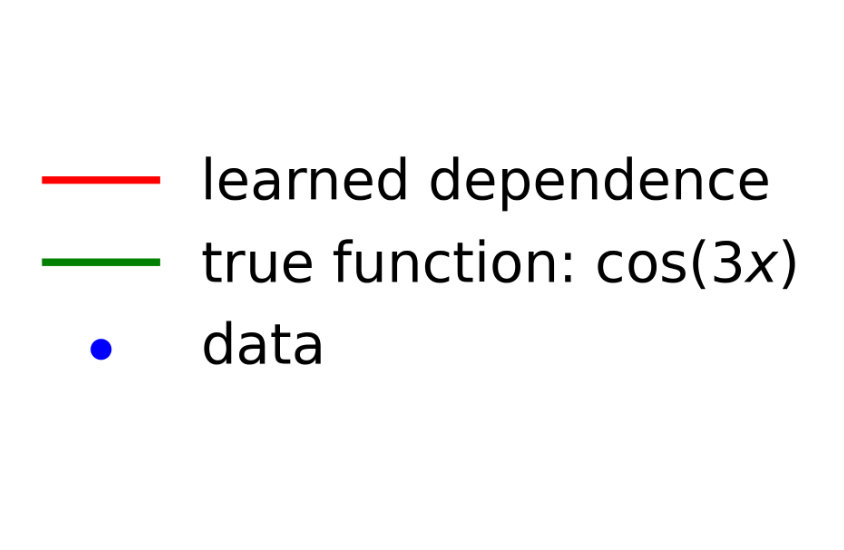}
         \caption{Legend for all the plots.}
         \label{fig:legend}
     \end{subfigure}
    \caption{Learning $\cos(3 x)$ using linear regression with different featurizations. The data points $(x_i, y_i)_{i=1}^{60}$ were generated i.i.d. such that $x_i$ have uniform distribution on $[0, \pi]$ and $y_i$ have normal distribution with mean $\cos(3 x_i)$ and standard deviation $0.4$. The OLS estimator was used when $n$ is larger than the number of features, and the minimum norm interpolating solution was used otherwise.}
    \label{fig:learning cartoon}
\end{figure}

\subsection{The ground provided by the previous work}
\label{sec::benign overfitting story}

 \cite{benign_overfitting} studied  the variance term for ridgeless regression under the additional assumption that the data vectors have independent components. To give an overview of their results, introduce the following quantities: for any $k \in \{0, 1, 2, \dots, p-1\}$ define
 \[
 r_k := \frac{1}{\lambda_{k+1}}\left(\lambda + \sum_{i > k} \lambda_i\right), \quad \rho_k := r_k/n.
 \]
In the ridgeless setting, meaning $\lambda = 0$, $r_0$ is a well-known effective rank of the matrix $\Sigma$,
 $r_k$ is the effective rank of the same matrix, but after restricting it to the span of its last $p-k$ eigenvectors, and $\rho_k$ measures how large that effective rank is compared to the number of data points.  
 
 Given this notation,  \cite{benign_overfitting} defined $k^*$ as the minimum $k$ for which $\rho_k$ is larger than a universal constant. Their result is then that if such a $k^*$ doesn't exist or if $k^*/n$ is at least a constant, then $V$ is lower bounded by a constant. Otherwise, they show that with high probability
 $V$ is equal up to a constant factor to the following quantity:
 \[
 \frac{k^*}{n} + \frac{n\sum_{i > k^*}\lambda_i^2}{\left(\sum_{i > k^*}\lambda_i\right)^2}.
 \]
 
 Inspection of the proof shows that the "essentially low-dimensional" rate $k/n$ comes from the first $k$ components of the vector $\hat\theta(\eps)$ and the term ${\left(n\sum_{i > k^*}\lambda_i^2\right)}/{\left(\sum_{i > k^*}\lambda_i\right)^2}$ comes from the rest of the components of  $\hat\theta(\eps)$. Note that if one plugs in $\lambda_i = \lambda_j$ for all $i, j > k^*$, then it becomes ${\left(n\sum_{i > k^*}\lambda_i^2\right)}/{\left(\sum_{i > k^*}\lambda_i\right)^2} = n/(p-k^*)$ --- exactly the variance term of the "essentially high-dimensional" regime of Section \ref{sec::regimes of linear regression}. The conclusion of  \cite{benign_overfitting} is therefore that  the only way that an interpolating solution can damp the noise by more than a constant factor is the following: the data is such that after removing $k$ components, it becomes "essentially high-dimensional", meaning that the effective rank of its covariance is large  compared to the number of data points. After that the variance in the first $k$ components is the same as for the classical least squares, and the variance in the rest of the components corresponds to the "essentially high-dimensional" case, where you cannot learn but the noise is still damped.  Note, however, that that story was not complete because only the variance term was bounded  sharply in that work.

 \subsection{Our contribution}
 \label{sec::our contribution intuition}
 We complete the story of  \cite{benign_overfitting} by providing sharp bounds on the bias term, extending the results to the setting of ridge regression with nonzero $\lambda$, and replacing the assumption of independence of the components by a much broader sufficient condition. 
 From our point of view, $k^*$ is the main discovery of \cite{benign_overfitting}.  In our work we also start with separation of the first $k$ eigendirections and show that the same split 
 leads to a bound for the bias term that is in full alignment with the intuitive explanation given above. 
 
Let's introduce some notation. Recall that we fixed the basis to be the eigenbasis of the covariance in \eqref{eq::lambdas definition}. For any $k \in \{0, 1, \dots, p\}$ we denote $X_\uptok$ and $Z_\uptok$ to be the matrices comprised of the first $k$ columns of $X$ and $Z$ respectively.\footnote{When $k=0$ these matrices are just empty and all the terms that involve $\uptok$ index become zero.}  Analogously, we denote $X_\ktoinf$ and $Z_\ktoinf$ to be the matrices comprised of the last $p-k$ columns of $X$ and $Z$, $\Sigma_\uptok = \diag(\lambda_1, \dots, \lambda_{k})$ and $\Sigma_\ktoinf = \diag(\lambda_{k+1}, \dots, \lambda_{p}).$ For any $\theta\in \R^p$ we denote $\theta_\uptok$ to be the vector comprised of the first $k$ components of $\theta$, and $\theta_\ktoinf$ --- of the remaining components. We choose the $\ktoinf$ notation instead of $k:p$ to emphasize that our results don't depend on $p$, and only the notions of effective dimension implicitly given by the sequence $\{\lambda_i\}_{i=1}^p$ matter. For example, if one increases the dimension to $p' > p$ and pads the sequence $\{\lambda_i\}_{i=1}^p$ with $p' - p$ zeros, our results will still hold.

The central object in our analysis is the following matrix:
\begin{equation}
\label{eq::A_k definition}
A_k := X_\ktoinf X_\ktoinf^\top + \lambda I_n.
\end{equation}
The matrix $X_\ktoinf X_\ktoinf^\top$ is the Gram matrix of the data after removing the first $k$ components. $A_k$ is obtained from that Gram matrix by shifting all eigenvalues by the ridge regularization parameter $\lambda$. 

In \cite{benign_overfitting}, the crucial step was to show that the singular values of $A_k$ are within a constant factor of each other for $k = k^*$ (see their Lemma 5). When the components of data vectors are independent, such control over the condition number is a consequence of high effective rank. In this paper, the roles of effective rank and condition number of $A_k$ are reversed. We prove sharp bounds assuming that there is some oracle that guarantees that with high probability all eigenvalues of $A_k$ are within a constant factor of each other. Independence of components is not needed. Moreover, such control implies that $\rho_k$ is at least a constant, which, in turn, implies sharpness of the bounds. In other words, we provide a more general condition under which the tail of the data is "essentially high dimensional" --- instead of assuming independent components and high effective rank, only oracle control of condition number of $A_k$ is needed. In Section \ref{sec:assumption on A_k} we provide an extensive discussion of this assumption: we show that a version of a small-ball condition for the tails of the data is required and that a  stronger version of the same condition is sufficient if the data is sub-Gaussian.

The bound that we obtain for the bias term is given informally by the following expression:
\[
B \approx  \|\theta^*_\ktoinf\|_{\Sigma_\ktoinf}^2 +  \|\theta_\uptok^*\|_{\Sigma_\uptok^{-1}}^2\left(\frac{\lambda + \sum_{i > k} \lambda_i}{n}\right)^2.
\]
One can see how it aligns with the 
intuition of ``essentially low-dimensional'' and ``essentially high-dimensional'' parts: one cannot estimate the signal in the high dimensional part, so almost all of its energy $\|\theta^*_\ktoinf\|_{\Sigma_\ktoinf}^2$ goes into the error. When it comes to the low-dimensional part, the high-dimensional part acts as a ridge regularizer for it, so the bias in the first $k$ components is the same as that of ridge regression with regularization coefficient $\lambda + \sum_{i > k} \lambda_i$ (i.e., the full regularization is equal to the explicitly imposed part $\lambda$ plus "implicit regularization", which is equal to the energy of the tail.)

Our extension of the results to the ridge regression scenario allows us to answer the following question: can it happen that the "essentially high dimensional part" is too high dimensional, meaning that it provides too much regularization and negative $\lambda$ is needed to compensate for that? In Section \ref{sec::negative regularization main body}, we show that this indeed can happen and that the following is sufficient for it to be true: the noise and the energy of the signal in the tail (components $\ktoinf$) are small compared to the signal in the spiked\footnote{Here we use the word "spiked" as in the "spiked covariance models", which usually assume that the eigenvalues of $\Sigma_\ktoinf$ are all equal and of smaller order than eigenvalues of $\Sigma_\uptok$. One way to interpret our results is that only spiked-covariance-like models can exhibit benign overfitting, and we derive general conditions for a model to be spiked-covariance-like.}  part (components $\uptok$),  but the effective rank of the tail abruptly becomes much larger than $n$.

\subsection{Additional notation}
\label{sec::additional notation}
For any symmetric matrix $M \in \R^{n \times n}$ and any $i \in \{1, 2, \dots, n\}$ we write $\mu_i(M)$ for the $i$-th largest eigenvalue of $M$. For example, $\mu_1(M)$ is its largest eigenvalue, and $\mu_n(M)$ is the smallest. We write $M[i,j]$ for the element of $M$ standing in the $i$-th row and $j$-th column.

Throughout the paper the following objects will be needed: for any $i$ denote $z_i$ to be the $i$-th column of $Z$. Then define 
\[
A_{-i} := X_{0:i-1}X_{0:i-1}^\top + X_{i:\infty}X_{i:\infty}^\top + \lambda I_n = \lambda I_n + \sum_{j \neq i}\lambda_j z_j z_j^\top,
\]
an analogue of the matrix $A$, but we throw away the $i$-th component of the data vectors. Denote also
\[
\rho_k(0) = \frac{1}{n\lambda_{k+1}}\sum_{i > k}\lambda_i,
\]
the ratio of the effective rank of the tail to the number of data points without taking regularization $\lambda$ into account.  

Finally,  denote our data points to be $\{x^i\}_{i=1}^n$, i.e., $X^\top = [x^1, \dots, x^n] \in \R^{p \times n}$.

For the readers convenience, we compile all the notation in Appendix~\ref{sec::notations appendix}.
\section{Main results}
\label{sec::main result main body}

As we explained in the previous section, the central objects in our proof are $A_k$ and $\rho_k$.  In principle, any control of the spectrum of $A_k$ leads to some upper bound on $B$ and $V$ (see our Theorem \ref{th::main upper main body}), the question is when that bound is tight. The intuitive answer is the following: the bound is tight when the condition number of $A_k$ is a constant and $k$ is chosen correctly, meaning that either $\rho_k$ is a constant or  $k$ is the smallest number such that $\rho_{k}$ is larger than a constant (i.e., $k = k^*$).\footnote{Note that there may be several values of $k$ that satisfy these conditions. Applying our upper bound for any of those $k$ will yield the same result up to a constant factor.} Our arguments, however, only support this intuition when the following technical assumption holds for some constant $\gamma < 1$:
\begin{enumerate}[leftmargin=3cm]
\item[\namedlabel{as::gamma}{{\it NoncritReg}}($k, \gamma$)] Assume that $\lambda > -\gamma \sum_{i > k} \lambda_i$.
\end{enumerate}
The reason why this assumption is needed is that as $\lambda$ approaches $-\sum_{i > k}\lambda_i$, $\E A_k$ approaches zero. It still can be possible to bound the eigenvalues of $A_k$ with high probability in such regime, but their magnitude will be smaller, and some error terms that were dominated before become significant. We do investigate such a regime in Section~\ref{sec::negative regularization main body}, where we show that negative regularization may give better rates than any value of non-negative regularization, but we only provide an upper bound there. 
For all the results we discuss in this section, we make  Assumption  \ref{as::gamma}$(k, \gamma)$.

The focus of our work was  to obtain the tight upper bound on the excess risk under minimal assumptions. Such minimal assumption turns out to be 
\begin{enumerate}[leftmargin=3cm]
\item[\namedlabel{as::condition number of A_k}{{\it CondNum}}($k, \delta, L$)] Assume that with probability at least $1-\delta$ the matrix $A_k$ is positive-definite (PD)  with condition number at most $L$.
\end{enumerate}  We provide a thorough discussion of this assumption in Section \ref{sec:assumption on A_k}, for example we derive sufficient and almost matching necessary conditions for it to hold when the distribution is sub-Gaussian. The reason why we don't just assume those sufficient conditions is that we believe that sub-Gaussianity is not essential for our results to hold, as we discuss in Section \ref{sec::role of sub-Gaussianity}. Moreover, the matrix $A_k$ is the central object in our argument, and making an assumption on its condition number explicitly makes  presentation easier.

A careful reader will notice that we have just stated that another condition is needed for the bound to be tight: $k$ should be chosen in the right way. This, however, can be achieved by shifting $k$ to $k^*$ if necessary: indeed, assumptions  \ref{as::gamma}$(k, \gamma)$ and \ref{as::condition number of A_k}$(k, \delta, L)$ imply a constant lower bound on $\rho_k$ (see Corollary \ref{cor::upper bound given L main body}). That means that either $\rho_k$ is a constant, or it is more than a constant, i.e., $k > k^*$. In the latter case one can shift from $k$ to $k^*$ meaning that Assumption \ref{as::condition number of A_k}$(k^*, \delta', L')$ also holds with modified constants $\delta', L'$ (see Lemma \ref{lm::k star must give eigvals} for the exact statement). Now applying the upper bound (Corollary \ref{cor::upper bound given L main body}) with $k = k^*$ gives tight result, as given by the following
\begin{theorem}
\label{th::showcase upper}
Fix any constants $b > 0,$ $ \gamma \in [0, 1),$ $L > 0.$ Denote \[
k^* = \min\{\kappa: \rho_\kappa > b\}.
\]
There exists a constant $c$ which only depends on $\sigma_x$, $b$, $\gamma$, $L$ such that  the following holds: suppose \ref{as::gamma}$(\bar{k}, \gamma)$ and \ref{as::condition number of A_k}$(\bar{k}, \delta, L)$ are satisfied for some $\bar{k} < n/c$ and $\delta < 1-ce^{-n/c}$. Take $k = \min(\bar{k}, k^*)$.  Then  with probability at least $1 - ce^{-n/c} - \delta$
\begin{align}
B/c \leq& \|\theta^*_\ktoinf\|_{\Sigma_\ktoinf}^2 +  \|\theta_\uptok^*\|_{\Sigma_\uptok^{-1}}^2\left(\frac{\lambda + \sum_{i > k} \lambda_i}{n}\right)^2\label{eq::bias bound main},\\
V/c \leq& \frac{k}{n} + \frac{n\sum_{i > k} \lambda_i^2}{\left(\lambda + \sum_{i > k} \lambda_i\right)^2}.\label{eq::variance bound main}  
\end{align}
Moreover $\rho_k \geq c^{-1}$,  \ref{as::gamma}$(k, \gamma)$ holds, and there exist $L', c'$ that only depend on $\sigma_x, b, \gamma, L$ s.t. \ref{as::condition number of A_k}$(k, \delta + c'e^{-n/c'}, L')$ holds.\footnote{That is, the assumptions still hold if we substitute $\bar{k}$ by $k$, but  with different $L, \delta$. Further we will see that satisfaction of these assumptions implies tightness of the bounds for the chosen $k$.}
\end{theorem}
\begin{proof}
In this proof let's call any quantities that only depend on $\sigma_x$, $\gamma$, $b$ and $L$ "constants". 
 First of all, if $\bar{k} \leq k^*$ then $k = \bar{k}$. Since we are given that \ref{as::gamma}$(\bar{k}, \gamma)$  and \ref{as::condition number of A_k}$(\bar{k}, \delta, L)$ are satisfied, we immediately get that \ref{as::gamma}$(k, \gamma)$ and \ref{as::condition number of A_k}$(k, \delta + c'e^{-n/c'}, L')$ are satisfied with $L' = L$ and any $c' > 0$. However, if $\bar{k} > k^*$ then $k = k^*$ and by Lemma \ref{lm::k star must give eigvals} \ref{as::gamma}$(k, \gamma)$ and \ref{as::condition number of A_k}$(k, \delta + c'e^{-n/c'}, L')$  are still satisfied for some constants $c', L'$. Note that the larger the constants, the looser the assumptions, so we can take our final choice of $c', L'$ to be the maximum over two cases. 

Now that we know that \ref{as::gamma}$(k, \gamma)$ and \ref{as::condition number of A_k}$(k, \delta + c'e^{-n/c'}, L')$ are satisfied, by Corollary \ref{cor::upper bound given L main body}, there is a constant $c_1$ such that $\rho_k > 1/c_1$ and with probability at least $1 - c_1e^{-n/c_1} - c'e^{-n/c'} - \delta$ 
\begin{align*}
B/c_1 \leq& \|\theta^*_\ktoinf\|_{\Sigma_\ktoinf}^2 +  \|\theta_\uptok^*\|_{\Sigma_\uptok^{-1}}^2\left(\frac{\lambda + \sum_{i > k} \lambda_i}{n}\right)^2,\\
V/c_1 \leq& \frac{k}{n} + \frac{n\sum_{i > k} \lambda_i^2}{\left(\lambda + \sum_{i > k} \lambda_i\right)^2}.  \\
\end{align*}
Taking $c \geq c_1 + c'$ gives the first part.
\end{proof}

Algebraically, under Assumption  \ref{as::condition number of A_k}$(k, \delta, L)$  all eigenvalues of $A_k^{-1}$ are within a constant factor of each other, so one can pull its operator norm from the expressions and obtain an upper bound without losing tightness. This strategy, however, doesn't produce lower bounds, so we derive them in a different way: we decompose bias and variance into sums with respect to individual coordinates of the predictor, and bound each term in each sum from below.  Because of that, we impose different assumptions, namely 
\begin{enumerate}[leftmargin=3cm]
\item[\namedlabel{as::independent components}{{\it IndepCoord}}] Assume that all elements of matrix $X$ are independent (i.e., data vectors have independent coordinates).
\end{enumerate}
for the variance term, and 
\begin{enumerate}[leftmargin=3cm]
\item[\namedlabel{as::exchangeable components}{{\it ExchCoord}}] Assume that the sequence of coordinates of $\Sigma^{-1/2}x$ is exchangeable (any deterministic permutation of the coordinates of whitened data vectors doesn't change their distribution).
\item[\namedlabel{as::prior on theta^*}{{\it PriorSigns}}($\bar{\theta}$)] Assume that $\theta^*$ is sampled from a prior distribution in the following way: one starts with vector $\bar{\theta}$ and flips signs of all its coordinates with probability $0.5$ independently.
\end{enumerate}
for the bias term. Because of this mismatch in assumptions, our lower bounds don't show that our upper bound is always tight. What they show is that one needs some specific knowledge about the distribution to obtain better bounds. We provide a more detailed discussion of the relations between those assumptions in Section \ref{sec::lower bounds main body}. The lower bounds themselves are given by the following
\begin{theorem}
\label{th::showcase lower}
Fix any constants $b > a > 0$,   $ \gamma \in [0, 1),$ $L > 0.$ Denote \[
k^* = \min\{\kappa: \rho_\kappa > b\}.
\]
There exists a constant $c$ which only depends on $\sigma_x$, $a$, $b$, $\gamma$, $L$ such that all the following hold:
\begin{enumerate}
\item For any $k \in \{0, 1, \dots, k^*\}$ under assumptions \ref{as::independent components},  \ref{as::gamma}$(k, \gamma)$, if $\rho_k > a$ then with probability at least $1 - 2\delta - ce^{-c/n}$
\[
V \geq\frac{1}{c}\left( \frac{k}{n} + \frac{n\sum_{i > k} \lambda_i^2}{\left(\lambda + \sum_{i > k} \lambda_i\right)^2}\right).
\]

\item  For any $k \in \{1, 2, \dots, k^*\}$ under assumptions \ref{as::gamma}$(k, \gamma)$, \ref{as::condition number of A_k}$(k, \delta, L)$, \ref{as::prior on theta^*}$(\bar{\theta})$ and \ref{as::exchangeable components}, if $\rho_k > a$ then  with probability at least $1 - 2\delta - ce^{-c/n}$
\[
\E_{\theta^*}B \geq\frac{1}{c}\left( \|\bar{\theta}_\ktoinf\|_{\Sigma_\ktoinf}^2 +  \|\bar{\theta}_\uptok\|_{\Sigma_\uptok^{-1}}^2\left(\frac{\lambda + \sum_{i > k} \lambda_i}{n}\right)^2\right),
\]
where $\E_{\theta^*}$ denotes expectation over a random draw of $\theta^*$ from the distribution described in assumption \ref{as::prior on theta^*}$(\bar{\theta})$.\footnote{ Note that under this distribution $\|\bar{\theta}_\ktoinf\|_{\Sigma_\ktoinf}  = \|{\theta^*}_\ktoinf\|_{\Sigma_\ktoinf}$ and $\|\bar{\theta}_\uptok\|_{\Sigma_\uptok^{-1}} = \|{\theta^*}_\uptok\|_{\Sigma_\uptok^{-1}}$ almost surely.}
\end{enumerate}
\end{theorem}
\begin{proof}
 Lemma \ref{lm::var lower main body} gives a lower bound for $V$, and Lemmas~\ref{lm::bias lower main body} and \ref{lm::exchangeable coordinates plus condition number of A_k} give the lower bound for B. Those lower bounds have the desired probability, but different algebraic form. To bring them to the same form as the upper bounds one needs the right $k$ to be chosen. We assumed that $\rho_k > a$. Moreover, since $k \leq k^*$ by definition of $k^*$ we either have $\rho_k \leq b$ or $k = k^*$. In both of those cases Theorem \ref{th:upper same as lower main body} guarantees that  these lower bounds are the same as what we need up to multiplicative constants that only depend on $\sigma_x$, $\gamma$, $a$, $b$ and $L$.
\end{proof}

One can notice from this proof that having separate arguments for the lower bounds results in a different algebraic form of the same bound. This different form turns out to be convenient to draw explicit connections between our results and results from earlier works. We do that in Section \ref{sec::alternative form of the bound}.

\section{Effective ranks and control of the spectrum of $A_k$}
\label{sec:assumption on A_k}
The central assumption that we need to compute the excess risk is Assumption  \ref{as::condition number of A_k}$(k, \delta, L)$, which provides control over condition number of $A_k$. In this section we discuss when this assumption is known to be satisfied and what are the necessary conditions for it to happen. 
\subsection{Effect of $\lambda$ on the condition number}
Recall that $A_k = X_\ktoinf X_\ktoinf^\top + \lambda I_n$, so its spectrum is the shift by $\lambda$ of the spectrum of $X_\ktoinf X_\ktoinf^\top$, the random matrix that is equal to the Gram matrix of the projected data. There are therefore three ways of establishing a constant upper bound on the condition number of $A_k$:
\begin{enumerate}
\item Establish an upper bound $\bar{\mu}$ on $\mu_1(X_\ktoinf X^\top_\ktoinf)$ and take $\lambda > \bar{\mu}/c$ for some constant $c > 0$. In this case, the singular values of $A_k$ are all equal to $\lambda$ (and greater than $\bar{\mu}$) up to a constant multiplier.
\item Establish upper and lower bounds $\bar{\mu}$ and $\underline{\mu}$ on $\mu_1(X_\ktoinf X^\top_\ktoinf)$  and $\mu_n(X_\ktoinf X^\top_\ktoinf)$ respectively, such that $\bar{\mu}/\underline{\mu}$ is a constant. Then take $\lambda > -\underline{\mu}/c$ for some constant $c > 1$.  In this case, the singular values of $A_k$ are all equal to $\bar{\mu}$ (or $\underline{\mu}$) up to a constant multiplier.
\item  Establish upper and lower bounds $\bar{\mu}$ and $\underline{\mu}$ on $\mu_1(X_\ktoinf X^\top_\ktoinf)$  and $\mu_n(X_\ktoinf X^\top_\ktoinf)$ respectively, and take $\lambda = -\underline{\mu} + \Diamond$, where $\Diamond \geq c(\bar{\mu} - \underline{\mu})$ for a constant $c > 0$.
In this case, the singular values of $A_k$ are all equal to $\Diamond$ up to a constant multiplier.  This case can be substantially different from the previous case when the singular values of $X_\ktoinf X^\top_\ktoinf$ are very well concentrated, i.e., the gap $\bar{\mu} - \underline{\mu}$ is of smaller order than $\underline{\mu}$ itself. In this case $\Diamond$ can be a smaller order term.
\end{enumerate}

Our bounds are sharp when assumption \ref{as::gamma}($\gamma$) is satisfied for some $\gamma < 1$, i.e., in the first and the second case above. The third case is quite rare because it requires very good concentration of the spectrum of $X_\ktoinf X^\top_\ktoinf$. Moreover, in this case $\lambda$ is very close to the critical negative value under which it is impossible to even guarantee that $A_k$ is PD as it becomes negative definite in expectation.  We use this regime to investigate how negative regularization can improve excess risk by more than a constant factor in Section \ref{sec::negative regularization main body}. However, we don't expect our bounds to always be sharp in this regime. 

Therefore, we focus our attention on the first two cases. In Section \ref{sec::A_k necessary conditions informal} we discuss informally what conditions on the distribution are necessary to bound $\mu_1(X_\ktoinf X^\top_\ktoinf)$ and $\mu_n(X_\ktoinf X^\top_\ktoinf)$, and show how notions of high effective rank and norm concentration condition arise.  In Section \ref{sec::condition number via sub-Gaussianity} we combine those bounds for sub-Gaussian data with the choice of $\lambda$ to provide necessary and almost matching sufficient conditions for the condition number of $A_k$ to be constant under sub-Gaussianity. In Section \ref{sec::A_k heavy tailed} we show that sub-Gaussianity is not actually required for the condition number of $A_k$ to be controlled with high probability: Theorem \ref{th::A_k heavy tailed} states that norm concentration condition and a modified version of high effective rank condition are sufficient even if the data  only has bounded $4 + \eps$ moments.

\subsection{Informal necessary conditions}
\label{sec::A_k necessary conditions informal}
There are several easy observations that help understand what is needed for the condition number of $A_k$ to be bounded. 

\begin{enumerate}
\item 
The first observation is that $X_\ktoinf X^\top_\ktoinf\succeq \lambda_{k+1} z_{k+1} z_{k+1}^\top$, where $z_{k+1}$ is the first column of $Z_\ktoinf$ ---a vector with $n$ i.i.d. coordinates with unit variance. By the law of large numbers, $\|z_{k+1}\|^2 \approx n$, meaning that $\|\lambda_{k+1} z_{k+1} z_{k+1}^\top\|\approx\lambda_{k+1}n$. Therefore, $\bar{\mu} \gtrsim \lambda_{k+1}n$.

\item The second observation is that  the diagonal elements of $X_\ktoinf X^\top_\ktoinf$ are squared norms of the tails of data vectors. Recall that we denoted the data points to be $\{x^i\}_{i=1}^n$. We can write
\[
(X_\ktoinf X^\top_\ktoinf)[i,i] = \|x^i_\ktoinf\|^2\text{ --- i.i.d. r.v's.}
\]
Once again, by the law of large numbers, $
\tr(X_\ktoinf X^\top_\ktoinf) \approx n\sum_{i > k} \lambda_i,$ which implies that $\bar{\mu} \gtrsim \sum_{i > k} \lambda_i \gtrsim \underline{\mu}$. 
Combining it with the first observation shows that $\bar{\mu}$ and $\underline{\mu}$ can only be within a constant multiplier of each other when $\sum_{i > k} \lambda_i \geq c\lambda_{k+1}n$ for some constant $c$. This is exactly the high effective rank condition $\rho_k > c$ for $\lambda = 0$.

\item The third observation is that the diagonal elements of a  PD matrix themselves provide bounds on the singular values:
\[
\mu_n(X_\ktoinf X^\top_\ktoinf)\leq \min_{i \in [n]} (X_\ktoinf X^\top_\ktoinf)[i,i] \leq \max_{i \in [n]} (X_\ktoinf X^\top_\ktoinf)[i,i] \leq \mu_1(X_\ktoinf X^\top_\ktoinf).
\]

Therefore, to control condition number of $X_\ktoinf X^\top_\ktoinf$ by a constant $L$ with probability $1-\delta$, it is necessary to guarantee that
\[
\max_i\|x^i_\ktoinf\|^2 \leq L \min_j \|x^j_\ktoinf\|^2,
\]
i.e., $n$ independent random draws of the random variable $\|x_\ktoinf\|^2$  should all lie within a constant factor of some value, meaning that the norm of the tail of a data vector should be within a constant factor of a fixed value  with probability $(1-\delta)^{1/n}$.
\end{enumerate}

\subsection{Controlling condition number under sub-Gaussianity}
\label{sec::condition number via sub-Gaussianity}
Sub-Gaussianity of the data implies an upper bound on $\mu_1(A_k)$, but doesn't help with $\mu_n(A_k)$. To see this one can consider a well-known construction: take a sub-Gaussian distribution and construct another distribution in the following way: to sample from this new distribution take a vector from the old distribution and multiply it by $\sqrt{2}$ with probability $1/2$ and by zero otherwise. The new distribution is still sub-Gaussian with the same covariance, but the Gram matrix of $n$ i.i.d. samples from it is degenerate with probability at least $1-2^{-n}$. Therefore, an additional assumption is needed to lower bound $\mu_n(A_k)$. As we already mentioned in Section \ref{sec::A_k necessary conditions informal}, we need norm concentration. Since sub-Gaussianity allows to bound the norm from above, it reduces to a version of the small-ball condition: $\|x_\ktoinf\|$ should be lower-bounded with high probability. The formal result is given by the following
\begin{restatable}[Controlling $\mu_1(A_k)/\mu_n(A_k)$ under sub-Gaussianity]{lemma}{controllingAk}\label{lm::controlling A_k iff small ball and rank}
For any $\gamma \in [0, 1)$ and $\sigma_x > 0$ there exists $c> 0$ that only depends on $\sigma_x$ and $\gamma$ such that under Assumption  \ref{as::gamma}$(k, \gamma)$  the following holds: for any $L \geq 1$
\begin{itemize}
\item If $\rho_k \geq L^2$ and with probability at least $(1 - \delta)^{1/n}$
\[
\lambda + \|x_\ktoinf\|^2 \geq \frac{c}{L}\left(\lambda + \E\|x_\ktoinf\|^2\right), 
\]
then with probability at least $1 - \delta - ce^{-n/c}$ 
\[
\mu_n(A_k) \geq L^{-1}\mu_1(A_k). 
\]
\item Suppose that it is known that with probability at least $ce^{-n/c}$ $\mu_n(A_k) \geq L^{-1}\mu_1(A_k)$. Then $\rho_k \geq \frac{1}{cL}$ and with probability at least $\left(1 - ce^{-n/c}\right)^{1/n}$
\[
\lambda + \|x_\ktoinf\|^2 \geq \frac{1}{cL}\left(\lambda + \E\|x_\ktoinf\|^2\right).
\]
\end{itemize}
\end{restatable}
The proof is given in Appendix \ref{sec::singular values}. One can see that both the necessary and the sufficient conditions are that $\rho_k$ is lower bounded by a constant and a version of small-ball condition that says that the regularized squared norm of the data exceeds a constant fraction of its expectation with probability $(1-\delta)^{1/n}$. There is, however, a gap in those constants.

\subsection{Heavy-tailed case}
\label{sec::A_k heavy tailed}
The following is a direct corollary of Theorem 2.1 from \cite{heavy_tailed_columns_guedon}
\begin{theorem}
\label{th::A_k heavy tailed}
Suppose that the distribution of the tail satisfies the following two assumptions:
\begin{enumerate}
\item {\bf Norm concentration:} For some $\delta \in (0, 1/n)$, $L > 1$ and $M > 0$  
\[
\Pbb(L^{-1} \leq \|x_\ktoinf\|/M \leq L) \geq 1-\delta.
\]
\item {\bf Heavy-tailed effective rank:}
for some  $h > 4$ denote  $r_{h,k} > 0$ to be the maximum number such that for any $a \in \Sc^{p-k-1}$ and $t > 0$
\[
\Pbb\left(\frac{\sqrt{r_{h,k}}\left|a^\top x_\ktoinf\right|}{M} > t\right) \leq t^{-h}.
\]
\end{enumerate}
There exists a constant $c$ that only depends on $h$ such that
with probability at least $1 - cn^{1-h/4} - n\delta$
\begin{align*}
\mu_1(X_\ktoinf X_\ktoinf^\top) \leq& M^2 \left(L^2 + cL^2\left(n^{1-h/4} +  \sqrt{\frac{n}{r_{h,k}L^2}} + \frac{n}{r_{h,k}L^2}\right)\right),\\
\mu_n(X_\ktoinf X_\ktoinf^\top) \geq& M^2 \left(L^{-2} - cL^2\left(n^{1-h/4} +  \sqrt{\frac{n}{r_{h,k}L^2}} + \frac{n}{r_{h,k}L^2}\right)\right).
\end{align*}
\end{theorem}
\begin{proof}
First, note that by union bound with probability at least $1 - n\delta$ all the diagonal elements of the matrix $X_\ktoinf X_\ktoinf^\top$ belong to the segment $[L^{-2}M^2, L^2M^2]$. 
Next, take the bound on $B_k$ from the Case 1 of  Theorem 2.1 from \cite{heavy_tailed_columns_guedon} with the following choice of their parameters: $k = N$, $\tau = 1$, $\lambda = p$, $\sigma = 1 + p/4$, $t = \sqrt{n}$.
Use that bound for vectors $\sqrt{r_{h,k}}x^i_\ktoinf/M$. Note that that  $B_k$ is exactly the  operator norm of the off-diagonal part of $r_{h,k}X_\ktoinf X_\ktoinf^\top/M^2$.
\end{proof}


The quantity $r_{h,k}$ that we introduced in  Theorem \ref{th::A_k heavy tailed} can be interpreted as a notion of effective rank for heavy tailed distributions. Indeed, one can write
\[
\sqrt{r_{h, k}} = \frac{M}{\inf\left\{\tau: \forall a \in \Sc^{p-k-1} \forall t>0\; \Pbb\left(\left|a^\top x_\ktoinf\right|/\tau > t\right) \leq t^{-h}\right\}}.
\]
--- the ratio of the typical norm of the random vector to the scale of the worst case deviations of its one-dimensional projection. This is completely analogous to our usual definition of the effective rank: $r_k = \lambda_{k+1}^{-1}\sum_{i > k}\lambda_i$. Indeed, in sub-Gaussian case $\sqrt{\sum_{i > k}\lambda_i}$ is the typical value of the norm of the vector $x_\ktoinf$, and $\sqrt{\lambda_{k+1}}$ is up to constant the largest sub-Gaussian norm of its one-dimensional projection. We see that the conditions under which the eigenvalues of $X_\ktoinf X_\ktoinf^\top$ are within a constant factor of each other with high probability remain the same even in the heavy-tailed case: the norm of $\|x_\ktoinf\|$  concentrates within a constant factor of a fixed quantity, and the heavy-tailed effective rank $r_{h, k}$ should be large compared to the number $n$ of data points.

\section{Structure of the proof and role of sub-Gaussianity}
\label{sec::proof structure main body}
\subsection{Upper bound}

The core of our argument is Theorem \ref{th::main upper main body} given below. There are two important things to note about it: first, it only requires sub-Gaussianity and matrix $A_k$ being positive semidefinite (which always holds with probability $1$ for non-negative $\lambda$). Second, its proof decomposes very clearly into two parts: an algebraic part, which only requires $A_k$ being PD and holds with probability $1$ conditionally on  this event, and a probabilistic part, where standard concentration results are directly plugged into the algebraic bounds. Because of this decomposition, it is straightforward to track how the sub-Gaussianity is used and how it can be relaxed. We provide the sketch of the proof to show these details.
\begin{restatable}{theorem}{mainthm}\label{th::main upper main body}
There exists a (large) constant $c$, which only depends on $\sigma_x$, s.t.\ for any $k < n/c$  with probability at least $1 - ce^{-n/c}$, if the matrix $A_k$ is PD, then 

\begin{align*}
B/c \leq& \|\theta^*_\ktoinf\|_{\Sigma_\ktoinf}^2\left(1 +  \frac{\mu_1(A_k^{-1})^2}{\mu_n(A_k^{-1})^2} + n\lambda_{k+1}\mu_1(A_k^{-1})\left(1 +\max(0, -\lambda)\mu_1(A_k^{-1})\right) \right)\\
+& \|\theta_\uptok^*\|_{\Sigma_\uptok^{-1}}^2\left(\frac{1}{n^2\mu_n(A_k^{-1})^2} +  \frac{\lambda_{k+1}}{n}\frac{\mu_1(A_k^{-1})}{\mu_n(A_k^{-1})^2}\left(1 +\max(0, -\lambda)\mu_1(A_k^{-1})\right) \right), \\
V/c \leq& \frac{\mu_1(A_k^{-1})^2}{\mu_n(A_k^{-1})^2}\frac{k}{n} + n \mu_1(A_k^{-1})^2 \sum_{i > k} \lambda_i^2.  \\
\end{align*}
\end{restatable}
\begin{proofsketch}
The full 
proof of Theorem \ref{th::main upper main body} can be found in Section~\ref{subsec::upper} of the appendix. The following is a sketch of its derivation.

Recall the following notation: for any $y$
\[
\hat\theta(y) = X^\top(\lambda I_n + XX^\top)^{-1}y.
\]

As explained in Section \ref{sec::benign overfitting story}, \cite{benign_overfitting} introduced the notion of $k^*$ for which the  behaviour of the variance term in the first $k^*$ coordinates is qualitatively different than in the rest of the coordinates. Their argument, however, relies crucially on independence of the components of the data. The main idea that allowed us to get rid of that assumption and to obtain the tight bound for the bias term  was to separate the first $k$ coordinates from the very beginning and to use some sort of uniform convergence argument in that low-dimensional subspace. 

The crucial tool that allowed us to realise this idea turned out to be the following algebraic identity that we prove in Section \ref{sec::identity} of the appendix:
\[
\hat\theta( y)_\uptok + X_\uptok^\top A_k^{-1} X_\uptok\hat\theta( y)_\uptok = X_\uptok^\top A_k^{-1}y.
\]
This identity allows convenient access to the error in the first $k$ coordinates (the spiked part).

The argument decomposes clearly into two parts: algebraic and probabilistic. The algebraic part is to decompose the excess risk  (up to a constant multiplier) into four terms and show that the following inequalities hold  on the event that the matrix $A_k$ is PD: \\
\noindent (1) Bias error in the spiked part: 
\begin{align*}
\|\hat\theta( X\theta^*)_\uptok - \theta^*_\uptok\|_{\Sigma_\uptok}
\leq& \frac{\mu_1(A_k^{-1})}{\mu_n(A_k^{-1})}\frac{\mu_1\left(\Sigma^{-1/2}_\uptok X_\uptok^\top X_\uptok\Sigma_\uptok^{-1/2} \right)^{1/2}}{\mu_k\left(\Sigma^{-1/2}_\uptok X_\uptok^\top X_\uptok\Sigma_\uptok^{-1/2} \right)}\|X_\ktoinf \theta^*_\ktoinf\| \\*
&\qquad {} + \frac{\|\theta_\uptok^*\|_{\Sigma_\uptok^{-1}}}{\mu_n(A_k^{-1})\mu_k\left(\Sigma^{-1/2}_\uptok X_\uptok^\top X_\uptok\Sigma_\uptok^{-1/2} \right)}.
\end{align*}
\noindent (2) Variance error in the spiked part:  
\begin{align*}
\E_{\eps}\|\hat\theta(\eps)_\uptok\|_{\Sigma_\uptok}^2
\leq \frac{\mu_1( A_k^{-1})^2\|X_\uptok\Sigma_\uptok^{-1/2} \|^2}{\mu_n(A_k^{-1})^2 \mu_k\left(\Sigma_\uptok^{-1/2}X_\uptok^\top X_\uptok \Sigma_\uptok^{-1/2}\right)^2}.
\end{align*}
\noindent (3) Variance error in the tail:
\begin{align*}
\E_\eps\|\hat\theta( \eps)_\ktoinf - \theta^*_\ktoinf\|_{\Sigma_\ktoinf}^2\leq  \mu_1(A_k^{-1})^2 \tr(X_\ktoinf\Sigma_\ktoinf X_\ktoinf^\top).
\end{align*}
\noindent (4) Bias  error in the tail:
\begin{align*}
&\frac13\|\hat\theta( X\theta^*)_\ktoinf - \theta^*_\ktoinf\|_{\Sigma_\ktoinf}^2\\
\leq&\|\theta^*_\ktoinf\|_{\Sigma_\ktoinf}^2+\lambda_{k+1} \bigl(1+ \max(0, -\lambda)\mu_1(A_k^{-1})\bigr) \mu_1(A_k^{-1})\|X_\ktoinf\theta^*_\ktoinf\|^2\\
+&\lambda_{k+1} \bigl(1+ \max(0, -\lambda)\mu_1(A_k^{-1})\bigr)\frac{\mu_1(A_k^{-1})}{\mu_n(A_k^{-1})^2} \frac{\mu_1(\Sigma_{\uptok}^{-1/2}X_\uptok^\top X_\uptok\Sigma_{\uptok}^{-1/2})}{\mu_k(\Sigma_{\uptok}^{-1/2}X_\uptok^\top X_\uptok\Sigma_{\uptok}^{-1/2})^2}\|\Sigma_{\uptok}^{-1/2}\theta^*_\uptok\|^2.
\end{align*}

The probabilistic part of the argument is to control the quantities that arise in the algebraic bound with high probability. Namely, we plug in 
\begin{itemize}
\item Concentration of $k$-dimensional sample covariance  with $n$ samples: w.h.p.
\[
\mu_k\left(\frac1n\Sigma_\uptok^{-1/2}X_\uptok^\top X_\uptok \Sigma_\uptok^{-1/2}\right)   \approx_{\sigma_x}
\mu_1\left(\frac1n\Sigma_\uptok^{-1/2}X_\uptok^\top X_\uptok \Sigma_\uptok^{-1/2}\right) \approx_{\sigma_x} 1.
\]
\item Concentration of norm of vectors with i.i.d. components: w.h.p.
\begin{equation*}
\frac1n\|X_\uptok\Sigma_\uptok^{-1/2}\|^2 \lesssim_{\sigma_x} k,\quad 
\frac1n\|X_\ktoinf \Sigma_\ktoinf^{1/2}\|^2\lesssim_{\sigma_x} \sum_{i > k} \lambda_i^2,\quad
\frac1n\|X_\ktoinf\theta^*_\ktoinf\|^2\lesssim_{\sigma_x} \|\theta^*_\ktoinf\|_{\Sigma_\ktoinf}^2.
\end{equation*}
\end{itemize}
After plugging in the probabilistic bounds, the final result is obtained by a straightforward computation. 
\end{proofsketch}

Note that the only probabilistic statements that are used in this proof are concentration of sample covariance in dimension $k$ and concentration of the sum of $n$ i.i.d. random variables. The same concentration results hold with weaker assumptions, but with larger probability. For example, under rather weak moment assumptions only a linear in dimension number of samples is needed for the sample covariance matrix to concentrate within a constant factor of the population covariance, see \cite{10.1093/imrn/rnx067} and references therein. It is also interesting to point out that the "uniform convergence" result that we mentioned in the beginning of this sketch is nothing but the convergence of the empirical covariance matrix $\Sigma_\uptok^{-1/2}X_\uptok^\top X_\uptok \Sigma_\uptok^{-1/2}/n$ to its expectation $I_k$, which is exactly the uniform convergence result that gives the bound in the "essentially low-dimensional" regime from Section~\ref{sec::regimes of linear regression}.

Despite the fact that the bounds of Theorem \ref{th::main upper main body} apply under very general assumptions, we don't expect them to be tight if the condition number of $A_k$ is not bounded by a constant. When some oracle control of the condition number of $A_k$ is provided, the bound becomes the following.
\begin{restatable}{corollary}{upperboundgivenL}\label{cor::upper bound given L main body}
Fix any constants $\gamma \in [0, 1)$ and $L >0$. There exists a constant $c$ that only depends on $\sigma_x$, $\gamma$, $L$ s.t. for any $k < n/c$ and $\delta < 1 - ce^{-n/c}$  under assumptions \ref{as::gamma}$(k, \gamma)$ and \ref{as::condition number of A_k}$(k, \delta, L)$,  it holds that $\rho_k > c^{-1}$, and with probability at least $1 - \delta - ce^{-n/c}$,
\begin{align*}
B/c \leq& \|\theta^*_\ktoinf\|_{\Sigma_\ktoinf}^2 +  \|\theta_\uptok^*\|_{\Sigma_\uptok^{-1}}^2\left(\frac{\lambda + \sum_{i > k} \lambda_i}{n}\right)^2,\\
V/c \leq& \frac{k}{n} + \frac{n\sum_{i > k} \lambda_i^2}{\left(\lambda + \sum_{i > k} \lambda_i\right)^2}.  \\
\end{align*}
\end{restatable}
\begin{proofsketch}
Assumptions \ref{as::gamma}$(k, \gamma)$ and \ref{as::condition number of A_k}$(k, \delta, L)$  imply that all the eigenvalues of $A_k$ are equal to $\lambda + \sum_{i > k} \lambda_i$ up to a multiplicative constant that depends on $L, \gamma, \sigma_x$. Plugging it into Theorem \ref{th::main upper main body} gives the result. The full proof is given in Appendix \ref{subsec::upper}.
\end{proofsketch}
The sub-Gaussianity is used in Corollary \ref{cor::upper bound given L main body} to ensure that $\tr(A_k)$ concentrates around $n\left(\lambda + \sum_{i > k} \lambda_i\right)$. Since the diagonal elements of $A_k$ are i.i.d. random variables, the same concentration would also hold under weaker assumptions with lower but still high probability.

It is also worth mentioning that the story about "essentially high-dimensional" and "essentially low-dimensional" parts is not just an interpretation of the final result: the whole proof strategy is in accordance with it, as we explicitly separate the two parts and bound errors in them separately.

\subsection{Lower bounds}
\label{sec::lower bounds main body}
Our lower bounds have a different form from the upper bounds. We show separately that they match if the condition on effective rank is satisfied. One benefit of this approach is that the lower bounds provide a different form of the same result, which allows for different analysis. We employ it in Section \ref{sec::alternative form of the bound}.

The lower bound for the variance term is given by the following lemma, whose proof is given in Appendix \ref{sec::variance lower appendix}:
\begin{restatable}[Lower bound for the variance term]{lemma}{varlower}\label{lm::var lower main body}
Fix any constant $\gamma \in [0, 1)$. There exists a constant $c$ that only depends on $\sigma_x$ and $\gamma$ s.t. for any $k < n/c$ under  assumptions \ref{as::gamma}$(k, \gamma)$ and \ref{as::independent components}  w.p. at least $1 - ce^{-n/c}$
\[
V \geq \frac{1}{cn}\sum_{i = 1}\min\left\{1, \frac{\lambda_i^2}{ \lambda_{k+1}^2(\rho_k + 1)^2}\right\}.
\]
\end{restatable}
One can see that the assumptions under which the lower bound is proved are different from the assumptions required for the upper bound: we require independent components here. On the one hand, it means that there could be a gap between upper and lower bounds in some particular cases where one can control the condition number of $A_k$ without independence of components. On the other hand, it means that even such strong additional assumption as independence of components does not allow the upper bounds to be improved, which suggests that those specific cases for which the bound is not tight are rare and require even stronger additional assumptions.

The most general lower bound for the bias term that we prove requires the following assumption
\begin{enumerate}[leftmargin=4cm]
\item[\namedlabel{as::lowest eigenvalue of A_-k}{{\it StableLowerEig}}($k, \delta, L$)] Assume that for any $j \in \{1, 2, \dots, p\}$ with probability\footnote{Note that the condition on probability is separate for every $j$, i.e., we don't assume that events hold simultaneously for all $j$.} at least $1-\delta$
\[
\mu_n(A_{-j}) \geq \mu_n(\E A_{k})/L = \left(\sum_{i > k}\lambda_i + \lambda\right)/L,
\]
and that $\lambda > - \sum_{i > k}\lambda_i$.
\end{enumerate}
Then the bound is given by the following lemma, whose proof is given in Appendix \ref{sec::bias lower appendix}
\begin{restatable}[Lower bound for the bias term]{lemma}{biaslower}\label{lm::bias lower main body}
Fix any constant $L > 0$. There exists $c$ that only depends on $\sigma_x$ and $L$ s.t. for any $k\in \{1, 2, \dots, p\}$
under assumptions \ref{as::prior on theta^*}$(\bar{\theta})$ and \ref{as::lowest eigenvalue of A_-k}$(k, \delta, L)$  w.p. at least $1 - 2\delta - ce^{-n/c}$
\[
\E_{\theta^*} B \geq \frac1c\sum_i\frac{\lambda_i\bar{\theta}^2_i}{\left(1 + \frac{\lambda_i}{\lambda_{k+1}\rho_k}\right)^2},
\]
where $\E_{\theta^*}$ denotes the expectation over the random draw of $\theta^*$ from the prior distribution described in assumption \ref{as::prior on theta^*}$(\bar{\theta})$.
\end{restatable}
Assumption \ref{as::lowest eigenvalue of A_-k}$(k, \delta, L)$ is formally not comparable to Assumption \ref{as::condition number of A_k}$(k, \delta, L)$, but informally if $k \geq 1$ then \ref{as::lowest eigenvalue of A_-k}$(k, \delta, L)$ is weaker: indeed, the matrix $A_{-i}$ is obtained from the matrix $A$ by subtracting $\lambda_i z_i^\top z_i$, while the matrix $A_k$ is obtained from $A$ by subtracting $\sum_{i = 1}^k \lambda_i z_i^\top z_i$, i.e., the sum of $k$ "largest" of the terms $  \lambda_i z_i^\top z_i$. Therefore, the matrix $A_{-i}$ is "larger" than $A_k$, and controlling its lowest singular value should be easier. The following lemma, whose proof is given in Appendix \ref{sec::bias lower appendix}, formalizes this argument under Assumption \ref{as::exchangeable components}:
\begin{restatable}{lemma}{ExchCoordPlusCondNum}\label{lm::exchangeable coordinates plus condition number of A_k}
For any  $\gamma < 1$  there exists a constant $c$ that only depends on $\gamma$ and $\sigma_x$ such that if assumptions \ref{as::condition number of A_k}$(k, \delta, L)$, \ref{as::gamma}$(k, \gamma)$ and \ref{as::exchangeable components} are satisfied for some $L \geq 1$ and $k \in \{ 1, 2, \dots, p\}$, then \ref{as::lowest eigenvalue of A_-k}$(k, \delta +2e^{-n/c}, cL)$ is also satisfied.
\end{restatable}

When it comes to averaging over the prior given by the assumption  \ref{as::prior on theta^*}$(\bar{\theta})$, it just means that it is impossible to obtain a better lower bound without some specific knowledge of how signs of components of $\theta^*$ interact with the probability distribution of the data. 

\subsection{Connecting upper and lower bounds}
One slight inconvenience with our approach of imposing oracle control over the spectrum of $A_k$ via Assumption \ref{as::condition number of A_k}$(k, \delta, L)$ is the following: what if the oracle provides control for the wrong value of $k$? There can in principle be many values of $k$ for which such oracle control is possible, with not all of them giving the right point where the behaviour changes from "essentially low-dimensional" to "essentially high-dimensional". As an example, consider the isotropic setting with $p \gg n$: one can exclude any number $k$ of components such that $p-k \gg n$ and still be able to control the condition number. 

First of all, in accordance with the result of \cite{benign_overfitting}, the following theorem shows that the "right $k$" is the $k$ that is not larger than $k^*$.
\begin{restatable}[The lower bound is the same as the upper bound]{theorem}{uppersameaslower}
\label{th:upper same as lower main body}
Denote
\begin{align*}
\underline{B}&:= \sum_i\frac{\lambda_i|\theta^*_i|^2}{\left(1 + \frac{\lambda_i}{\lambda_{k+1}\rho_k}\right)^2},\\
\overline{B}&:= \|\theta^*_\ktoinf\|_{\Sigma_\ktoinf}^2 +  \|\theta_\uptok^*\|_{\Sigma_\uptok^{-1}}^2\left(\frac{\lambda + \sum_{i > k} \lambda_i}{n}\right)^2,\\
\underline{V}&:= \frac{1}{n}\sum_{i }\min\left\{1, \frac{\lambda_i^2}{ \lambda_{k+1}^2(\rho_k + 1)^2}\right\},\\
\overline{V} &:= \frac{k}{n} + \frac{n\sum_{i > k} \lambda_i^2}{\left(\lambda + \sum_{i > k} \lambda_i\right)^2}.
\end{align*}
Fix constants $a> 0$ and $b > 1/n$. There exists a constant $c > 0$ that only depends on $a, b$, s.t.\ the following holds: if either $\rho_k \in (a, b)$  or $k = \min\{\kappa: \rho_\kappa > b\}$, then 
\[
c^{-1}\leq {\underline{B}}\mathbin{/}{\overline{B}} \leq 1, \quad
 c^{-1} \leq {\underline{V}}\mathbin{/}{\overline{V}} \leq 1.
\]
\end{restatable}
\begin{proof}
The proof is a rather straightforward comparison of pairs of sums term by term. It is given in Appendix \ref{subseq::upper is lower}.  
\end{proof}

Secondly, if the data is sub-Gaussian, then oracle control for any $k < n$ results in tight bounds, but with worse constants. This happens because of the following lemma.
\begin{restatable}[$k$ can be taken to be $k^*$]{lemma}{kandkstar}\label{lm::k star must give eigvals}
Fix any constants $\gamma \in [0, 1)$, $b > 0$, $L > 0$. Denote 
\[
k^* = \min\{k: \rho_k > b\}.
\]
There exist constants $c, L'$ that only depend on $\sigma_x$, $\gamma$, $b$, $L$ s.t. the following holds: suppose assumptions \ref{as::gamma}$(k, \gamma)$ and \ref{as::condition number of A_k}$(k, \delta, L)$ hold for some $k \in [k^*, n]$. Then assumptions \ref{as::gamma}($k^*, \gamma$) and  \ref{as::condition number of A_k}$(k^*, \delta + ce^{-n/c}, L')$ hold too.
\end{restatable}
\begin{proofsketch}
Since $k \geq k^*$, $\mu_n(A_k)$ provides a lower bound for $\mu_n(A_{k^*})$. When it comes to $\mu_1(A_{k^*})$, it can be bounded with high-probability because the data is sub-Gaussian. The full proof is given in Appendix \ref{sec::singular values}.
\end{proofsketch}

\subsection{The role of sub-Gaussianity}
\label{sec::role of sub-Gaussianity}
As can be seen from the proof of Theorem \ref{th::showcase upper}, the strategy to obtain a tight bound is the following: ask the oracle to control the condition number of $A_k$, if that $k$ is too large, shift it to $k^*$, and then apply the bound from Corollary \ref{cor::upper bound given L main body}. In Section \ref{sec::A_k heavy tailed} we showed that if the norm $\|x_\ktoinf\|$ concentrates, and the effective rank $r_{h,k}$ is high enough, then the control over the condition number of $A_k$ is possible even if we have very weak moment assumptions instead of sub-Gaussianity. Moreover, as we have discussed in the proof sketches, if we didn't shift from $k$ to $k^*$,  we would only need the usual concentration results such as the law of large numbers or concentration of $k$-dimensional empirical covariance matrix with $n$ samples, which also hold under weak moment assumptions. Therefore, sub-Gaussianity is not essential to obtain the bound in the form given in Corollary \ref{cor::upper bound given L main body}, one just needs to substitute the sub-Gaussian concentration results with their heavy-tailed analogues.  However it may not necessarily give a tight result unless the oracle is guaranteed to choose the appropriate $k$ (e.g., $k = k^*$). To shift from $k$ to $k^*$ we also need an upper bound on $\|A_{k^*}\|$, which we derive from sub-Gaussianity. According to Section \ref{sec::A_k heavy tailed}, an analogous bound is still possible under weak moment assumptions, but additional work is required: to use Theorem \ref{th::A_k heavy tailed} for $k = k^*$ one would need to obtain a high-probability upper bound on $\|x_{k^*:\infty}\|$ under moment assumptions and to relate $r_k$ which we use in definition of $k^*$ to $r_{h,k}$, which is introduced in Theorem \ref{th::A_k heavy tailed}. 

\section{Alternative forms of the bounds and effect of increasing regularization}
\label{sec::alternative form of the bound}
\subsection{Alternative form of the bound and its relation to classical in-sample analysis}
\label{sec::relation to in-sample} 
Theorem \ref{th:upper same as lower main body} reveals an alternative form of the bounds: when $\rho_k$ is lower- and upper-bounded by constants or when $k = k^*$, the bounds on the bias and variance respectively become  equal to the following up to a constant multiplier:
 \begin{align}
\tilde{B} &:= \sum_{i=1}^p\lambda_i|\theta_i^*|^2 \frac{\rho_k^2\lambda_{k+1}^2}{\left(\rho_k\lambda_{k+1} + \lambda_i\right)^2},\label{eq::bias through weighted combination}\\
\tilde{V}&:=  \frac1n \sum_{i=1}^p\frac{\lambda_i^2}{\left(\rho_k\lambda_{k+1} + \lambda_i\right)^2}.\label{eq::variance through weighted combination}
\end{align}

These expressions closely resemble the classical expressions for the in-sample bias and variance of ridge regression. Indeed, a straightforward computation gives
\begin{align*}
&\frac{1}{n}\E_{\eps}\|X\hat{\theta} - X\theta^*\|^2\\
=& \frac{1}{n}\|(XX^\top (XX^\top + \lambda I_n)^{-1} - I_n)X\theta^*\|^2 + \frac{v_\eps^2}{n}\|XX^\top (XX^\top + \lambda I_n)^{-1}\|_F^2\\
=& \underbrace{\sum_{i=1}^p\hat{\lambda}_i\langle v_i, \theta^*\rangle^2 \frac{(\lambda/n)^2}{\left(\lambda/n + \hat{\lambda}_i\right)^2}}_{\text{in-sample bias}}+ v_\eps^2\underbrace{\frac1n \sum_{i=1}^p\frac{\hat{\lambda}_i^2}{\left(\lambda/n + \hat{\lambda_i}\right)^2}}_{\text{in-sample variance}},
\end{align*}
where $\{\hat{\lambda}_i\}_{i = 1}^p$ are eigenvalues of the empirical covariance $n^{-1} X^\top X$ and $\{v_i\}_{i = 1}^p$ are the corresponding eigenvectors. Recall that $\rho_k\lambda_{k+1}  = \left(\lambda + \sum_{i > k}\lambda_i\right)/n$. One can see that Equations \eqref{eq::bias through weighted combination}--\eqref{eq::variance through weighted combination} can be obtained from the classical equations for the in-sample risk by substituting the empirical eigenvalues with population eigenvalues and increasing the regularization level $\lambda$ by $\sum_{i > k}\lambda_i$ --- the energy of the tail of the data. 

Similarly, $\tilde{B}$ has an interpretation as the bias term of ridge regression with infinite data: for $\bar{\lambda} > 0$ denote $\theta^*_{\bar{\lambda}}$ to be the solution to the following "population ridge regression" problem:
\[
\theta^*_\lambda = \argmin_{\theta} \left[\E\|X\theta - y\|^2 + \bar{\lambda}\|\theta\|^2\right] = \left(\Sigma + \frac{\bar{\lambda}}{n} I_p\right)^{-1}\Sigma \theta^*.
\]
A straightforward computation gives
\[
\|\theta - \theta^*\|_{\Sigma}^2 = \sum_i \lambda_i |\theta_i^*|^2\frac{(\bar{\lambda}/n)^2}{(\lambda_i + \bar{\lambda}/n)^2},
\]
which is equal to $\tilde{B}$  when $\bar{\lambda} = n\lambda_{k+1}\rho_k = \lambda + \sum_{i > k}\lambda_i$.

\subsection{Dependence on $\lambda$}
The alternative form of the bounds presented in Section \ref{sec::relation to in-sample} provides a convenient way to investigate the dependence on $\lambda$, which is cumbersome in the initial form because increasing $\lambda$ may decrease $k^*$. This effect, however, is negligible when Equations  \eqref{eq::bias through weighted combination}--\eqref{eq::variance through weighted combination} are considered. Indeed, in Appendix \ref{sec::alternative form appendix} we show the following
\begin{restatable}{lemma}{nokinaltform}
\label{lm::change of k is negligible in alternative form}
Suppose $k < n/c$ for some $c > 1$ and $k^* < k$. Then 
\[
\lambda_{k+1}\rho_{k} \leq \lambda_{k^* + 1} \rho_{k*} \leq \lambda_{k+1}\rho_{k} /(1 - b^{-1}c^{-1}).
\]
\end{restatable}
Because of this lemma, any $k \in [k^*, n/c]$ gives the same result (up to a constant factor) in Equations  \eqref{eq::bias through weighted combination}--\eqref{eq::variance through weighted combination}. One can, therefore, start with some $\lambda$ and the corresponding $k =k^*$ and then consider larger values of $\lambda$ without decreasing $k$ in Equations  \eqref{eq::bias through weighted combination}--\eqref{eq::variance through weighted combination}. The result will give sharp (up to a constant factor) bounds, which depend on $\lambda$ as follows:
 \begin{align*}
\tilde{B} &= \sum_i\lambda_i|\theta_i^*|^2 \frac{n^{-2}\left(\lambda + \sum_{i > k}\lambda_i\right)}{\left(n^{-1}\left(\lambda + \sum_{i > k}\lambda_i\right) + \lambda_i\right)^2},\\
\tilde{V}&=  \frac1n \sum_i\frac{\lambda_i^2}{\left(n^{-1}\left(\lambda + \sum_{i > k}\lambda_i\right) + \lambda_i\right)^2},
\end{align*}
which are obtained by simply plugging in the definition of $\rho_k$ into  \eqref{eq::bias through weighted combination}--\eqref{eq::variance through weighted combination}.

A particularly interesting case arises when $\lambda$ is large enough that it dominates $\sum_{i > k}\lambda_i$ and all eigenvalues of $A_k$ are equal to $\lambda$ up to a constant multiplier. The corresponding result is given by the following corollary.
\begin{restatable}{cor}{lambdadominatestail}
\label{cor::lambda dominates the tail}
There is a large positive constant $c$ that only depends on $\sigma_x$ such that if
\[
\lambda > cn\lambda_{\lfloor n/c\rfloor} + 2\sum_{i > \lfloor n/c\rfloor}\lambda_i,
\]
then 
 \begin{align*}
B/c \leq& \sum_i\lambda_i|\theta_i^*|^2 \frac{(\lambda/n)^2}{\left(\lambda/n + \lambda_i\right)^2},\\
V/c \leq&   \frac1n \sum_i\frac{\lambda_i^2}{\left(\lambda/n + \lambda_i\right)^2}.
\end{align*}

\end{restatable}

\begin{proof}
The full proof is given in Appendix \ref{sec::alternative form appendix}; the following is its outline:
\begin{enumerate}
\item Use Lemma \ref{lm::controlling A_k iff small ball and rank} to control the eigenvalues of $A_{\lfloor n/c\rfloor}$.
\item Use Theorem \ref{th::showcase upper} to obtain the bounds for $k = k^*$.
\item Use  Theorem \ref{th:upper same as lower main body} to convert the bounds into the form given in Equations \eqref{eq::bias through weighted combination}--\eqref{eq::variance through weighted combination}. 
\item Use  Lemma \ref{lm::change of k is negligible in alternative form} to substitute $k^*$ back with $\lfloor n/c\rfloor$.
\item  Since $\lambda > 2\sum_{i > k}\lambda_i$, $\lambda/n$  is equal to $\rho_k\lambda_{k+1}$ up to a multiplicative constant.
\end{enumerate}
\end{proof}

Note that the statement of Corollary~\ref{cor::lambda dominates the tail} does not require the notion of $k^*$.

\subsection{Comparison with other results}
\label{sec::comparison with Hsu and Hastie}
As we saw in the previous section, the alternative form given by Equations \eqref{eq::bias through weighted combination}--\eqref{eq::variance through weighted combination} has milder dependence on the choice of $k^*$ than our main bounds \eqref{eq::bias bound main}--\eqref{eq::variance bound main} and allows to compare to classical results for in-sample error of ridge regression. 
In this section we use it to compare with more recent developments: the non-asymptotic bounds in \cite{Hsu_random_design_ridge} and \cite{hastie2020surprises}.

First of all, we follow \cite{Hsu_random_design_ridge} and introduce the following notion of effective dimension of the problem:
\[
d(\bar{\lambda}):= \sum_i \frac{\lambda_i}{\bar{\lambda} + \lambda_i},
\]
where $\bar{\lambda}$ is a parameter which can informally be understood as effective level of regularization.  \cite{Hsu_random_design_ridge}  provide non-asymptotic bounds for $B$ and $V$ in the regime when
\begin{equation}
\label{eq::Hsu_random_design_ridge lambda range}
n \geq cd(\lambda/n) \log(1 + d(\lambda/n)),
\end{equation}
(see their Theorem 2).\footnote{Note that in \cite{Hsu_random_design_ridge}, the scaling of the regularization parameter is different from ours: to express their results in our terms one needs to substitute their $\lambda$ by $\lambda/n$ in our notation.} The simplified version of their results given in Remark 17 gives the following bounds:\footnote{Note that under our assumptions, $\mathrm{approx}(x) = 0$, where  $\mathrm{approx}(x)$ is defined in Equation~(7) in \cite{Hsu_random_design_ridge}.}
\begin{align*}
B \leq& \left(1 + \frac{c(1 + d(\lambda/n))}{n}\right)\sum_i\lambda_i|\theta_i^*|^2 \frac{(\lambda/n)^2}{\left(\lambda/n + \lambda_i\right)^2},\\
V \leq&\frac{c}{n}\sum_i \frac{\lambda_i^2}{(\lambda/n + \lambda_i)^2},\\
\end{align*}
where $c$ is some constant that depends on the concentration properties of the data. 
This is the same as the result of Corollary \ref{cor::lambda dominates the tail}, but with different constants. However, our  Corollary \ref{cor::lambda dominates the tail} covers a wider range of $\lambda$ if $n$ is large enough. This follows from the following lemma, which is proven in Appendix~\ref{sec::alternative form appendix}:
\begin{restatable}{lemma}{invertstiltjes}
\label{lm::range of lambda through Stiltjes transform}
Suppose that $n \geq c^2 + c$ for some $c > 0$ and take
\[
\lambda = cn\lambda_{\lfloor n/c\rfloor} + 2\sum_{i > \lfloor n/c\rfloor}\lambda_i.
\]

Then \[d(\lambda/n) \geq \frac{n}{2\max(2, (c+1)^2)}.\]
\end{restatable}

Indeed, $d(\lambda/n)$ is a decreasing function of $\lambda$, and due to Lemma $\ref{lm::range of lambda through Stiltjes transform}$ the range of $\lambda$ for which Corollary \ref{cor::lambda dominates the tail} is applicable when $d(\lambda/n) = O(n)$, while Equation \eqref{eq::Hsu_random_design_ridge lambda range} restricts to the range $d(\lambda/n) = O(n/\log n)$.

After we posted the first preprint of this paper, the following non-asymptotic bound for the interpolating regime (i.e., $\lambda = 0$) appeared in \citep{hastie2020surprises}: informally
\[
|V - V_S|\leq  \frac{c}{n^{1/7}}, \quad |B - B_S|\leq \frac{c\|\theta^*\|^2}{n},
\]
 where $c$ is a constant, $V_S$ and $B_S$ are defined as\footnote{Here we introduce the notation $\tilde{\lambda}:= (\gamma c_0)^{-1}$, where $\gamma$ and $c_0$ are parameters used in \cite{hastie2020surprises}.}
\begin{align}
V_S:=& \tilde{\lambda}^{-1}  \frac{\sum_i \frac{\lambda_i^2}{(1 + \tilde{\lambda}^{-1}  \lambda_i)^2}}{\sum_i \frac{\lambda_i}{(1 + \tilde{\lambda}^{-1} \lambda_i)^2}}\label{eq::surprises variance},\\
B_S :=& \left(1 + V_S\right)\sum_i \frac{\lambda_i |\theta^*_i|^2}{(1 + \tilde{\lambda}^{-1} \lambda_i)^2}\label{eq::surprises bias},
\end{align}
 and $\tilde{\lambda}$ is the solution to the  equation $n = d(\tilde{\lambda}).$
See their Definition 1 and Theorem 2 for the exact statement.\footnote{Note that there is a typo in their definition of $\mathscr{V}$: a multiplicative factor of $c_0$ is missing.}

Note that because of the equation for $\tilde\lambda$
\[
{\sum_i \frac{\lambda_i^2}{(\tilde{\lambda} +  \lambda_i)^2}} + {\sum_i \frac{\tilde{\lambda}\lambda_i}{(\tilde{\lambda} + \lambda_i)^2}} = {\sum_i \frac{\lambda_i(\lambda_i + \tilde\lambda)}{(\tilde{\lambda} +  \lambda_i)^2}} = d(\tilde{\lambda}) = n.
\]
This allows us to rewrite  \eqref{eq::surprises variance}--\eqref{eq::surprises bias} as
\begin{align}
V_S:=&\frac{1}{1 - \frac{1}{n}\sum_i \frac{\lambda_i^2}{(\tilde{\lambda}  + \lambda_i)^2}}\cdot \frac{1}{n} \sum_i \frac{\lambda_i^2}{(\tilde{\lambda}  + \lambda_i)^2}\label{eq::Hastie_variance_for_comparison},\\
B_S :=& \left(1 + V_S\right)\sum_i \lambda_i |\theta^*_i|^2 \frac{\tilde{\lambda}^2}{(\tilde{\lambda} +  \lambda_i)^2}\label{eq::Hastie_bias_for_comparison}.
\end{align}

Comparing these equations with \eqref{eq::bias through weighted combination}--\eqref{eq::variance through weighted combination} reveals that they are the same up to a constant multiplier whenever $\tilde{V}\leq 1 - 1/c$ for some constant $c$ and $\rho_k\lambda_{k+1}$ is up to a constant equal to $\tilde{\lambda}$. In the following, we show that this is indeed the case.

Recall that these results from \citep{hastie2020surprises} are for the interpolating regime, i.e., $\lambda = 0$. Let's see how $\tilde{\lambda}$ is related to $\lambda_{k+1}\rho_k$. The connection is given by the following lemma.
\begin{lemma}
Suppose that $k < n/c$ and  $\rho_k > c$ for some constant $c > 1$ . Then 
\[
\frac{\tilde\lambda}{\lambda_{k+1}\rho_k} \in \left(1 - \frac1c, \frac{1}{1 - \frac1c} \right).
\]
\end{lemma}
\begin{proof}
Denote $a = \frac{\tilde\lambda}{\lambda_{k+1}\rho_k}$. Then we can write
\begin{equation*}
n = \sum_i \frac{\lambda_i}{\lambda_i + a\lambda_k \rho_k}
\geq \sum_{i > k} \frac{\lambda_i}{\lambda_{k+1} (a\rho_k + 1)}
=  \frac{n\rho_k}{a\rho_k + 1},\\
\end{equation*}
which implies $a\rho_k + 1 \geq \rho_k$, so $a \geq 1 - 1/\rho_k > 1 - 1/c$.

For the upper bound on $a$ we write
\begin{equation*}
n = \sum_i \frac{\lambda_i}{\lambda_i + a\lambda_k \rho_k}
\leq k + \sum_{i > k} \frac{\lambda_i}{a\lambda_{k+1} \rho_k }
= k +  \frac{n}{a},\\
\end{equation*}
which gives $a \leq n/(n-k) < c/(c-1)$.
\end{proof}

The similarity of Equations \eqref{eq::Hastie_variance_for_comparison}--\eqref{eq::Hastie_bias_for_comparison} with our results should not be taken for granted, and it is actually quite surprising.  As we explain in Section \ref{sec::detailed comparison}, the regime considered in \cite{hastie2020surprises} is significantly different, so it is rather unclear why the results would have the same form.

\section{Negative regularization}
\label{sec::negative regularization main body}

The aim of this section is to find a family of regimes in which the optimal level of ridge regularization is negative. 
Since we are comparing different values of $\lambda$ in this section, the following notation will be useful: recall that for any $k$
\[
\rho_k(0) := \frac{1}{n\lambda_{k+1}}\sum_{i > k}\lambda_i,
\]
the value of $\rho_k$ for $\lambda = 0$. Intuitively, the components of the tail provide regularization for the first $k$ components, and the larger $\rho_k$ is, the more is that regularization. Thus, one could expect that if there is an abrupt jump in the sequence  $\{\rho_k(0)\}_{k=0}^p$, then that additional regularization is too large and negative $\lambda$ may be optimal.

 As we investigate further, a jump in $\rho_k(0)$ is indeed one of the sufficient conditions for optimality of negative regularization, but not the only one: the strength of the noise and how the signal is distributed among the principal components of the data also play an important role.

We start the discussion with several informal observations. The first observation one can make is that $V$ is a decreasing function of $\lambda$: indeed, $ V = \tr(\Sigma^{1/2} X^\top A^{-2} X \Sigma^{1/2})$ and increasing $\lambda$ increases all eigenvalues of $A$.
Thus, negative regularization cannot help with damping the noise compared to non-negative regularization, and the noise should be sufficiently small in order for negative regularization to be beneficial. 

Now let's look at the role of the signal in the tail. It contributes to error in two ways: first --- the components in the tail are not getting estimated themselves, second --- the signal that comes from those components acts as additional noise for estimation of the first $k$ components. When $\lambda$ is non-negative, the error of the first type dominates the error of the second type, but negative $\lambda$ can amplify the noise and result in error of the second type dominating. Therefore, the signal in the tail also needs to be sufficiently small in order for negative regularization to be optimal.

The final observation is the following: since we only compute the bounds up to a constant multiplier, the bound in Theorem \ref{th::showcase upper} cannot distinguish between negative and zero regularization. To see this, consider the form of the bound given in Section \ref{sec::alternative form of the bound}: up to a constant factor the bound is a weighted combination in each component with weight $\lambda + \sum_{i > k}\lambda_i$, and as $\lambda$ increases there is no need to change $k$. Now it is easy to see that for all $\lambda$ in range from $-\gamma \sum_{i > k}\lambda_i$ to zero, that weight is the same up to a constant factor. Thus, negative regularization can only decrease the excess risk by more than a constant factor in the critical regime, i.e., $\lambda = -\sum_{i > k} \lambda_i + \Diamond$ where $\Diamond$ is of smaller order than $\sum_{i > k} \lambda_i$. To consider such $\lambda$ and have $A_k$ PD we need tight concentration of eigenvalues of  $X_\ktoinf X_\ktoinf^\top$ around $\sum_{i > k}\lambda_i$. To ensure such tight control we restrict ourselves to the case of independent components, i.e., when Assumption \ref{as::independent components} is satisfied. In this case, the eigenvalues of $X_\ktoinf X_\ktoinf^\top$ can be bounded according to the following statement that was shown as an intermediate step in the proof of Lemma S.9 in \citep{benign_overfitting}.
\begin{lemma}
\label{lm::tight control of eigenvalues for independent coordinates}
Under assumption \ref{as::independent components} there exists a constant $c$ that only depends on $\sigma_x$ s.t. with probability at least $1 - ce^{-n/c}$,
\begin{align*}
\mu_1(X_\ktoinf X_\ktoinf^\top) \leq& \sum_{i > k} \lambda_i + c\left(n\lambda_{k+1} + \sqrt{n\sum_{i > k} \lambda_i^2}\right),\\
\mu_n(X_\ktoinf X_\ktoinf^\top) \geq& \sum_{i > k} \lambda_i - c\left(n\lambda_{k+1} + \sqrt{n\sum_{i > k} \lambda_i^2}\right).
\end{align*}
\end{lemma}

The fluctuations $n\lambda_{k+1} + \sqrt{n\sum_{i > k} \lambda_i^2}$  will be of smaller order than $\sum_{i > k} \lambda_i$ if $\rho_k(0)$ is larger than a constant, which is shown by the following  bounds:
\begin{gather}
n\lambda_{k+1} =\frac{1}{\rho_k(0)}\sum_{i > k} \lambda_i,\label{eq::eigenvalue deviation via rho_k first} \\
\sqrt{n\sum_{i > k} \lambda_i^2 }\leq \sqrt{n\lambda_{k+1}  \sum_{i > k} \lambda_i } = \frac{1}{\sqrt{\rho_k(0)}}\sum_{i > k} \lambda_i\label{eq::eigenvalue deviation via rho_k second}.
\end{gather}

Using this lemma allows us to obtain  following two lemmas. See Appendix \ref{sec::negative regularization appendix} for the proofs. 
\begin{restatable}[Lower bound on the bias for any non-negative regularization]{lemma}{uniformlower}
\label{lm::non-negative regularization uniform lower bound main body}
There exist constants $b, c$ that only depend on $\sigma_x$  such that the following holds: suppose that assumptions \ref{as::independent components} and \ref{as::prior on theta^*}$(\bar{\theta})$ hold. Take $k = \min\{\kappa: \rho_\kappa(0) > b\}$ and suppose that $k > 0$. Then with probability at least $1 - ce^{-n/c}$ for any $\lambda \geq 0$
\[
\E_{\theta^*}B \geq  \frac{1}{c}\|\bar{\theta}_\uptok\|_{\Sigma_\uptok^{-1}}^2\frac{\left(\sum_{i > k} \lambda_i\right)^2}{n^2}.
\]
\end{restatable}

\begin{restatable}[Upper bound on excess risk for some negative regularization]{lemma}{negreggeneralization}
\label{lm::negative regularization generalization bound main body}
There exists a constant $c$ that only depends on $\sigma_x$ such that the following holds: suppose that assumptions \ref{as::prior on theta^*}$(\bar{\theta})$ and \ref{as::independent components}  hold and that $\rho_k(0) > c$ for some $k < n/c$. Assume also that 
\begin{equation}
\label{eq::low noise condition for negative lambda bound}
v_\eps^2 \leq\frac{1}{c}\|\bar{\theta}_\uptok\|_{\Sigma_\uptok^{-1}}^2 \frac{\left(\sum_{i > k} \lambda_i\right)^2}{n^3\left(\sum_{i > k}\lambda_i^2\right)^2}.
\end{equation}
Then there exists such $\lambda < 0$ that with probability at least $1 - ce^{-n/c}$

\begin{align*}
\E_{\theta^*}B + v_\eps^2 V \leq& c\left(v_\eps^2\frac{k}{n} + v_\eps \|\bar{\theta}_\uptok\|_{\Sigma_\uptok^{-1}}\sqrt{\frac{\sum_{i > k} \lambda_i^2}{n}}+ \|\bar{\theta}_\uptok\|_{\Sigma_\uptok^{-1}}^2\frac{\lambda_{k+1}\sum_{i > k} \lambda_i}{n}   + \|\bar{\theta}_\ktoinf\|_{\Sigma_\ktoinf}^2  \right).
\end{align*}

\end{restatable}

Lemma \ref{lm::non-negative regularization uniform lower bound main body} provides a lower bound on the expected (over noise and $\theta^*$) excess risk which holds w.h.p. uniformly over all non-negative $\lambda$. Lemma \ref{lm::negative regularization generalization bound main body} provides an upper bound that can be achieved by some negative $\lambda$. Combining these two lemmas gives a sufficient condition for the optimal $\lambda$ to be negative, which is given by the following theorem.
\begin{theorem}
\label{th::negative regularization main main body}
There exist constants $b$ and $c$ that only depend on $\sigma_x$ such that the following holds. Suppose that assumptions \ref{as::prior on theta^*}$(\bar{\theta})$ and \ref{as::independent components} hold. Take $k = \min\{\kappa: \rho_\kappa(0) > b\}$ and suppose that $k < n/c$. The value of $\lambda$ that minimizes $\E_{\theta^*}B + v_\eps V$ will be negative with probability at least $1 - ce^{-n/c}$ if the following conditions are satisfied:
\begin{align*}
\text{small noise:}&& v_\eps^2 \leq& \frac{\|\bar{\theta}_\uptok\|_{\Sigma_\uptok^{-1}}^2}{c} \min\left( \frac{\left(\sum_{i > k}\lambda_i\right)^2}{nk}, \frac{\left(\sum_{i > k}\lambda_i\right)^4}{n^3\sum_{i > k}\lambda_i^2}\right), \\
\text{jump in effective rank:}&&\rho_k(0) >& c,\\
\text{small signal in the tail:}&&\|\bar{\theta}_\ktoinf\|_{\Sigma_\ktoinf}^2  \leq&\frac{1}{c}\|\bar{\theta}_\uptok\|_{\Sigma_\uptok^{-1}}^2\left(\frac{\sum_{i > k}\lambda_i}{n}\right)^2.\\
\end{align*}
\end{theorem}
\begin{proof}
It is easy to see that by taking $c$ large enough, the conditions of Lemmas \ref{lm::negative regularization generalization bound main body} and \ref{lm::non-negative regularization uniform lower bound main body} are satisfied, and the upper bound in Lemma \ref{lm::negative regularization generalization bound main body} becomes lower than the lower bound in Lemma \ref{lm::non-negative regularization uniform lower bound main body}.
\end{proof}

We see that the conditions indeed align with the intuition outlined in the beginning of this section: we need small variance, small signal in the tail, and a sharp jump in effective rank. However, we do not have matching lower bounds in the critical regime when Assumption \ref{as::gamma}$(k, \gamma)$ is not satisfied for a constant $\gamma > 1$. Thus, we don't know whether these conditions are also necessary.

\section{Comparison to other works}
\label{sec::detailed comparison}

As we mentioned in Section \ref{sec::related work}, recently there has been a number of papers studying population risk of interpolating solutions of linear regression, and we gave a rough split of those results into three categories there. Here we elaborate on the comparison between the approaches and results.

Results from the first category \citep{Dobriban2015HighDimensionalAO, Hastie2019SurprisesIH, wu2020optimal, richards2020asymptotics} compute exact asymptotic expressions for the excess risk assuming that $p/n$ goes to some constant as $p, n$ go to infinity, and that the spectral distribution of $\Sigma$ converges to some limiting distribution. From the point of view of our approach, such distributions are indistinguishable from isotropic: indeed, the very existence of limiting spectral measure implies that almost all eigenvalues are within a constant factor of each other.  Many of those works even assume explicitly that the spectrum of $\Sigma$ is upper- and lower-bounded by two constants \citep[page 7]{richards2020asymptotics}, \citep[Assumption 1]{wu2020optimal}, \citep[Theorem 3]{Hastie2019SurprisesIH}.  Our results don't need any asymptotic set up, and apply to $p = \infty$ with some fixed summable sequence $\lambda_i$, which has no meaningful notion of limiting distribution, and  no separation from zero is needed. For example, our setup covers kernel regression with a fixed kernel and increasing number of data points. On the other hand, when all $\lambda_i$ are within a constant factor of each other, our lower bounds become $B \geq \|\theta^*\|_\Sigma/c$ and $V \geq 1/c$, so the constant part of the whole signal doesn't get learned and the variance term is at least a constant, i.e., the asymptotic expressions obtained in the works from this category are all just different constants and our approach cannot distinguish them.  Therefore, we answer significantly different questions: while the asymptotic work distinguishes between constant error rates, we investigate when the error can be less than a constant. The final difference with our work is rather technical but quite strong: all the works in this category assume that the coordinates of the data become independent if multiplied by the inverse square root of the covariance.  This assumption stems from asymptotic random matrix theory techniques, on which these papers are based. To the best of our knowledge, it is not known how to extend these techniques beyond random matrices with independent elements. Our approach, however, does not require the coordinates to be independent. 

When it comes to the second category, featurized or kernel regression \citep{ montanari2020interpolation, ghorbani2020neural, mei2019generalization, ghorbani2020linearized,liang2020multiple}, the difference from our approach is that we do not assume any particular mechanism for data generation or how the features are constructed, but we directly make assumptions about feature vectors. Our results can in principle be applied in this setting if one computes the spectrum of the population covariance for particular features or kernels and the corresponding sub-Gaussian norms. The major difficulty that precludes such a direct comparison is that that computation is not straightforward. The works from this category operate in a more particular setting and circumvent the computation of the spectrum of $\Sigma$. On the other hand, it is not hard to trace strong similarities with our approach on the level of the proof. First of all,  all the papers in this category that we are aware of assume that the data comes from a very regular distribution: either  $d$-dimensional isotropic data with i.i.d. coordinates \citep[Assumption 1]{liang2020multiple}, or data from the uniform distribution on the sphere \citep[abstracts]{mei2019generalization, ghorbani2020linearized}, \citep[Section 3.2]{montanari2020interpolation}, or  data from the product of two uniform distributions on spheres  \citep[ Section 2.1]{ghorbani2020neural}. Second, in all those papers the kernel is either spherically symmetric \citep[Section 2.2]{ghorbani2020neural}, \citep[Equation 4]{liang2020multiple} or close to being spherically symmetric due to isotropic initialization of the neural network or isotropic choice of random features  \citep[Assumption 1]{ghorbani2020linearized}, \citep[ Thorem 2]{mei2019generalization}, \citep[Section 3.2]{ montanari2020interpolation}. After that, they consider the regime where $n$ is large compared to $d^\alpha$ for some $\alpha$ \citep[Assumption 1]{ montanari2020interpolation}, \citep[Theorem 1]{ ghorbani2020neural}, \citep[abstracts]{ mei2019generalization, ghorbani2020linearized,liang2020multiple}\footnote{In \citep{ mei2019generalization} $\alpha = 1$.}. Finally, all those papers  derive that kernel regression works effectively as ridge regression with polynomial features up to degree $\alpha $  \citep[abstracts]{montanari2020interpolation, ghorbani2020linearized}, \citep[Theorem 1]{ ghorbani2020neural}, \citep[Proposition 1 and Section 2.3]{liang2020multiple}. The only exception is \citet{mei2019generalization}, who derive asymptotic expressions for excess risk when the true function is affine (i.e., a polynomial of degree $1$) plus Gaussian misspecification. The connection with our results is that in such a regime (uniform distribution on the sphere, spherically symmetric kernel) polynomials are exactly the eigenfunctions of the kernel operator, which plays the role of the covariance operator, and there are $k \approx d^\alpha$ of polynomials of degree at most $\alpha$. Thus, their approach is similar to ours: separate the first $k$ eigendirections (or their approximations) and show that other directions act as regularization. 

The third category is where this paper belongs, so a more concrete comparison to other results is possible. Sections \ref{sec::benign overfitting story} and \ref{sec::our contribution intuition}  provide a detailed explanation of how our work generalizes the work of \citet{benign_overfitting}. \citet{Kobak2020OptimalRP} proved that negative ridge regularization is optimal in a spiked covariance model with one spike, which is a simple particular case with $k=1$ of our results. In Section \ref{sec::negative regularization main body}, we showed that negative regularization is optimal under a rich set of covariance structures, and gave general sufficient conditions.  \citet{chinot2021robustness} obtain non-asymptotic bounds for bias and variance in the ridgeless setting. They assume Gaussian data and the existence of $k^*$, which means that our results apply in their setting. Our bound for the bias term is tight, so it cannot be worse than theirs by more than a constant multiplier. At the same time, their bound on the bias term can be much worse than ours: note that their bound depends on $\|\theta^*\|$, while our bound scales with $\|\theta^*\|_\Sigma$, therefore their bound can be arbitrarily close to infinity while our bound stays finite. When it comes to the variance term, the bound of \citet{chinot2021robustness} is larger but holds with smaller probability, as they discuss when they compare their results to those in \citet{benign_overfitting}. \citet{Derezinski} start with an arbitrary covariance matrix and construct a specific data distribution for which the approximation error $\E\|\hat{\theta} - \theta^*\|^2$ has an explicit expression. We provide bounds for the excess risk $\|\hat{\theta} - \theta^*\|^2_\Sigma$, so our results are not directly comparable to theirs.  \citet{derezinski2020precise} consider expectation of the projector on the orthogonal complement to the span of i.i.d. data with arbitrary covariance and  derive tight upper and lower bounds for it with respect to Loewner order. The bias term in our setting is exactly such a projection of $\theta^*$, but measured in $\|\cdot\|_\Sigma$. Because of this mismatch in the norm, the results of \citet{derezinski2020precise} do not translate into our results  directly, even if we consider the expectation of the bias term.

\section{Conclusions }
\label{sec::conclusions}
We studied the excess risk of ridge regression and showed how geometry of the data can influence both which part of the signal is learned and how the noise is damped. For a range of values of the regularization parameter we showed that learning can be seen as the composition of two parts: classical ridge regression in the first $k$ components (the "essentially low-dimensional part") and learning the zero estimator in the rest of the components (the "essentially high-dimensional part"). We introduced a general assumption under which the data is “essentially high-dimensional”, and provided geometric sufficient conditions for its satisfaction. Moreover, we investigated the regime in which the “essentially high-dimensional part” is too high-dimensional, and derived general sufficient conditions for negative regularization to be optimal: small noise, small energy of the  "essentially low-dimensional part", but an abrupt jump in the effective rank.

On the technical side, our proof decouples cleanly into an algebraic part, which holds with probability 1 for non-negative regularization,\footnote{For the case of negative regularization we need to condition on the event that all the necessary symmetric matrices are PD.} and the probabilistic part, where we plug in well-known concentration results from high-dimensional probability. This makes it easy to trace how different terms in the bound correspond to the parts of the estimator, and supports the geometric interpretation given above.

We provided a thorough overview of the related papers, and explained how our results are significantly different from them despite some optical similarities. Those similarities, however, are intriguing,  and hint at the task of developing a unified treatment of different regimes of overparameterized linear regression as a promising direction of future work.

\section*{Acknowledgements}
We gratefully acknowledge the support of the NSF through grants DMS-2023505 and DMS-2031883 and of the Simons Foundation through award \#814639.

\appendix
\newpage

\section{Definitions and Notation}
\label{sec::notations appendix}
\subsection{Sub-Gaussianity}
\label{sec::sub-Gaussianity appendix}
A random variable $z$ is sub-Gaussian if it has a finite sub-Gaussian norm 
\[
\|z\|_{\psi_2}:=\inf\left\{t>0:\E\exp(z^2/t^2)\le 2\right\}.
\]
The sub-Gaussian norm of a random vector $Z$ is 
\[
\|Z\|_{\psi_2}:=\sup_{s\not=0}\|\langle s,Z\rangle/\|s\|\rangle\|_{\psi_2}.
\]

\subsection{Standard mathematical objects}
\begin{itemize}
\item $M[i,j]$ denotes the element of the matrix $M$ which stands at the intersection of the $i$-th row and $j$-th column.
\item $\|v\|$ denotes the Euclidean norm for a vector $v$ in $\R^d$ for some $d$, .
\item $\|M\|$ denotes the operator norm (i.e., maximum singular value)  for a matrix $M$ in $\R^{m \times n}$ for some $m, n$.
\item $\tr(M)$ denotes the trace of a square matrix $M$.
\item $\|M\|_F $ denotes the  Frobenius norm for  a matrix $M$ in $\R^{m \times n}$,  i.e., $\|M\|_F := \sqrt{\tr(MM^\top)}$.
\item $\Sc^{p-1}$ denotes the unit sphere in $\R^p$, i.e., $\Sc^{p-1} = \{x \in \R^p: \|x\| = 1\}$.
\item $I_m$ is the $m\times m$ identity matrix.
\item $\|v\|_M := \sqrt{v^\top Mv}$ for any positive semidefinite (PSD) matrix $M \in \R^{m\times m}$ and any $v \in \R^m$.
\item $\mu_1(M)\ge\cdots\ge\mu_m(M)$ are the eigenvalues of a symmetric matrix $M\in \R^{m \times m}$ in decreasing order.
\end{itemize}

\subsection{Data and the learning procedure}

Recall from Section~\ref{sec::ridge setup main body} that 
\begin{itemize}
\item $X \in \R^{n\times p}$ --- a random matrix with i.i.d. centered rows.
\item $y = X\theta^* + \eps$ is the response vector, where $\theta^* \in \R^p$ is some unknown vector, and $\eps$ is noise,
\item components of $\eps$ are independent and have variance $v_\eps$, 
\item $\{x^i\}_{i = 1}^n$ are columns of $X^\top$ (i.e., $\{x^i\}_{i = 1}^n$ are our i.i.d. data points in $\R^p$).
\item $x$ denotes a new random draw from the data distribution, i.e., $x$  is independent from $X, \eps$ and $x$ has the same distribution as $x^1$.
\item $\Sigma = \diag(\lambda_1, \dots, \lambda_p)$ is the covariance matrix of a row of $X$.
\item $Z = X\Sigma^{-1/2}$, $\{z_i\}_{i=1}^p$ are columns of $Z$.
\item the rows of $Z$ are sub-Gaussian with sub-Gaussian norm at most $\sigma_x$,

\item ridge regression outputs $\hat{\theta}( y) := X^\top(\lambda I_n + XX^\top)^{-1}y$.
\end{itemize}

\subsection{Splitting the coordinates}
For some $k < n$ we  spit the coordinates into two groups: the first $k$ components and the rest of the components. Thus we introduce the following notation. Consider integers $a, b$ from $0$ to $\infty$ (we always either take $a=0$ and $b=k$ or $a=k$ and $b = \infty$).  
\begin{itemize}
\item For any matrix $M\in \R^{n \times p}$ denote $M_{a:b}$ to be the matrix that is comprised of the columns of $M$ from $a+1$-st to $b$-th. 
\item For any vector $\eta \in \R^p$ denote $\eta_{a:b}$ to be the vector comprised of components of $\eta$ from $a+1$-st to $b$-th. 
\item $\Sigma_\uptok = \diag(\lambda_1, \dots, \lambda_k),$ and $\Sigma_\ktoinf = \diag(\lambda_{k+1}, \lambda_{k+2}, \dots).$ 
\item $A_k = \lambda I_n + X_\ktoinf X_\ktoinf^\top$.
\item $r_k = \frac{1}{\lambda_{k+1}}\left(\lambda + \sum_{i > k}\lambda_i\right).$
\item $\rho_k = r_k/n$.
\item $\rho_k(0) = \frac{1}{n\lambda_{k+1}}\sum_{i > k}\lambda_i$.
\item For any $i$ we denote $A_{-i} = A_{-i} := X_{0:i-1}X_{0:i-1}^\top + X_{i:\infty}X_{i:\infty}^\top + \lambda I_n.$
\end{itemize}

\section{Ridge regression}
\label{sec::ridge appendix}
We are interested in evaluating the MSE of the ridge estimator. For positive regularization parameter $\lambda$ that estimator is defined as
\begin{align*}
\hat{\theta}(y) = \hat{\theta}( y) =& \argmin_\theta\left\{\|X\theta - y\|_2^2 + \lambda\|\theta\|_2^2\right\}\\
=& \left(\lambda I_p + X^\top X\right)^{-1}X^\top y.\\
\end{align*}

In the overparametrized case (i.e., $p > n$), however, the latter expression has a singularity at zero, because the matrix $X^\top X$ does not have full rank. If $\lambda = 0$ the solution to the minimization problem above is not unique. Moreover, if $\lambda < 0$, no solution exists at all because we are minimizing a quadratic form whose matrix has negative singular values. To alleviate these issues and extend the definition of the solution to non-positive values of $\lambda$, we propose the following: since the matrix $X^\top X$ doesn't have full rank, we can apply the Sherman-Morrison-Woodbury formula:
\[
\left( \lambda I_p + X^\top X\right)^{-1} = \lambda^{-1}I_p   - \lambda^{-2}X^\top(I_n + \lambda^{-1}XX^\top)^{-1}X.
\]
So, 
\begin{align*}
\hat{\theta}( y) &= \left(\lambda I_p + X^\top X\right)^{-1}X^\top\\
=& \lambda^{-1}X^\top - \lambda^{-2}X^\top(I_n + \lambda^{-1}XX^\top)^{-1}XX^\top\\
=& \lambda^{-1}X^\top - \lambda^{-1}X^\top(I_n + \lambda^{-1}XX^\top)^{-1}(\lambda^{-1}XX^\top + I_n - I_n)\\
=& \lambda^{-1}X^\top(I_n + \lambda^{-1}XX^\top)^{-1}\\
=& X^\top(\lambda I_n + XX^\top)^{-1}y.
\end{align*}

The matrix $XX^\top$ has full rank, and the expression above is continuous in $\lambda$ as long as $XX^\top + \lambda I_n$ stays PD. When $\lambda = 0$, $X^\top(\lambda I_n + XX^\top)^{-1}y$ is the minimum norm interpolating solution (the same solution that was considered in \citep{benign_overfitting}. Therefore, we use the expression
\begin{equation*}
\hat{\theta}( y) := X^\top(\lambda I_n + XX^\top)^{-1}y
\end{equation*}
to define the ridge regression solution for any $\lambda > -\mu_n(XX^\top)$. 

Note that $\hat{\theta}( y) $ is linear in $y$. Since we have $y = X\theta^* + \eps $ we can also write
\[
\hat{\theta}( y) = \hat{\theta}( X\theta^*) + \hat{\theta}( \eps). 
\]
The first term is the noiseless estimate; its error gives the bias term. The second term is the estimate obtained when the signal is pure noise. It gives the variance term. 

For the full MSE we have
\begin{align*}
\|\hat{\theta}( y) - \theta^*\|_\Sigma^2 =& \|\hat{\theta}( X\theta^*) + \hat{\theta}( \eps)- \theta^*\|_\Sigma^2\\
\leq& 2\|\hat{\theta}( X\theta^*)- \theta^*\|_\Sigma^2 + 2\|\hat{\theta}( \eps)\|_\Sigma^2\\
=&2(B + V_\eps),
\end{align*}
where we introduced bias $B$ and variance $V_\eps$:
\begin{alignat*}{3}
&B := &&\|\hat{\theta}( X\theta^*)- \theta^*\|_\Sigma^2 &&= \|(I_p - X^\top(\lambda I_n + XX^\top)^{-1}X)\theta^*\|^2_\Sigma,\\
&V_\eps := &&\|\hat{\theta}( \eps)\|_\Sigma^2 &&= \|X^\top(\lambda I_n + XX^\top)^{-1}\eps\|_\Sigma^2. 
\end{alignat*}

Finally, since $V_\eps$ is a quadratic form in $\eps$, by Lemma \ref{lm::quadratic form concentration} if the noise is sub-Gaussian, then its value is controlled by its expectation with high probability. That expectation, in its turn, scales linearly with the variance $v_\eps^2$ of the noise. Therefore, we can decouple the effect of the noise and only study the following purified variance term:
\begin{align*}
V :=& \frac{1}{v_\eps^2}\E_\eps V_\eps\\
=& \tr((\lambda I_n + XX^\top)^{-1}X\Sigma X^\top(\lambda I_n + XX^\top)^{-1})\\
=&\tr(\Sigma X^\top (\lambda I_n + XX^\top)^{-2} X).
\end{align*}

The main aim of our work is to give sharp non-asymptotic bounds for $B$ and $V$. 
\section{Concentration inequalities}

\begin{lemma}[Non-standard norms of sub-Gaussian vectors ] 
\label{lm::sub-Gauss_norm_upper}
Suppose $z$ is a sub-Gaussian vector in $\R^p$ with $\|z\|_{\psi_2} \leq \sigma$.  Consider $\Sigma = \diag(\lambda_1, \dots, \lambda_p)$ for some positive non-increasing sequence $\{\lambda_i\}_{i=1}^p$. Then for some absolute constant $c$ for any $t > 0$
\[
\Pbb\left\{\|\Sigma^{1/2} z\|^2 > c\sigma^2\left(t\lambda_1 + \sum_i \lambda_i\right)\right\} \leq 2e^{-t/c}.
\]
\end{lemma}
\begin{proof}
The argument consists of two parts: first, we obtain a bound that only works well in the case when all $\lambda_i$ are approximately the same. Next, we split the sequence $\{\lambda_i\}$ into pieces with approximately equal values within each piece and obtain the final result by applying the first part of the argument to each piece.

{\bf First part:} Consider a $1/4$-net $\{u_j\}_{j=1}^m$ on $\Sc^{p-1}$, such that $ m \leq 9^p$.  Note that for any vector $v \in \Sc^{p-1}$ there exists an element $u_j$ of that net such that $\langle v, u_j\rangle \geq {3}/{4}\cdot \|v\|$. Thus, we have
\[
\|\Sigma^{1/2} z\| \leq \frac{4}{3} \sqrt{\lambda_1}\max_j \langle z, u_j\rangle \leq 2\sqrt{\lambda_1}\max_j \langle z, u_j\rangle.
\]

Since the random variable $\langle z, u_j\rangle$ is $\sigma$-sub-Gaussian, it also holds for any $t > 0$ and some absolute constant $c$ that 
\begin{align*}
\Pbb(|\langle z, u_j\rangle| > t) &\leq 2e^{-ct^2/\sigma^2}, \\
\Pbb(4\lambda_1\langle z, u_j\rangle^2 > 4\lambda_1t\sigma^2) &\leq 2e^{-ct}.
\end{align*}
By multiplicity correction, we obtain
\[
\Pbb\left(\|\Sigma^{1/2} z\|^2 > 4\lambda_1\sigma^2 t + \frac{4\sigma^2\lambda_1\log 9}{c}p\right) \leq 2e^{-ct}.
\]

We see that the random variable $\left(\|\Sigma^{1/2} z\|^2 - \frac{4\sigma^2\lambda_1\log 9}{c}p\right)_+$ has sub-exponential norm bounded by $C\sigma^2\lambda_1$. 

{\bf Second part:} Now, instead of applying the result that we have just obtained to the whole vector $z$, split it in the following way: define the sub-sequence $\{i_j\}$ in such that $i_1 = 1$, and for any $l \geq 1$ $i_{l + 1} = \min\{i: \lambda_i < \lambda_{i_l} / 2\}$. Denote $z_l$ to be a sub-vector of $z$ comprised of components from the $i_l$-th to $(i_{l+1}-1)$-th. Let $\Sigma_l = \diag(\lambda_{i_l}, \dots, \lambda_{i_{l+1}-1}).$ 

Then by the initial argument, the random variable $\left(\|\Sigma_l^{1/2} z_l\|^2 - \frac{4\sigma^2\lambda_{i_l}\log 9}{c}(i_{l+1}-i_l)\right)_+$ has sub-exponential norm bounded by $C\sigma^2\lambda_{i_l}$. Since each next $\lambda_{i_l}$ is at most half of the previous, we obtain that the sum (over $l$) of those random variables has sub-exponential norm at most $2C\sigma^2\lambda_1.$ Combining this with the fact that
\[
\sum_{i = i_l}^{i_{l+1}-1} \lambda_i \geq (i_{l+1}-i_l)\lambda_{i_{l+1}-1} \geq (i_{l+1}-i_l)\lambda_{i_{l+1}}/2,
\]
we obtain that for some absolute constants $c_0, c_1, \dots$ for any $t > 0$
\begin{align*}
2e^{-c_0t}
\geq& \Pbb\left\{\sum_l \left(\|\Sigma_l^{1/2} z_l\|^2 - c_1\sigma^2\lambda_{i_l}(i_{l+1}-i_l)\right) > c_2\sigma^2\lambda_1t\right\}\\
\geq&\Pbb\left\{\|\Sigma^{1/2} z\|^2\geq c_3\sigma^2\sum_i \lambda_i + c_2\sigma^2\lambda_1 t\right\}.
\end{align*}
\end{proof}

\begin{lemma}[Concentration of the sum of squared norms]
\label{lm::sum of norms}
Suppose $Z \in \R^{n \times p}$ is a matrix with independent isotropic sub-Gaussian rows with $\|Z[i, *]\|_{\psi_2} \leq \sigma$.  Consider $\Sigma = \diag(\lambda_1, \dots, \lambda_p)$ for some positive non-increasing sequence $\{\lambda_i\}_{i=1}^p$.  Then for some absolute constant $c$ and any $t \in (0,n)$ with probability at least $1 - 2\exp(-ct)$,
\[
(n-\sqrt{nt}\sigma^2)\sum_{i > k} \lambda_i \leq \sum_{i=1}^n\|\Sigma_\ktoinf^{1/2} Z_{i, \ktoinf}\|^2 \leq (n+\sqrt{nt}\sigma^2)\sum_{i > k} \lambda_i.
\]
\end{lemma}
\begin{proof}
Since $\{Z_{i, \ktoinf}\}_{i=1}^n$ are independent, isotropic and sub-Gaussian, $\|\Sigma_\ktoinf^{1/2} Z_{i, \ktoinf}\|^2$ are independent sub-exponential r.v.'s with expectation $\sum_{i > k} \lambda_i$ and sub-exponential norms bounded by $c_1\sigma^2\sum_{i > k}\lambda_i$. Applying Bernstein's inequality gives
\[
\Pbb\left(\left|\frac1n\sum_{i=1}^n\|\Sigma_\ktoinf^{1/2} Z_{i, \ktoinf}\|^2 - \sum_{i > k} \lambda_i\right| \geq t\sigma^2\sum_{i > k} \lambda_i \right) \leq 2\exp\left(-c_2\min(t, t^2)n\right).
\]
Changing $t$ to $\sqrt{t/n}$ gives the result.
\end{proof}

\begin{lemma}[Weakened Hanson-Wright inequality]
\label{lm::quadratic form concentration}
Suppose $M \in \R^{n\times n}$ is  a (random) PSD matrix and $\eps \in \R^n$ is a centered vector whose components $\{\eps_i\}_{i=1}^n$ are independent and have sub-Gaussian norm at most $\sigma$. Then for some absolute constants $c, C$ and any $t > 1$ with probability at least $1 - 2e^{-t/c}$,
\[
\eps^\top M \eps \leq C\sigma^2t\tr(M).
\]
\end{lemma}
\begin{proof}
By Theorem 6.2.1 (Hanson-Wright inequality) in \citep{vershynin_hdp},  for some absolute constant $c_1$  for any $t > 0$,
\[
\Pbb_M\left\{|\eps^\top M \eps  - \E\eps^\top M \eps | \geq t\right\} \leq 2\exp\left(-c_1\min\left\{\frac{t^2}{\|M\|_F^2\sigma^4 }, \frac{t}{\|M\|\sigma^2 }\right\}\right),
\]
where $\Pbb_M$ denotes conditional probability given $M$.

Since for any $i$, $\E\eps_i = 0$, and $\Var(\eps_i) \lesssim \sigma^2$, and since $M$ is PSD, we have 
\[
\E\eps^\top M \eps \leq c_2 \sigma^2 \tr(M).
\]

Moreover, since $\|M\|_F^2 \leq \tr(M)^2$ and $\|M\| \leq \tr(M),$ we obtain
\[
\Pbb_M\left\{\eps^\top M \eps > \sigma^2(c_2 + t)\tr(M)\right\} \leq 2\exp\{-c_1\min(t, t^2\}).
\]
Restricting to $t > 1$ and adjusting the constants gives the result (note that since the RHS doesn't depend on $M$, we can replace $\Pbb_M$ with $\Pbb$).
\end{proof}

\section{Controlling the singular values}

\label{sec::singular values}

\begin{lemma}[Bound on the norm of non-diagonal part of a Gram matrix]
\label{lm::non-diag part}
Denote $\mathring{A}_k$ to be the matrix $A_k$ with zeroed out diagonal elements: $\mathring{A}_k[i,j] = (1-\delta_{i,j})A_k[i,j]$.
Then for some absolute constant $c$ for any $t>0$ with probability at least $1-4e^{-t/c}$,
\[
\|\mathring{A}_k\| \leq c\sigma^2\sqrt{(t +  n)\left(\lambda_{k+1}^2(t + n) + \sum_{i > k} \lambda_i^2\right)}.
\]

\end{lemma}
\begin{proof}
We follow the lines of the decoupling argument from \cite{Vershynin_2012}. Consider a $1/4$-net $\{u_j\}_{j=1}^m$ on $\Sc^{n-1}$ s.t. $m \leq 9^n$. Then 
\[
\|\mathring{A}_k\| \leq 2 \max_j |u_j^\top   \mathring{A}_k u_j|.
\]
Indeed, take $v \in \Sc^{n-1}$ to be the eigenvector of $\mathring{A}_k$ whose eigenvalue has the largest absolute value $\mu$ (i.e., $\|\mathring{A}_k\| = \mu$), and let $u_j$ be the closest point in the net to $v$. Then
\begin{align*}
\| v - u_j\| \leq& 1/4,\\
 u_j^\top v \geq& 3/4,\\
|u_j^\top  \mathring{A}_k u_j| \geq& |u_j^\top \mathring{A}_k v| - |u_j^\top \mathring{A}_k (v - u_j)| \\
=& |\mu|u_j^\top v - |u_j^\top \mathring{A}_k (v - u_j)|\\
\geq& |\mu|u_j^\top v - \|u_j\|\|\mathring{A}_k \|\|v - u_j\|\\
\geq& |\mu|\left(\frac34 - \frac14\right).
\end{align*}
Denote the $k$-th coordinate of $u_j$ as $u_j[k]$. Note that 
\[
u_j^\top  \mathring{A}_k u_j = 4\E_T \sum_{k \in T \not\ni l} u_j[k]u_j[l] \mathring{A}_k[k,l],
\]
where the expectation is taken over a uniformly chosen random subset $T$ of $\{1, \dots, n\}$ (since $\mathring{A}_k$ has zeroed-out diagonal, we don't need to consider terms with $m=l$ which allows us to sum over $k \in T \not\ni l$). Thus, 
\begin{align*}
|u_j^\top  \mathring{A}_k u_j|
& \leq 4\max_T \left|\sum_{l \in T \not\ni m} u_j[l]u_j[m] \mathring{A}_k[l,m]\right| \\
& = 4\max_T \left|\left\langle \sum_{l \in T}u_j[l] X_\ktoinf[l,*], \sum_{m \not\in T} u_j[m] X_\ktoinf[m, *]\right\rangle\right|.
\end{align*}

Fix $j$ and denote
\begin{align*}
\xi^\top :=& \sum_{l \in T}u_j[l] X_\ktoinf[l, *]\Sigma_\ktoinf^{-1/2},\\
\eta^\top:=& \sum_{m \not\in T} u_j[m] X_\ktoinf[m, *]\Sigma_\ktoinf^{-1/2}.
\end{align*}

Note that since $u_j$ is from the sphere, $\{X_\ktoinf[i, *]\}_{i=1}^n$ are independent, and $l, m$ live in disjoint subsets, the vectors $\xi$ and $\eta$ are independent sub-Gaussian with sub-Gaussian norms bounded by $C\sigma$ for some absolute constant $C$.

First, that means that for some absolute constant $c_1$ we have
\[
\Pbb\left\{\left|\left\langle \Sigma^{1/2}\xi,\Sigma^{1/2}\eta\right\rangle\right| \geq t\sigma\|\Sigma \eta\|\right\} \leq 2e^{-c_1t^2}.
\]

Second, by Lemma \ref{lm::sub-Gauss_norm_upper}, for some constant $c_2$ for any $t > 0$
\[
\Pbb\left\{\|\Sigma \eta\|^2 \geq c_2\sigma^2\left(\lambda_{k+1}^2t + \sum_{i> k} \lambda_i^2\right)\right\} \leq 2e^{-t/c_2}.
\]

We obtain that for some absolute constant $c$ for any $t > 0$ with probability at least $1 - 4e^{-t/c}$
\[
\left|\left\langle \Sigma^{1/2}\xi,\Sigma^{1/2}\eta\right\rangle\right| < c\sigma^2\sqrt{t\left(\lambda_{k+1}^2t + \sum_{i> k} \lambda_i^2\right)}.
\]

Finally, making multiplicity correction for all $j$ (there are at most $9^n$ of them), and all subsets $T$ (at most $2^n$), we obtain that for some absolute constant $c$ with probability at least $1 - 4e^{-t/c}$
\[
\|\mathring{A}_k\| \leq c\sigma^2\sqrt{(t +  n)\left(\lambda_{k+1}^2(t + n) + \sum_{i > k} \lambda_i^2\right)}.
\]
\end{proof}

\begin{lemma}
\label{lm::top eigenvalue}
For some absolute constant $c$, for any $t>0$, with probability at least $1 - 6e^{-t/c}$,
\[
\|X_\ktoinf X_\ktoinf^\top\| \leq  c\sigma_x^2\left(\lambda_{k+1}(t + n) + \sum_{i> k} \lambda_i\right).
\]

\end{lemma}
\begin{proof}
Note that $\|A\| \leq \max_i \|X_{i, *}\| + \|\mathring{A}\|.$ Combining Lemma \ref{lm::sub-Gauss_norm_upper} (with multiplicity correction) and Lemma \ref{lm::non-diag part} gives with probability $1-6e^{-t/c_1}$
\[
\|A\| \leq  c_1\sigma^2\left((t + c_1\log n)\lambda_1 + \sum_i \lambda_i+ \sqrt{(t +  n)\left(\lambda_1^2(t + n) + \sum_i \lambda_i^2\right)}\right).
\]

Now note that 
\begin{align*}
(t + c_1\log n)\lambda_1
\leq& c_1\sqrt{(t +  n)\left(\lambda_1^2(t + n) + \sum_i \lambda_i^2\right)} \\
\leq& c_1\sqrt{\lambda_1^2(t + n)^2 + \lambda_1(t + n)\sum_i \lambda_i}\\
\leq& c_1\left(\lambda_1(t + n) + \sum_i \lambda_i\right),
\end{align*}
where we used $\sqrt{a^2 + ab}\leq a + b$ in the last transition. Removing the dominated (up to a constant multiplier) terms gives the result.

\end{proof}

\controllingAk*
\begin{proof}
We start with the high-probability bounds that we can derive assuming only sub-Gaussianity and independence of data vectors. By Lemma~\ref{lm::sub-Gauss_norm_upper}, for some absolute constant $c$ and for any $t > 0$,
\[
\Pbb\left\{\|X_\ktoinf[i, *]\|^2 > c\sigma_x^2\left(t\lambda_{k+1} + \sum_{i > k} \lambda_i\right)\right\} \leq 2e^{-t/c}.
\]
By Lemma \ref{lm::non-diag part}, for some absolute constant $c$ and for any $t>0$, with probability at least $1-4e^{-t/c}$,
\[
\|\mathring{A}_k\| \leq c\sigma_x^2\sqrt{(t +  n)\left(\lambda_{k+1}^2(t + n) + \sum_{i > k} \lambda_i^2\right)}.
\]

Since $\|A_k\| \leq \lambda + \|\mathring{A}_k\| + \max_i \|X_\ktoinf[i, *]\|$, the above two statements imply that for any $t > 0$ with probability at least $1 - 4e^{-n/c} - 2ne^{-t/c}$,
\begin{align*}
\mu_1(A_k) \leq& \lambda + c\sigma_x^2\sqrt{n\left(\lambda_{k+1}^2n + \sum_{i > k} \lambda_i^2\right)}  + c\sigma_x^2\left(t\lambda_{k+1} + \sum_{i > k} \lambda_i\right)\\
\leq& \lambda + 2c\sigma_x^2\left((t+n)\lambda_{k+1} + \sum_{i > k}\lambda_i + \sqrt{n\sum_{i > k}\lambda_i^2}\right)\\
\leq& \lambda + 3c\sigma_x^2\left((t+n)\lambda_{k+1} + \sum_{i > k}\lambda_i\right),
\end{align*}
where we used the following chain of inequalities to make the last transition:
\[
2\sqrt{n\sum_{i > k}\lambda_i^2} \leq 2\sqrt{n\lambda_{k+1}\sum_{i > k}\lambda_i}\leq n\lambda_{k+1} + \sum_{i > k}\lambda_i.
\]
On the same event,
\[
\mu_n(A_k) \geq \lambda + \min_i\|X_\ktoinf[i, *]\|^2 - c\sigma_x^2\left(n\lambda_{k+1} + \sqrt{n\sum_{i > k}\lambda_i^2}\right).
\]

On the other hand, note that the sum of eigenvalues of $A_k$ is equal to 
\[
\tr(A_k)=  \lambda n + \sum_{i=1}^n\|\Sigma_\ktoinf^{1/2} Z_\ktoinf[i, *]^\top\|^2.
\]

By Lemma \ref{lm::sum of norms}, for some absolute constant $c$ and any $t \in (0,n)$, with probability at least $1 - 2e^{-ct}$,
\[
(n-\sqrt{nt}\sigma_x^2)\sum_{i>k} \lambda_i \leq \sum_{i=1}^n\|\Sigma_\ktoinf^{1/2} Z_\ktoinf[i, *]^\top\|^2 \leq (n+\sqrt{nt}\sigma_x^2)\sum_{i>k} \lambda_i.
\]

On this event
\begin{align*}
\mu_1(A_k) \geq& \lambda +\left(1-\sqrt{\frac{t}{n}}\sigma_x^2\right)\sum_{i>k} \lambda_i, \\
\mu_n(A_k) \leq& \lambda + \left(1+\sqrt{\frac{t}{n}}\sigma_x^2\right)\sum_{i>k} \lambda_i.
\end{align*}

Finally, note  that $\mu_1(A_k) \geq  \lambda_{k+1} \|Z_\ktoinf[*,1]\|^2 + \lambda.$ By Lemma \ref{lm::sum of norms}, for some $c_3$ and for any $t \in (0,n)$, with probability. at least $1 - 2e^{-c_3t}$,
\[
\|Z_\ktoinf[*,1]\|^2 \geq n - \sqrt{nt}\sigma_x^2,
\]which means that 
\[
\mu_1(A_k) \geq \lambda + n\lambda_{k+1}\left(1 - \sqrt{\frac{t}{n}}\sigma_x^2\right).
\]

Combining all those bounds together gives that there is a constant $c_x$ that only depends on $\sigma_x$ such that with probability at least $1 - c_x e^{-n/c_x}$ all the following inequalities hold simultaneously:
\begin{align*}
\mu_1(A_k) \leq& \lambda + c_x\left(n\lambda_{k+1} + \sum_{i > k}\lambda_i\right),\\
\mu_1(A_k) \geq& \lambda + \frac{1}{c_x} \sum_{i > k}\lambda_i,\\
\mu_1(A_k) \geq& \lambda + \frac{1}{c_x} n\lambda_{k+1},\\
\mu_n(A_k) \geq& \lambda + \min_i \|X_\ktoinf[i, *]\|^2 - c_x\left(n\lambda_{k+1} + \sqrt{n \sum_{i > k}\lambda_i^2}\right),\\
\mu_n(A_k) \leq& \lambda + c_x \sum_{i > k}\lambda_i,\\
\mu_n(A_k) \leq& \lambda + \min_i \|X_\ktoinf[i, *]\|^2.
\end{align*}

In view of the bounds that we derived above, the following inequality is a sufficient condition for the statement that with probability at least $1 - c_x e^{-n/c_x}$ the condition number of $A_k$ does not exceed $L$:
\[
\frac{1}{L}\left(\lambda + c_x\left(n\lambda_{k+1} + \sum_{i > k}\lambda_i\right)\right) \leq\lambda + \min_i \|X_\ktoinf[i, *]\|^2 - c_x\left(n\lambda_{k+1} + \sqrt{n \sum_{i > k}\lambda_i^2}\right).
\]

Note that for any $\zeta > 0$ 
\[
\sqrt{n\sum_{i > k}\lambda_i^2} < 2\sqrt{n\sum_{i > k}\lambda_i^2} \leq 2\sqrt{n\lambda_{k+1}\sum_{i > k}\lambda_i}\leq \zeta n\lambda_{k+1} + \zeta^{-1}\sum_{i > k}\lambda_i,
\]
which implies that for any $\zeta$ the following is also a sufficient condition:
\[
\lambda + \min_i \|X_\ktoinf[i, *]\|^2 \geq \lambda L^{-1} + c_x(1 + L^{-1} + \zeta) n\lambda_{k+1} + c_x(L^{-1} + \zeta^{-1})\sum_{i > k}\lambda_i.
\]

Recall that $\lambda > -\gamma \sum_{i > k}\lambda_i$, so 
\[
\sum_{i > k}\lambda_i \leq \frac{1}{1-\gamma}\left(\lambda + \sum_{i > k}\lambda_i\right),
\]
which allows us to  upper bound the right-hand side of that condition. We write
\begin{align*}
&\lambda L^{-1} + c_x(1 + L^{-1} + \zeta) n\lambda_{k+1} + c_x(L^{-1} + \zeta^{-1})\sum_{i > k}\lambda_i\\
\leq& L^{-1}\left(\lambda + \sum_{i > k}\lambda_i\right) + c_x(1 + L^{-1} + \zeta)\rho_k^{-1}\left(\lambda + \sum_{i > k}\lambda_i\right) +  \frac{c_x(L^{-1} + \zeta^{-1})}{1 - \gamma}\left(\lambda + \sum_{i > k}\lambda_i\right)\\
=& \left(\lambda + \sum_{i > k}\lambda_i\right)\left(L^{-1}\left(1 + c_x\rho_k^{-1} + \frac{c_x}{1-\gamma}\right) + \rho_k^{-1}\left(c_x + c_x\zeta\right) + \frac{c_x\zeta^{-1}}{1-\gamma}\right).
\end{align*}

Now  take $\zeta = \rho_k^{1/2}$ and a constant $c$ that is big enough depending on $\gamma$ and $c_x$. Then if $\rho_k > L^2 > 1$ and with probability at least $1 - \delta$,
\[
\lambda + \min_i \|X_\ktoinf[i, *]\|^2 \geq \frac{c}{L}\left(\lambda + \sum_{i > k}\lambda_i\right),
\]
then with probability at least $1 - \delta  - c_xe^{-n/c_x}$,
\[
\mu_n(A_k)\geq L^{-1}\mu_1(A_k). 
\]

Note that since the rows of $X_\ktoinf$ are i.i.d., the first condition is equivalent to that with probability at least $(1 - \delta)^{1/n}$
\[
\lambda +  \|X_\ktoinf[1, *]\|^2 \geq \frac{c}{L}\left(\lambda + \sum_{i > k}\lambda_i\right).
\]

Now let's derive a necessary condition. Suppose it is known that with probability at least $c_x e^{-n/c_x}$ $\mu_n(A_k) \geq L^{-1}\mu_1(A_k)$. Then
\begin{align*}
\lambda + \min_i \|X_\ktoinf[i, *]\|^2 \geq& \frac{1}{L}\left(\lambda + \frac{1}{c_x} \sum_{i > k}\lambda_i\right),\\
\lambda + c_x\sum_{i > k} \lambda_i \geq& \frac{1}{L}\left(\lambda + \frac{1}{c_x}n\lambda_{k+1}\right).
\end{align*}

For the first equation, we can write
\begin{align*}
\lambda + \min_i \|X_\ktoinf[i, *]\|^2& \geq \frac{1}{L}\left(\lambda + \frac{1}{c_x} \sum_{i > k}\lambda_i\right)\\
\lambda(1 - L^{-1} + L^{-1}c_x^{-1}) + \min_i \|X_\ktoinf[i, *]\|^2& \geq \frac{1}{Lc_x}\left(\lambda + \sum_{i > k}\lambda_i\right),\\
\lambda+ \min_i \|X_\ktoinf[i, *]\|^2& \geq \frac{1}{Lc_x(1 - L^{-1} + L^{-1}c_x^{-1}) }\left(\lambda + \sum_{i > k}\lambda_i\right)\\
&\geq \frac{1}{Lc_x}\left(\lambda + \sum_{i > k}\lambda_i\right),
\end{align*}
where we used the fact that $c_x > 1$ and $L > 1$. 

When it comes to the second equation, we write
\begin{align*}
\lambda + c_x\sum_{i > k} \lambda_i \geq& \frac{1}{L}\left(\lambda + \frac{1}{c_x}n\lambda_{k+1}\right),\\
(L - 1)\lambda + c_x L \sum_{i > k}\lambda_i \geq& \frac{1}{c_x}n\lambda_{k+1} = \frac{1}{c_x}\rho_k^{-1}\left(\lambda + \sum_{i > k}\lambda_i\right),\\
(L-1)\left(\lambda + \sum_{i > k}\lambda_i\right) + (c_x L - L + 1) \sum_{i > k}\lambda_i \geq& \frac{1}{c_x}\rho_k^{-1}\left(\lambda + \sum_{i > k}\lambda_i\right)\\
\left(L-1 + \frac{c_x L - L + 1}{1 - \gamma}\right)\left(\lambda + \sum_{i > k}\lambda_i\right) \geq& \frac{1}{c_x}\rho_k^{-1}\left(\lambda + \sum_{i > k}\lambda_i\right)\\
\rho_k \geq& c_x^{-1}\left(L-1 + \frac{c_x L - L + 1}{1 - \gamma}\right)^{-1} \geq c^{-1}L^{-1},
\end{align*}
where $c$ is a large enough constant that only depends on $\gamma$ and $c_x$. 
\end{proof}

\begin{lemma}
\label{lm::cond number for k}
Suppose assumptions \ref{as::gamma}$(k, \gamma)$ and \ref{as::condition number of A_k}$(k, \delta, L)$ are satisfied and $\gamma < 1$. Then for some absolute constant $c$ for any $t \in (0, n)$ with probability  at least $1 - \delta - 2e^{-ct}$
\[
\frac{1}{L}\left(1 - \frac{\sqrt{t}\sigma_x^2}{\sqrt{n}(1-\gamma)}\right)\left(\lambda + \sum_{i>k} \lambda_i\right) \leq \mu_n(A_k) \leq \mu_1(A_k) \leq  L\left(1 - \frac{\sqrt{t}\sigma_x^2}{\sqrt{n}(1-\gamma)}\right)\left(\lambda + \sum_{i>k } \lambda_i\right).
\]

Moreover, if   $\delta < 1-4e^{-ct}$ for some $t \in (0,n)$, then 
\[
\frac{\lambda + \sum_{i > k} \lambda_i}{n\lambda_{k+1}}  \geq \frac{1 - \sigma_x^2\sqrt{t/n}}{L+\frac{\gamma}{1-\gamma} + \frac{\sqrt{t}\sigma_x^2L }{\sqrt{n}(1-\gamma)}}.
\]

\end{lemma}
\begin{proof}
First of all, note that the sum of eigenvalues of $A_k$ is equal to 
\[
\tr(A_k)=  \lambda n + \sum_{i=1}^n\|\Sigma_\ktoinf^{1/2} Z_\ktoinf[i, *]^\top\|^2.
\]

By Lemma \ref{lm::sum of norms} for some absolute constant $c$ and any $t \in (0,n)$ with probability at least $1 - 2e^{-ct}$
\[
(n-\sqrt{nt}\sigma_x^2)\sum_{i>k} \lambda_i \leq \sum_{i=1}^n\|\Sigma_\ktoinf^{1/2} Z_\ktoinf[i, *]^\top\|^2 \leq (n+\sqrt{nt}\sigma_x^2)\sum_{i>k} \lambda_i.
\]

Now we know that with probability at least $1 - \delta - 2\exp(-c_2t)$ the following two conditions hold:
\begin{gather*}
\mu_1(A_k) \leq L \mu_n(A_k),\\
n\lambda  + (n-\sqrt{nt}\sigma_x^2)\sum_{i > k} \lambda_i \leq \sum_{i=1}^n \mu_i(A_k)  \leq  n\lambda  +(n+\sqrt{nt}\sigma_x^2)\sum_{i> k} \lambda_i.
\end{gather*}

The first line of the display above implies that
\[
n\mu_1(A_k)/L\leq \sum_{i=1}^n \mu_i(A_k)  \leq n\mu_n(A_k)\cdot L
\]

Thus, with probability  at least $1 - \delta - 2\exp(-c_2t)$,
\begin{gather*}
\frac{\lambda}{L} + \frac{n-\sqrt{nt}\sigma_x^2}{nL}\sum_{i > k} \lambda_i \leq \mu_n(A_k) \leq \mu_1(A_k) \leq \lambda L  + \frac{(n+\sqrt{nt}\sigma_x^2)L}{n}\sum_{i> k} \lambda_i,\\
\frac{1}{L}\left(\lambda  + \sum_{i} \lambda_i\right) - \frac{\sqrt{t}\sigma_x^2}{\sqrt{n}L}\sum_{i> k} \lambda_i \leq \mu_n(A_k) \leq \mu_1(A_k) \leq L\left(\lambda  + \sum_{i} \lambda_i\right) + \frac{\sqrt{t}\sigma_x^2L}{\sqrt{n}}\sum_{i> k} \lambda_i.
\end{gather*}
Using the fact that $\sum_{i > k} \lambda_i \leq \left(\lambda  + \sum_{i > k} \lambda_i\right)/(1-\gamma)$, we obtain
\[
\frac{1}{L}\left(\lambda  + \sum_{i> k} \lambda_i\right)\left(1 - \frac{\sqrt{t}\sigma_x^2}{\sqrt{n}(1-\gamma)}\right)\leq \mu_n(A_k) \leq \mu_1(A_k) \leq L \left(\lambda  + \sum_{i> k} \lambda_i\right)\left(1 + \frac{\sqrt{t}\sigma_x^2}{\sqrt{n}(1-\gamma)}\right),
\]
which gives the first assertion of the lemma. 

Next, note that $\mu_1(A_k) \geq  \lambda_{k+1} \|Z_\ktoinf[*,1]\|^2 + \lambda.$ By Lemma \ref{lm::sum of norms} for some $c_3$ for any $t \in (0,n)$ w.p. at least $1 - 2e^{-c_3t}$,  $\|Z_\ktoinf[*,1]\|^2 \geq n - \sqrt{nt}\sigma_x^2$, which means that if $1 - \delta - 2e^{-c_2t}- 2e^{-c_3t} > 0$ then with positive probability 
\begin{align*}
\lambda L  + \frac{(n+\sqrt{nt}\sigma_x^2)L}{n}\sum_{i>k} \lambda_i \geq&  \lambda_{k+1}(n - \sqrt{nt}\sigma_x^2) + \lambda,\\
\lambda (L-1)  + \frac{(n+\sqrt{nt}\sigma_x^2)L}{n}\sum_{i > k} \lambda_i \geq&  \lambda_{k+1}(n - \sqrt{nt}\sigma_x^2), \\
\left(\lambda  + \sum_{i > k} \lambda_i\right)(L-1)   + \left(1 + \frac{\sqrt{t}\sigma_x^2L }{\sqrt{n}}\right)\sum_{i > k} \lambda_i \geq&  \lambda_{k+1}(n - \sqrt{nt}\sigma_x^2), \\
\left(\lambda  + \sum_{i > k} \lambda_i\right)\left(L+\frac{\gamma}{1-\gamma} + \frac{\sqrt{t}\sigma_x^2L }{\sqrt{n}(1-\gamma)}\right)\geq&  \lambda_{k+1}(n - \sqrt{nt}\sigma_x^2). \\
\end{align*}

Taking $c_4 = \min(c_2, c_3)$ we see that if $\delta < 1-4e^{-c_4t}$, then 
\[
\frac{\lambda + \sum_{i > k} \lambda_i}{n\lambda_{k+1}}  \geq \frac{1 - \sigma_x^2\sqrt{t/n}}{L+\frac{\gamma}{1-\gamma} + \frac{\sqrt{t}\sigma_x^2L }{\sqrt{n}(1-\gamma)}}.
\]

\end{proof}

\kandkstar*
%
\begin{proof}
First, by Lemma \ref{lm::cond number for k}
 for any $t \in (0, n)$ with probability  at least $1 - \delta - 2e^{-c_1t}$,
\[
\frac{1}{L}\left(1 - \frac{\sqrt{t}\sigma_x^2}{\sqrt{n}(1-\gamma)}\right)\left(\lambda + \sum_{i > k} \lambda_i\right) \leq \mu_n(A_k) \leq \mu_n(A_{k^*}).
\]

Next, by Lemma \ref{lm::top eigenvalue} we know that with probability at least $1 - 6e^{-t/c_3}$,
\[
\mu_1(A_{k^*})\leq  c_3\sigma_x^2\left(\lambda_{k^* + 1}(t + n) + \sum_{i> k^*} \lambda_i\right) + \lambda.
\]

By definition of $k^*$ and $\rho_k$
\[
\lambda_{k^* + 1}n = \rho_{k^*}^{-1}\left(\lambda + \sum_{i > k^*} \lambda_i\right) \leq b^{-1}\left(\lambda + \sum_{i > k^*} \lambda_i\right).
\]
 Therefore, 
 \[
 \lambda + \sum_{i > k} \lambda_i = \lambda + \sum_{i > k^*} \lambda_i - \sum_{i = k^* + 1}^k \lambda_i \geq \lambda + \sum_{i > k^*} \lambda_i - n\lambda_{k^* + 1} \geq (1 - b^{-1})\left(\lambda + \sum_{i > k^*} \lambda_i\right).
 \]
 
 Moreover, since $\lambda > -\gamma \sum_{i > k^*} \lambda_i$,
 \begin{align*}
 \lambda \leq& \lambda + \sum_{i > k^*}\lambda_i,\\
 \sum_{i > k^*} \lambda_i \leq& \frac{1}{1 - \gamma}\left(\lambda + \sum_{i > k^*}\lambda_i\right)\\
 \lambda_{k^* + 1}(t + n) \leq& b^{-1}(1 + t/n)\left(\lambda + \sum_{i > k^*}\lambda_i\right).
 \end{align*}
Thus, with probability at least $1 - \delta - 8e^{-t/c_4}$
\begin{align*}
\mu_n(A_{k^*}) \geq& \frac{1}{L}\left(1 - \frac{\sqrt{t}\sigma_x^2}{\sqrt{n}(1-\gamma)}\right)(1 - b^{-1})\left(\lambda + \sum_{i > k^*}\lambda_i\right),\\
\mu_1(A_{k^*}) \leq& \left(c_3 \sigma_x^2\left(\frac{1}{1-\gamma} + \frac{1}{b}\left(1 + \frac{t}{n}\right)\right) + 1\right)\left(\lambda + \sum_{i > k^*}\lambda_i\right).
\end{align*}

Taking $c_5$ large enough (depending on $L$, $b$, $\sigma_x$ and $\gamma$) and plugging in $t = n/c_5$ gives the result for $c = \max(8, c_4c_5)$ and 
\[
L' = \left(c_3 \sigma_x^2\left(\frac{1}{1-\gamma} + \frac{1}{b}\left(1 + c_5^{-1}\right)\right) + 1\right)\div\left(\frac{1}{L}\left(1 - \frac{\sigma_x^2}{\sqrt{c_5}(1-\gamma)}\right)(1 - b^{-1})\right).
\]

The derivation of \ref{as::gamma}($k^*, \gamma$) is obvious: indeed, assumption \ref{as::gamma}($k, \gamma$) states that
\[
 \lambda > -\gamma \sum_{i > k} \lambda_i.
\]
Since $k^* \geq k$, $\sum_{i > k} \lambda_i \leq \sum_{i > k^*} \lambda_i$, so
\[
 \lambda > -\gamma \sum_{i > k^*} \lambda_i,
\]
which is exactly assumption  \ref{as::gamma}($k^*, \gamma$).
\end{proof}

\section{Lower bounds}
\label{sec::lower}

A very convenient tool that we use to prove the lower bounds is the following 
\begin{lemma}[Lemma 9 from \citep{benign_overfitting}]
\label{lm::sum of non-neg lower bound whp}
Suppose that  $\{\eta_i\}_{i=1}^p$ is a sequence of non-negative random variables, and that $\{t_i\}_{i=1}^p$ is a sequence of non-negative real numbers (at least one of which is strictly positive) such that, for some $\delta \in (0,1)$ for any $i \leq p$ with probability at least $1-\delta$, $\eta_i > t_i$. Then with probability at least $1-2\delta$,
\[
\sum_{i=1}^n \eta_i \geq \frac12 \sum_{i=1}^p t_i.
\]
\end{lemma}

It turns out to be quite straightforward to express bias and variance terms as sums of non-negative series. This lemma allows us to give a separate high probability lower bound for each term in the series to obtain the high probability lower bound for the whole sum.

\subsection{Variance term}
\label{sec::variance lower appendix}
The argument for lower bounding the variance term is the same as in \citep{benign_overfitting}. We repeat it here because the result in \citep{benign_overfitting} is stated in a different form and in the ridgeless setting only.

\varlower*

\begin{proof}
The variance term can be written as
\[
V = \tr\left(\Sigma X^\top A^{-2} X\right) = \sum_{i=1}^\infty \frac{\lambda_i^2 z_i^\top A_{-i}^{-2}z_i}{(1 + \lambda_i z_i^\top A_{-i}^{-1}z_i)^2},
\]
where $z_i$ are columns of matrix $Z$ (recall that $Z = X\Sigma^{-1/2}$). Note that every term in this sum is non-negative, even if $A_{-i}$ is not PSD. Denote  $A_{-i+}$ to be the PSD square root of $A_{-i}^2$, i.e., the matrix with the same eigendecomposition as $A_{-i}$, but with eigenvalues substituted by their absolute values. It immediately follows that
\[
V \geq \sum_{i=1}^\infty \frac{\lambda_i^2 z_i^\top A_{-i}^{-2}z_i}{(1 + \lambda_i z_i^\top A_{-i+}^{-1}z_i)^2},
\]

By Cauchy-Schwartz we have
\[
\|z_i\|^2\cdot z_i^\top A_{-i}^{-2}z_i \geq (z_i^\top A_{-i+}^{-1}z_i)^2. 
\]

Thus, 
\[
V \geq  \sum_{i=1}^\infty\frac{1}{\|z_i\|^2\left(1 + (\lambda_i z_i^\top A_{-i+}^{-1}z_i)^{-1}\right)^2}.
\]

Now our goal is to lower-bound the largest eigenvalues of $A_{-i+}^{-1}$. Let's write
\[
A_{-i} = \lambda I_n + \sum_{j \neq i} \lambda_j z_j z_j^\top.
\]
The idea is, as always, to separate the first $k$ coordinates. Our initial goal is to bound the norm of $\sum_{j \neq i, j > k} \lambda_j z_j z_j^\top$. Using Lemma \ref{lm::top eigenvalue}, for some absolute constant $c_1$ and for any $t > 0$, with probability at least $1 - 6e^{-t/c_1}$,
\[
\left\|\sum_{j \neq i, j > k} \lambda_j z_j z_j^\top\right\|\leq \left\|\sum_{ j > k} \lambda_j z_j z_j^\top\right\| = \|X_\ktoinf X_\ktoinf^\top\|\leq c_1\sigma_x^2\left(\lambda_{k+1}(t+n) + \sum_{i > k} \lambda_i\right)  
\]

The matrix $\sum_{j \neq i} \lambda_j z_j z_j^\top$ is a correction to $\sum_{j \neq i, j > k} \lambda_j z_j z_j^\top$ of rank at most $k$. Therefore, with probability at least $1 - 6e^{-t/c_1}$ the bottom $k$ eigenvalues of $\sum_{j \neq i} \lambda_j z_j z_j^\top$ lie in the segment from $0$ to $c_1\sigma_x^2\left(\lambda_{k+1}(t+n) + \sum_{i > k} \lambda_i\right)$. The matrix $A_{-i}$ has the same eigenvalues, but with $\lambda$ added to each one, so on the same event all the eigenvalues of $A_{-i}$ are from $\lambda$ to $\lambda + c_1\sigma_x^2\left(\lambda_{k+1}(t+n) + \sum_{i > k} \lambda_i\right)$. We can write
\begin{align*}
&c_1\sigma_x^2\left(\lambda_{k+1}(t+n) + \sum_{i > k} \lambda_i\right) + \lambda\\
\leq&  c_1\sigma_x^2\left(\lambda_{k+1}(t+n) + \frac{1}{1-\gamma}\left(\lambda + \sum_{i > k} \lambda_i\right)\right)  + \frac{\gamma}{1-\gamma}\left(\lambda + \sum_{i > k} \lambda_i\right),
\end{align*}
where we used that $\lambda > -\gamma \sum_{i > k} \lambda_i$ in the second line (for $\lambda < 0$ it implies $|\lambda| < \gamma \sum_{i > k} \lambda_i$). Moreover, for the left end of the segment we also have that either $\lambda > 0$ or 
\[
|\lambda| \leq \gamma \sum_{i > k}\lambda_i \leq \frac{\gamma}{1-\gamma}\left(\lambda + \sum_{i > k} \lambda_i\right).
\]

Thus, for some constant $c_2$ which only depends on $\sigma$ and $\gamma$, for any $i$ with probability at least $1 - 6e^{-n/c_2}$, for any $j > k$
\[
|\mu_j(A_i)| \leq c_2\left(\lambda_{k+1}n + \lambda + \sum_{i > k} \lambda_i\right).
\]

In words, with high probability the matrix $A_{-i}$ has at least $n-k$ eigenvalues whose magnitude is bounded by $c_2\left(\lambda_{k+1}n + \lambda + \sum_{i > k} \lambda_i\right)$. Recall that $A_{-i+}$ is PSD with the same magnitudes of the eienvalues. Denote $P_{i,k}$ to be the projector on the linear space spanned by the first $k$ eigenvectors of $A_{-i+}$. We can now write that with probability at least $1 - 6e^{-n/c_2}$
\[
z_i^\top A_{-i+}^{-1}z_i \geq  \|(I - P_{i,k})z_i\|^{2} c_2^{-1}\left(\lambda_{k+1}n + \lambda + \sum_{i > k} \lambda_i\right)^{-1}
\]

Since $z_i$ is independent of $P_{i,k}$, by Theorem 6.2.1 (Hanson-Wright inequality) in \citep{vershynin_hdp},  for some absolute constant $c_2$ and for any $t > 0$,
\[
\Pbb\left\{\left|\|P_{i,k}z_i\|^{2}  - \E_{z_i}\|P_{i,k}z_i\|^{2}\right| \geq t\right\} \leq 2\exp\left(-c_2^{-1}\min\left\{\frac{t^2}{\sigma_x^4\|P_{i,k}^2\|_F^2 }, \frac{t}{\sigma_x^2 \|P_{i,k}^2\|}\right\}\right).
\]

Since $P_{i,k}$ is an orthogonal projector of rank $k$, $\|P_{i,k}^2\|_F^2 = k,$ $\|P_{i,k}^2\| = 1$, and $\E_{z_i}\|P_{i,k}z_i\|^{2} = \tr(P_{i,k}) = k.$ Thus, w.p.\ at least $ 1-2e^{-t/c_2}$,
\[
\left|\|P_{i,k}z_i\|^{2}  - k\right| \leq \sigma_x^2\max(\sqrt{kt}, t)\leq (t + \sqrt{kt})\sigma_x^2.
\]

Next, by Lemma \ref{lm::sum of norms} for some constant $c_3$ and any $t \in (0,n)$ w.p. at least $1-2e^{-t/c_3}$,
\[
n - \sqrt{nt}\sigma_x^2 \leq \|z_i\|^2 \leq n + \sqrt{nt}\sigma_x^2.
\]

Take constant $c_4$ large enough depending on $\sigma_x$ and set $t = n/c_4$. Then for any $k < n/c_5$, w.p.\ at least $1 - 10 e^{-n/c_6} - \delta$,
\[
z_i^\top A_{-i+}^{-1}z_i \geq \frac{n }{c_7\left(\lambda_{k+1}n + \lambda + \sum_{i > k} \lambda_i\right)},
\]
where constants $c_5$ and $c_6$ depend only on $\sigma_x$ and constant $c_7$ depends only on $\sigma_x$ and $\gamma$.

Rewrite this equation as 
\begin{equation*}
(z_i^\top A_{-i+}^{-1}z_i)^{-1} \leq  c_7\left(\lambda_{k+1} + \frac{1}{n}\left(\lambda + \sum_{i > k} \lambda_i\right)\right) = c_7\lambda_{k+1}(\rho_k + 1),
\end{equation*}
where $\rho_k := \frac{1}{n\lambda_{k+1}}\left(\lambda + \sum_{i > k} \lambda_i\right).$

On the same event
\[
\frac{1}{\|z_i\|^2\left(1 + (\lambda_i z_i^\top A_{-i+}^{-1}z_i)^{-1}\right)^2} \geq \frac{1}{c_8n\left(1 + \frac{\lambda_{k+1}}{\lambda_i}(\rho_k + 1)\right)^2},
\]
where $c_8$ depends only on $\sigma_x$ and $\gamma$.

Finally, by Lemma \ref{lm::sum of non-neg lower bound whp}, we can convert lower bounds for separate non-negative terms into a lower bound on their sum: with probability at least $1 - 20 e^{-n/c_6} $,
\[
V \geq \frac{1}{8c_8n}\sum_{i = 1}^p\min\left\{1, \frac{\lambda_i^2}{ \lambda_{k+1}^2(\rho_k + 1)^2}\right\},
\]
where we also used that $1/(a + b)^2 \geq \min(a^{-2}, b^{-2})/4$ for non-negative $a, b$.
\end{proof}

\subsection{Bias term}
\label{sec::bias lower appendix}
\biaslower*
\begin{proof}

Applying Sherman-Morrison-Woodbury yields
\[
\left( \lambda I_p + X^\top X\right)^{-1} = \lambda^{-1}I_p   - \lambda^{-2}X^\top(I_n + \lambda^{-1}XX^\top)^{-1}X.
\]
So, 
\begin{align*}
\left(\lambda I_p + X^\top X\right)^{-1}X^\top X - I_p
=& \left(\lambda I_p + X^\top X\right)^{-1}(\lambda I_p + X^\top X - \lambda I_p) - I_p\\
=& -\lambda\left(\lambda I_p + X^\top X\right)^{-1}\\
=& I_p   - \lambda^{-1}X^\top(I_n + \lambda^{-1}XX^\top)^{-1}X\\
=& I_p   - X^\top(\lambda I_n + XX^\top)^{-1}X.
\end{align*}

Thus, the bias term becomes 
\[
(\theta^*)^\top\left(I_p   - X^\top(\lambda I_n + XX^\top)^{-1}X\right)\Sigma\left(I_p   - X^\top(\lambda I_n + XX^\top)^{-1}X\right)\theta_*\,
\]
and taking expectation over the prior kills all the off-diagonal elements, so
\[
\E_{\theta^*} \mathcal{B} = \sum_i \left(\left(I_p   - X^\top(\lambda I_n + XX^\top)^{-1}X\right)\Sigma\left(I_p   - X^\top(\lambda I_n + XX^\top)^{-1}X\right)\right)[i,i]\cdot\bar{\theta}_i^2.
\]
Let's compute the diagonal elements of the matrix 
\[\left(I_p   - X^\top(\lambda I_n + XX^\top)^{-1}X\right)\Sigma\left(I_p   - X^\top(\lambda I_n + XX^\top)^{-1}X\right).
\]

The $i$-th diagonal element is equal to the bias term for the case when $\theta^* = e_i$ --- the $i$-th vector of the standard orthonormal basis. Note that the $i$-th row of $I_p   - X^\top(\lambda I_n + XX^\top)^{-1}X$ is equal to $e_i - \sqrt{\lambda_i}z_i^\top (\lambda I_n + XX^\top)^{-1}X,$ so the $i$-th diagonal element of the initial matrix is given by 
\begin{align*}
&\sum_{j = 1}^p \lambda_i \left(e_i[j] - \sqrt{\lambda_i\lambda_j}z_i^\top (\lambda I_n + XX^\top)^{-1}z_j\right)^2\\
&\lambda_i\left(1 - \lambda_i z_i^\top A^{-1}z_i\right)^2 + \sum_{j \neq i} \lambda_i\lambda_j^2(z_i^\top A^{-1}z_j)^2.
\end{align*}

 Recall that $A = \lambda I_n + \sum_{i = 0}^p \lambda_i z_i z_i^\top$, $A_{-i} := A - \lambda_i z_iz_i^\top.$ 

First, let's use Sherman-Morrison identity to convert $A$ in $z_i^\top A^{-1} z_i$ into $A_{-i}$:
\begin{align*}
    1 - \lambda_i z_i^\top A^{-1}z_i =& 1 - \lambda_i z_i^\top \left(A_{-i} + \lambda_i z_iz_i^\top\right)^{-1}z_i\\
	=&1 - \lambda_i z_i^\top \left(A_{-i}^{-1} - \lambda_iA_{-i}^{-1}z_i(1 + z_i^\top A_{-i}^{-1}z_i)^{-1}z_i^\top A_{-i}^{-1}\right)z_i\\
	=& 1 - \lambda_iz_i^\top A_{-i}^{-1}z_i + \frac{\left(\lambda_iz_i^\top A_{-i}^{-1}z_i\right)^2}{1 + \lambda_iz_i^\top A_{-i}^{-1}z_i}\\
	=&\frac{1}{1 + \lambda_iz_i^\top A_{-i}^{-1}z_i}.
	\end{align*}
  So the diagonal element becomes 
		\[
		\frac{\lambda_i}{(1 + \lambda_iz_i^\top A_{-i}^{-1}z_i)^2} +\sum_{j \neq i} \lambda_i\lambda_j^2(z_i^\top A^{-1}z_j)^2\geq \frac{\lambda_i}{(1 + \lambda_iz_i^\top A_{-i}^{-1}z_i)^2},
		\]
and thus
\[
\E_{\theta^*} {B} \geq \sum_i  \frac{\lambda_i\bar{\theta}^2_i}{(1 + \lambda_iz_i^\top A_{-i}^{-1}z_i)^2}.
\]
Let's bound each term in that sum from below with high probability. By our assumptions, for any $i$ with probability at least $1-\delta$
\[
 \mu_n(A_{-i}) \geq \frac{1}{L}\left(\lambda + \sum_{j > k} \lambda_j\right).
\]
Next, 
\[
\frac{\lambda_i}{(1 + \lambda_iz_i^\top A_{-i}^{-1}z_i)^2} \geq \frac{\lambda_i}{\left(1 + \lambda_i\mu_n(A_{-i})^{-1}\|z_i\|^2\right)^2},
\]
and by Lemma \ref{lm::sum of norms} for some absolute constant $c_1$ for any $t \in (0,n)$ w.p. at least $1-2e^{-t/c_1}$ we have $\|z_i\|^2 \leq n - \sqrt{tn}\sigma_x^2 \leq n/2$, where the last transition is true if additionally $t \leq n/(4\sigma_x^4).$

Recall that $\rho_k := \frac{\lambda + \sum_{j > k} \lambda_j}{n\lambda_{k+1}}.$ We obtain by plugging $t = n/(4\sigma_x^4)$ that w.p. at least $1 - \delta - 2e^{-n/c_2}$,
\[
 \frac{\lambda_i\bar{\theta}^2_i}{(1 + \lambda_iz_i^\top A_{-i}^{-1}z_i)^2} \geq  \frac{\lambda_i\bar{\theta}^2_i}{\left(1 + \frac{L\lambda_i}{2\lambda_{k+1}\rho_k}\right)^2},
\]
where $c_2$ only depends on $\sigma_x$.

Finally, since all the terms are non-negative and we need to obtain a lower bound on their sum, Lemma \ref{lm::sum of non-neg lower bound whp} gives the result.

\end{proof}
\ExchCoordPlusCondNum*
\begin{proof}
First of all, note that Assumption \ref{as::gamma}$(k, \gamma)$ with $\gamma < 1$ directly implies that $\lambda + \sum_{i > k}\lambda_i \geq 0$, which is the second part of Assumption \ref{as::lowest eigenvalue of A_-k}$(k, \delta, L)$.

Next, by Lemma \ref{lm::cond number for k} for some absolute constant $c_1$ for any $t \in (0, n)$ with probability  at least $1 - \delta - 2e^{-ct}$
\[
\mu_n(A_k) \geq \frac{1}{L}\left(1 - \frac{\sqrt{t}\sigma_x^2}{\sqrt{n}(1-\gamma)}\right)\left(\lambda + \sum_{i>k} \lambda_i\right).
\]
Taking $t = n/c_2$ where $c_2$ is large enough depending on $\gamma, \sigma_x$ we get that for $c$ large enough with probability at least $1 - 2e^{-n/c}$ 
\[
\mu_n(A_k) \geq \frac{1}{cL}\left(\lambda + \sum_{i>k} \lambda_i\right).
\]
Now we just need to propagate that result to $A_{-i}$ for all $i$. 

For $i \leq k$, we simply have $A_{-i} \succeq A_k$ with probability $1$, so indeed $\forall i \leq k$
\[
\Pbb\left(\mu_n(A_{-i}) \geq \frac{1}{cL}\left(\lambda + \sum_{i>k} \lambda_i\right)\right) \geq \Pbb\left(\mu_n(A_{k}) \geq \frac{1}{cL}\left(\lambda + \sum_{i>k} \lambda_i\right)\right) \geq 1 - 2e^{-n/c}.
\]

When it comes to $i > k$, we can write
\begin{align*}
A_{-i} =& \lambda I_n +  \sum_{j \neq i} \lambda_j z_j z_j^\top\\
=& \lambda I_n + \sum_{j\leq k} \lambda_j z_j z_j^\top + \sum_{j>k, j \neq i} \lambda_j z_j z_j^\top\\
\succeq& \lambda I_n + \lambda_{1}z_1z_1^\top + \sum_{j>k, j \neq i} \lambda_j z_j z_j^\top\\ 
\succeq& \lambda I_n + \lambda_{i}z_1z_1^\top + \sum_{j>k, j \neq i} \lambda_j z_j z_j^\top.
\end{align*}
Now note that due to Assumption \ref{as::exchangeable components}, the distribution of the matrix $\lambda I_n + \lambda_{i}z_1z_1^\top + \sum_{j>k, j \neq i} \lambda_j z_j z_j^\top$ is the same as the distribution of $A_k = \lambda I_n +\sum_{j>k} \lambda_j z_j z_j^\top$. Therefore
\begin{align*}
&\Pbb\left(\mu_n(A_{-i}) \geq \frac{1}{cL}\left(\lambda + \sum_{i>k} \lambda_i\right)\right)\\
\geq& \Pbb\left(\mu_n\left( \lambda I_n + \lambda_{i}z_1z_1^\top + \sum_{j>k, j \neq i} \lambda_j z_j z_j^\top\right) \geq \frac{1}{cL}\left(\lambda + \sum_{i>k} \lambda_i\right)\right)\\
=&\Pbb\left(\mu_n(A_k) \geq \frac{1}{cL}\left(\lambda + \sum_{i>k} \lambda_i\right)\right)\\
\geq& 1 - 2e^{-n/c},
\end{align*}
which finishes the proof.
\end{proof}

\section{Deriving a useful identity}
\label{sec::identity}

Motivated by the results of \cite{benign_overfitting}, we split the principal directions of the covariance matrix into two parts: small dimensional and high dimensional. The main idea of our argument is to use classical machinery (like some sort of uniform convergence argument) in the small dimensional subspace. To do this we write $\hat{\theta}( y)^\top = \bigl[\hat{\theta}( y)_\uptok^\top, \hat{\theta}( y)_\ktoinf^\top\bigr]$ and mentally split the search process for $\hat{\theta}(y)$ into two parts: first, for any fixed $\theta_\uptok$, optimize for $\theta_\ktoinf$. Then only the first $k$ coordinates are left. The result of that optimization in $\theta_\ktoinf$ is the following identity:
\begin{equation}
\label{eq::identity}
\hat{\theta}( y)_\uptok + X_\uptok^\top A_k^{-1} X_\uptok\hat{\theta}( y)_\uptok = X_\uptok^\top A_k^{-1}y.
\end{equation}
The goal of this section is to derive this identity.

\subsection{Derivation in the ridgeless case}
In the ridgeless case we are simply dealing with projections, and $\hat{\theta}(y)$ is the minimum norm interpolating solution. Note that $\hat{\theta}(y)_\ktoinf$ is also the minimum norm solution to the equation $X_\ktoinf\theta_\ktoinf = y - X_\uptok \hat{\theta}(y)_\uptok$, where $\theta_\ktoinf$ is the variable. Thus, we can write
\[
\hat{\theta}(y)_\ktoinf = X_\ktoinf^\top \left(X_\ktoinf X_\ktoinf^\top\right)^{-1}\left(y - X_\uptok \hat{\theta}(y)_\uptok\right).
\]

Now we need to minimize the norm in $\hat{\theta}(y)_\uptok$ (our choice of $\hat{\theta}(y)_\ktoinf$ already makes the solution interpolating): we need to minimize the norm of the following vector:
\[
v(\theta_\uptok) = \Bigl[\theta_\uptok^\top, \left(y - X_\uptok \theta_\uptok\right)^\top\left(X_\ktoinf X_\ktoinf^\top\right)^{-1}X_\ktoinf\Bigr]
\]

As $\theta_\uptok$ varies, this vector sweeps an affine subspace of our Hilbert space. The vector $\hat{\theta}(y)_\uptok$ gives the minimum norm if and only if for any additional vector $\eta_\uptok$ we have $v(\hat{\theta}(y)_\uptok) \perp v(\hat{\theta}(y)_\uptok + \eta_\uptok) - v(\hat{\theta}(y)_\uptok)$. Let's write out the second vector: $\forall \eta_\uptok \in \R^k$
\[
 v(\hat{\theta}(y)_\uptok + \eta_\uptok) - v(\hat{\theta}(y)_\uptok) = 
\Bigl[\eta_\uptok^\top,  - \eta_\uptok^\top X_\uptok^\top\left(X_\ktoinf X_\ktoinf^\top\right)^{-1}X_\ktoinf\Bigr]
\]
We see that the above mentioned orthogonality for any $\eta_\uptok$ is equivalent to the following:
\begin{align*}
\hat{\theta}(y)_\uptok^\top - \left(y - X_\uptok \hat{\theta}(y)_\uptok\right)^\top\left(X_\ktoinf X_\ktoinf^\top\right)^{-1}X_\uptok &= 0, \\
\hat{\theta}(y)_\uptok + X_\uptok^\top A_k^{-1} X_\uptok\hat{\theta}(y)_\uptok &= X_\uptok^\top A_k^{-1}y, \\
\end{align*}
where we replaced $X_\ktoinf X_\ktoinf^\top = :A_k$.

\subsection{Checking for the case of non-vanishing regularization}

So, now we have $\lambda \neq 0$ and we want to prove that $\hat{\theta}(y)_\uptok + X_\uptok^\top A_k^{-1} X_\uptok\hat{\theta}(y)_\uptok = X_\uptok^\top A_k^{-1}y$. Recall that
\begin{align*}
\hat{\theta}(y) =& X^\top(\lambda I_n + XX^\top)^{-1}y,\\
\hat{\theta}(y)_\uptok =& X_\uptok^\top(A_k + X_\uptok X_\uptok^\top)^{-1}y.
\end{align*}
This identity yields
\begin{align*}
&\hat{\theta}(y)_\uptok + X_\uptok^\top A_k^{-1} X_\uptok\hat{\theta}(y)_\uptok\\
=& X_\uptok^\top(A_k + X_\uptok X_\uptok^\top)^{-1}y + X_\uptok^\top A_k^{-1} X_\uptok X_\uptok^\top(A_k + X_\uptok X_\uptok^\top)^{-1}y\\
=&X_\uptok^\top A_k^{-1}( A_k + X_\uptok X_\uptok^\top)(A_k + X_\uptok X_\uptok^\top)^{-1}y\\
=&X_\uptok^\top A_k^{-1}y.
\end{align*}

\section{Variance}
\label{sec:: variance upper}
Recall that the variance term is 
\[
V = \frac{1}{v_\eps^2}\E_\eps\|\hat{\theta}( \eps)\|_\Sigma^2 = \frac{1}{v_\eps^{2}}\E_\eps\|X^\top(\lambda I_n + XX^\top)^{-1}\eps\|_\Sigma^2.
\]
In this section we prove the following lemma.
\begin{lemma}
\label{lm::variance}
If for some $k < n$ the matrix $A_k$ is PD, then
\[
V \leq\frac{\mu_1(A_k^{-1})^2\tr( X_\uptok\Sigma_\uptok^{-1} X_\uptok^\top)}{\mu_n(A_k^{-1})^2 \mu_k\left(\Sigma_\uptok^{-1/2}X_\uptok^\top X_\uptok \Sigma_\uptok^{-1/2}\right)^2} +  \mu_1(A_k^{-1})^2\tr(X_\ktoinf\Sigma_\ktoinf X_\ktoinf^\top).
\]
\end{lemma}

Note that the RHS of the inequality above is  straightforward to estimate if one knows the spectrum of $A_k$. Indeed, the matrices $X_\uptok\Sigma_\uptok^{-1} X_\uptok^\top$ and $X_\ktoinf\Sigma_\ktoinf X_\ktoinf^\top$ have i.i.d. elements on their diagonals, so their traces concentrate around expectations: 
\[
\tr( X_\uptok\Sigma_\uptok^{-1} X_\uptok^\top) \sim kn\text{ and }\tr(X_\ktoinf\Sigma_\ktoinf X_\ktoinf^\top) \sim n\sum_{i > k}\lambda_i^2,
\]
where we use $\sim$ informally to denote approximate equality with high probability.

When it comes to the matrix $\Sigma_\uptok^{-1/2}X_\uptok^\top X_\uptok \Sigma_\uptok^{-1/2}/n$, this is just a sample covariance matrix of $n$ isotropic vectors in $k$-dimensional space. Since $k$ is small compared to $n$, it concentrates around the identity. Thus, \[
\mu_k\left(\Sigma_\uptok^{-1/2}X_\uptok^\top X_\uptok \Sigma_\uptok^{-1/2}\right)\sim n.
\]
These computations are done rigorously in the proof of Theorem \ref{th::main upper main body}.

\subsection{First $k$ components}

It was shown in Section \ref{sec::identity} that the following identity holds (c.f. \eqref{eq::identity}):

\[
X_\uptok^\top A_k^{-1}\eps = \hat{\theta}( \eps)_\uptok + X_\uptok^\top A_k^{-1}X_\uptok \hat{\theta}( \eps)_\uptok.
\]
Multiplying the identity by $\hat{\theta}( \eps)_\uptok^\top$ from the left, and using that $\hat{\theta}( \eps)_\uptok^\top \hat{\theta}( \eps)_\uptok \geq 0$ we get
\begin{equation}\label{eqn:firstthetahatineq}
\hat{\theta}( \eps)_\uptok^\top X_\uptok^\top A_k^{-1}\eps 
 \geq \hat{\theta}( \eps)_\uptok^\top X_\uptok^\top A_k^{-1}X_\uptok \hat{\theta}( \eps)_\uptok.
\end{equation}

The leftmost expression is linear in $\hat{\theta}( \eps)_\uptok$, and the rightmost is quadratic. We use these expressions to bound  $\|\hat{\theta}( \eps)_\uptok\|_{\Sigma_\uptok}$.

First, we extract that norm from the quadratic part
\begin{align*}
\hat{\theta}( \eps)_\uptok^\top X_\uptok^\top A_k^{-1}X_\uptok \hat{\theta}( \eps)_\uptok
\geq& \mu_n(A_k^{-1})\hat{\theta}( \eps)_\uptok^\top X_\uptok^\top X_\uptok \hat{\theta}( \eps)_\uptok \\
\geq& \mu_n(A_k^{-1}) \|\hat{\theta}(\eps)_\uptok\|_{\Sigma_\uptok}^2\mu_k\left(\Sigma_\uptok^{-1/2}X_\uptok^\top X_\uptok \Sigma_\uptok^{-1/2}\right).
\end{align*}
Then we can substitute~\eqref{eqn:firstthetahatineq} and apply Cauchy-Schwarz to obtain
\begin{align*}
\|\hat{\theta}( \eps)_\uptok\|_{\Sigma_\uptok}^2\mu_n(A_k^{-1})\mu_k\left(\Sigma_\uptok^{-1/2}X_\uptok^\top X_\uptok \Sigma_\uptok^{-1/2}\right)
& \leq \hat{\theta}( \eps)_\uptok^\top X_\uptok^\top A_k^{-1}X_\uptok \hat{\theta}( \eps)_\uptok \\
& \leq \hat{\theta}( \eps)_\uptok^\top X_\uptok^\top A_k^{-1}\eps \\
& \leq  \|\hat{\theta}( \eps)_\uptok\|_{\Sigma_\uptok}\left\|\Sigma^{-1/2}_{\uptok} X_\uptok^\top A_k^{-1}\eps\right\|,
\end{align*}
and so
\[
\|\hat{\theta}( \eps)_\uptok\|_{\Sigma_\uptok}^2 \leq \frac{\eps^\top A_k^{-1}X_\uptok\Sigma_\uptok^{-1} X_\uptok^\top A_k^{-1}\eps}{\mu_n(A_k^{-1})^2 \mu_k\left(\Sigma_\uptok^{-1/2}X_\uptok^\top X_\uptok \Sigma_\uptok^{-1/2}\right)^2}.
\]
Since $\eps$ is independent of $X$, taking expectation in $\eps$ only leaves the trace in the numerator:
\begin{align*}
\frac{1}{v_\eps^2}\E_\eps\|\hat{\theta}( \eps)_\uptok\|_{\Sigma_\uptok}^2 \leq& \frac{\tr( A_k^{-1}X_\uptok\Sigma_\uptok^{-1} X_\uptok^\top A_k^{-1})}{\mu_n(A_k^{-1})^2 \mu_k\left(\Sigma_\uptok^{-1/2}X_\uptok^\top X_\uptok \Sigma_\uptok^{-1/2}\right)^2}\\
\leq& \frac{\mu_1(A_k^{-1})^2\tr( X_\uptok\Sigma_\uptok^{-1} X_\uptok^\top)}{\mu_n(A_k^{-1})^2 \mu_k\left(\Sigma_\uptok^{-1/2}X_\uptok^\top X_\uptok \Sigma_\uptok^{-1/2}\right)^2},
\end{align*}
where we transitioned to  the second line by using the fact that $\tr(MM'M)\leq \mu_1(M)^2\tr(M')$ for PD matrices $M, M'$.
\subsection{Components starting from $k+1$-st}

The rest of the variance term is 
\[
\left\|\Sigma_\ktoinf^{1/2}X_\ktoinf^\top A^{-1}\eps\right\|^2 = \eps^\top A^{-1}X_\ktoinf\Sigma_\ktoinf X_\ktoinf^\top A^{-1}\eps.
\]
Since $\eps$ is independent of $X$, taking expectation in $\eps$ only leaves the trace of the matrix:
\begin{align*}
\frac{1}{v_\eps^2}\E_\eps\left\|\Sigma_\ktoinf^{1/2}X_\ktoinf^\top A^{-1}\eps\right\|^2 =& \tr( A^{-1}X_\ktoinf\Sigma_\ktoinf X_\ktoinf^\top A^{-1})\\
\leq& \mu_1(A^{-1})^2\tr(X_\ktoinf\Sigma_\ktoinf X_\ktoinf^\top)\\
\leq& \mu_1(A_k^{-1})^2\tr(X_\ktoinf\Sigma_\ktoinf X_\ktoinf^\top).\\
\end{align*}
Here we again used the fact that $\tr(MM'M)\leq \mu_1(M)^2\tr(M')$ for PD matrices $M, M'$ to transition to the second line. We then used $A \succeq A_k$ to infer $\mu_1(A^{-1}) \leq \mu_1(A_k^{-1})$. 

\section{Bias}
\label{sec:: bias upper}
The bias term is given by $\|\theta^* - \hat{\theta}( X\theta^*)\|_\Sigma^2$. In this section we prove the following
\begin{lemma}[Bias term]
\label{lm::bias}
Suppose that for some $k < n$ the matrix $A_k$ is PD. Then there exists an absolute constant $c$ such that
\begin{align*}
&\lefteqn{\|\theta^* - \hat{\theta}( X\theta^*)\|_\Sigma^2/c} \\
& \leq \|\theta^*_\ktoinf\|_{\Sigma_\ktoinf}^2+ \frac{\mu_1(A_k^{-1})^2}{\mu_n(A_k^{-1})^2}\frac{\mu_1\left(\Sigma^{-1/2}_\uptok X_\uptok^\top X_\uptok\Sigma_\uptok^{-1/2} \right)}{\mu_k\left(\Sigma^{-1/2}_\uptok X_\uptok^\top X_\uptok\Sigma_\uptok^{-1/2} \right)^2}\|X_\ktoinf \theta^*_\ktoinf\|^2 \\ &\qquad +\frac{\|\theta_\uptok^*\|_{\Sigma_\uptok^{-1}}^2}{\mu_n(A_k^{-1})^2\mu_k\left(\Sigma^{-1/2}_\uptok X_\uptok^\top X_\uptok\Sigma_\uptok^{-1/2} \right)^2}\\
&\qquad +\lambda_{k+1} \bigl(1+ \max(0, -\lambda)\mu_1(A_k^{-1})\bigr) \mu_1(A^{-1})\|X_\ktoinf\theta^*_\ktoinf\|^2\\
&\qquad +\lambda_{k+1} \bigl(1+ \max(0, -\lambda)\mu_1(A_k^{-1})\bigr)\frac{\mu_1(A_k^{-1})}{\mu_n(A_k^{-1})^2} \frac{\mu_1(\Sigma_{\uptok}^{-1/2}X_\uptok^\top X_\uptok\Sigma_{\uptok}^{-1/2})}{\mu_k(\Sigma_{\uptok}^{-1/2}X_\uptok^\top X_\uptok\Sigma_{\uptok}^{-1/2})^2}\|\Sigma_{\uptok}^{-1/2}\theta^*_\uptok\|^2.
\end{align*}

\end{lemma}
\subsection{First $k$ components}
We need to bound $\|\theta^*_\uptok- \hat{\theta}(y)_\uptok(\lambda, X\theta^*)\|_{\Sigma_\uptok}^2$. By Section \ref{sec::identity}, in particular identity \eqref{eq::identity}, we have
\[
\hat{\theta}( X\theta^*)_\uptok + X_\uptok^\top A_k^{-1} X_\uptok\hat{\theta}( X\theta^*)_\uptok = X_\uptok^\top A_k^{-1}X\theta^*.
\]
Denote the error vector as $\zeta:= \hat{\theta}( X\theta^*) - \theta^*$. We can rewrite the equation above as
\[
\zeta_\uptok + X_\uptok^\top A_k^{-1} X_\uptok\zeta_\uptok = X_\uptok^\top A_k^{-1}X_\ktoinf \theta^*_\ktoinf - \theta^*_\uptok.
\]

Multiplying both sides by $\zeta_\uptok^\top$ from the left and using that $\zeta_\uptok^\top \zeta_\uptok = \|\zeta_\uptok\|^2 \geq 0$ we obtain
\[
\zeta_\uptok^\top X_\uptok^\top A_k^{-1} X_\uptok\zeta_\uptok \leq \zeta_\uptok^\top X_\uptok^\top A_k^{-1}X_\ktoinf \theta^*_\ktoinf - \zeta_\uptok^\top \theta^*_\uptok.
\]

Next, divide and multiply by $ \Sigma_\uptok^{1/2}$ in several places: 
\begin{align*}
\zeta_\uptok^\top \Sigma_\uptok^{1/2} \Sigma^{-1/2}_\uptok X_\uptok^\top A_k^{-1} X_\uptok\Sigma_\uptok^{-1/2} \Sigma^{1/2}_\uptok\zeta_\uptok \leq& \zeta_\uptok^\top \Sigma_\uptok^{1/2} \Sigma^{-1/2}_\uptok X_\uptok^\top A_k^{-1}X_\ktoinf \theta^*_\ktoinf\\
&\qquad - \zeta_\uptok^\top \Sigma_\uptok^{1/2} \Sigma^{-1/2}_\uptok \theta^*_\uptok.
\end{align*}
Now we pull out the lowest singular values of the matrices in the LHS and largest singular values of the matrices in the RHS to obtain lower and upper bounds respectively, yielding
\begin{align*}
& \|\zeta_\uptok\|_{\Sigma_\uptok}^2\mu_n(A_k^{-1})\mu_k\left(\Sigma^{-1/2}_\uptok X_\uptok^\top X_\uptok\Sigma_\uptok^{-1/2} \right) \\
& \leq  \|\zeta_\uptok\|_{\Sigma_\uptok} \mu_1(A_k^{-1})\sqrt{\mu_1\left(\Sigma^{-1/2}_\uptok X_\uptok^\top X_\uptok\Sigma_\uptok^{-1/2} \right)}\|X_\ktoinf \theta^*_\ktoinf\|\\
& + \|\zeta_\uptok\|_{\Sigma_\uptok} \|\theta_\uptok^*\|_{\Sigma_\uptok^{-1}},
\end{align*}
and so
\begin{align*}
\|\zeta_\uptok\|_{\Sigma_\uptok} \leq& \frac{\mu_1(A_k^{-1})}{\mu_n(A_k^{-1})}\frac{\mu_1\left(\Sigma^{-1/2}_\uptok X_\uptok^\top X_\uptok\Sigma_\uptok^{-1/2} \right)^{1/2}}{\mu_k\left(\Sigma^{-1/2}_\uptok X_\uptok^\top X_\uptok\Sigma_\uptok^{-1/2} \right)}\|X_\ktoinf \theta^*_\ktoinf\| \\ +&\frac{\|\theta_\uptok^*\|_{\Sigma_\uptok^{-1}}}{\mu_n(A_k^{-1})\mu_k\left(\Sigma^{-1/2}_\uptok X_\uptok^\top X_\uptok\Sigma_\uptok^{-1/2} \right)}.
\end{align*}

\subsection{The rest of the components}
Recall that the full bias term is $ \|(I_p - X^\top(\lambda I_n + XX^\top)^{-1}X)\theta^*\|^2_\Sigma$ and that $A = \lambda I_n + XX^\top$. The contribution of the components of $\zeta$, starting from the $k+1$st can be bounded as follows:
\begin{multline*}
\|\theta^*_\ktoinf - X_\ktoinf^\top A^{-1}X\theta^*\|^2_{\Sigma_\ktoinf}\\
\leq 3\left(\|\theta^*_\ktoinf\|_{\Sigma_\ktoinf}^2 + \|X_\ktoinf^\top A^{-1}X_\ktoinf\theta^*_\ktoinf\|^2_{\Sigma_\ktoinf} + \|X_\ktoinf^\top A^{-1}X_\uptok\theta^*_\uptok\|^2_{\Sigma_\ktoinf}\right).
\end{multline*}

First of all, let's deal with the second term:
\begin{align*}
\|X_\ktoinf^\top A^{-1}X_\ktoinf\theta^*_\ktoinf\|^2_{\Sigma_\ktoinf}
=& \|\Sigma_\ktoinf^{1/2}X_\ktoinf^\top A^{-1}X_\ktoinf\theta^*_\ktoinf\|^2\\
\leq&\|\Sigma_\ktoinf\|\|X_\ktoinf^\top A^{-1}X_\ktoinf\theta^*_\ktoinf\|^2\\
=&\lambda_{k+1}(\theta^*_\ktoinf)^\top X_\ktoinf^\top A^{-1} \underbrace{(A - \lambda I_n - X_\uptok X_\uptok^\top)}_{X_\ktoinf X_\ktoinf^\top}A^{-1}X_\ktoinf\theta^*_\ktoinf\\
\leq& \lambda_{k+1}(\theta^*_\ktoinf)^\top X_\ktoinf^\top A^{-1} {\bigl(A + \max(0, -\lambda) I_n\bigr)}A^{-1}X_\ktoinf\theta^*_\ktoinf\\
\leq& \lambda_{k+1} \bigl(\mu_1(A^{-1}) + \max(0, -\lambda)\mu_1(A^{-1})^2\bigr) \|X_\ktoinf\theta^*_\ktoinf\|^2\\
\leq& \lambda_{k+1} \bigl(1+ \max(0, -\lambda)\mu_1(A_k^{-1})\bigr)\mu_1(A_k^{-1}) \|X_\ktoinf\theta^*_\ktoinf\|^2,
\end{align*}
where we used that $\mu_1(A_k^{-1}) \geq \mu_1(A^{-1})$ in the last transition.

Now, let's deal with the last term. Note that $A = A_k + X_\uptok X_\uptok^\top$. By the Sherman–Morrison–Woodbury formula,
\begin{align*}
A^{-1}X_\uptok =& (A_k^{-1} + X_\uptok X_\uptok^\top)^{-1}X_\uptok\\
=&\left(A_k^{-1} - A_k^{-1}X_\uptok\left(I_k + X_\uptok^\top A_k^{-1}X_\uptok\right)^{-1}X_\uptok^TA_k^{-1}\right)X_\uptok\\
=&A_k^{-1}X_\uptok\left(I_n -  \left(I_k + X_\uptok^\top A_k^{-1}X_\uptok\right)^{-1}X_\uptok^TA_k^{-1}X_\uptok\right)\\
=&A_k^{-1}X_\uptok\left(I_n -  \left(I_k + X_\uptok^\top A_k^{-1}X_\uptok\right)^{-1}\left(I_k + X_\uptok^TA_k^{-1}X_\uptok - I_k\right)\right)\\
=& A_k^{-1}X_\uptok\left(I_k + X_\uptok^\top A_k^{-1}X_\uptok\right)^{-1}.
\end{align*}
Thus, 
\begin{align*}
&\|X_\ktoinf^\top A^{-1}X_\uptok\theta^*_\uptok\|^2_{\Sigma_\ktoinf}\\
=& \|X_\ktoinf^\top A_k^{-1}X_\uptok\left(I_k + X_\uptok^\top A_k^{-1}X_\uptok\right)^{-1}\theta^*_\uptok\|^2_{\Sigma_\ktoinf}\\
=& \|\Sigma_{\ktoinf}^{1/2}X_\ktoinf^\top A_k^{-1}X_\uptok\Sigma_{\uptok}^{-1/2}\left(\Sigma_{\uptok}^{-1} + \Sigma_{\uptok}^{-1/2}X_\uptok^\top A_k^{-1}X_\uptok\Sigma_{\uptok}^{-1/2}\right)^{-1}\Sigma_{\uptok}^{-1/2}\theta^*_\uptok\|^2\\
\leq& \|A_k^{-1/2}X_\ktoinf \Sigma_\ktoinf X_\ktoinf^\top A_k^{-1/2}\|\mu_1(A_k^{-1/2})^2 \frac{\mu_1(\Sigma_{\uptok}^{-1/2}X_\uptok^\top X_\uptok\Sigma_{\uptok}^{-1/2})}{\mu_k(\Sigma_{\uptok}^{-1/2}X_\uptok^\top A_k^{-1}X_\uptok\Sigma_{\uptok}^{-1/2})^2}\|\Sigma_{\uptok}^{-1/2}\theta^*_\uptok\|^2\\
\leq&\|\Sigma_\ktoinf \|\|A_k^{-1/2}X_\ktoinf X_\ktoinf^\top A_k^{-1/2}\|\frac{\mu_1(A_k^{-1})}{\mu_n(A_k^{-1})^2} \frac{\mu_1(\Sigma_{\uptok}^{-1/2}X_\uptok^\top X_\uptok\Sigma_{\uptok}^{-1/2})}{\mu_k(\Sigma_{\uptok}^{-1/2}X_\uptok^\top X_\uptok\Sigma_{\uptok}^{-1/2})^2}\|\Sigma_{\uptok}^{-1/2}\theta^*_\uptok\|^2\\
=& \lambda_1\|I_n - \lambda A_k^{-1}\|\frac{\mu_1(A_k^{-1})}{\mu_n(A_k^{-1})^2} \frac{\mu_1(\Sigma_{\uptok}^{-1/2}X_\uptok^\top X_\uptok\Sigma_{\uptok}^{-1/2})}{\mu_k(\Sigma_{\uptok}^{-1/2}X_\uptok^\top X_\uptok\Sigma_{\uptok}^{-1/2})^2}\|\Sigma_{\uptok}^{-1/2}\theta^*_\uptok\|^2\\
\leq&\lambda_1\bigl(1 + \max(0,-\lambda) \mu_1(A_k^{-1})\bigr)\frac{\mu_1(A_k^{-1})}{\mu_n(A_k^{-1})^2} \frac{\mu_1(\Sigma_{\uptok}^{-1/2}X_\uptok^\top X_\uptok\Sigma_{\uptok}^{-1/2})}{\mu_k(\Sigma_{\uptok}^{-1/2}X_\uptok^\top X_\uptok\Sigma_{\uptok}^{-1/2})^2}\|\Sigma_{\uptok}^{-1/2}\theta^*_\uptok\|^2,\\
\end{align*}
where in the last transition we used the fact that $I_n - \lambda A_k^{-1}$ is a PSD matrix with norm bounded by 1 for $\lambda > 0$.

Putting those bounds together yields the result.

\section{Main results}
\subsection{Upper bound on the prediction MSE}
\label{subsec::upper}

\mainthm*
%
\begin{proof}
Lemmas \ref{lm::variance} and \ref{lm::bias} bound the bias and variance on the event that $A_k$ is PD. Next to those lemmas we already put explanations of why those bounds are easy to assess via concentration arguments. Here we just do this rigorously.

Recall the bounds from Lemmas \ref{lm::variance} and \ref{lm::bias}: for some absolute constant $c$
\begin{align}
B/c \leq& \|\theta^*_\ktoinf\|_{\Sigma_\ktoinf}^2\\
+& \frac{\mu_1(A_k^{-1})^2}{\mu_n(A_k^{-1})^2}\frac{\mu_1\left(\Sigma^{-1/2}_\uptok X_\uptok^\top X_\uptok\Sigma_\uptok^{-1/2} \right)}{\mu_k\left(\Sigma^{-1/2}_\uptok X_\uptok^\top X_\uptok\Sigma_\uptok^{-1/2} \right)^2}\|X_\ktoinf \theta^*_\ktoinf\|^2 \label{eq::upper_bound_first_term}\\ +&\frac{\|\theta_\uptok^*\|_{\Sigma_\uptok^{-1}}^2}{\mu_n(A_k^{-1})^2\mu_k\left(\Sigma^{-1/2}_\uptok X_\uptok^\top X_\uptok\Sigma_\uptok^{-1/2} \right)^2}\\
+&\lambda_{k+1} \bigl(1+ \max(0, -\lambda)\mu_1(A_k^{-1})\bigr) \mu_1(A^{-1})\|X_\ktoinf\theta^*_\ktoinf\|^2\\
+&\lambda_{k+1} \bigl(1+ \max(0, -\lambda)\mu_1(A_k^{-1})\bigr)\frac{\mu_1(A_k^{-1})}{\mu_n(A_k^{-1})^2} \frac{\mu_1(\Sigma_{\uptok}^{-1/2}X_\uptok^\top X_\uptok\Sigma_{\uptok}^{-1/2})}{\mu_k(\Sigma_{\uptok}^{-1/2}X_\uptok^\top X_\uptok\Sigma_{\uptok}^{-1/2})^2}\|\Sigma_{\uptok}^{-1/2}\theta^*_\uptok\|^2,\\
V/c \leq &\frac{\mu_1(A_k^{-1})^2\tr(X_\uptok\Sigma_\uptok^{-1} X_\uptok^\top)}{\mu_n(A_k^{-1})^2 \mu_k\left(\Sigma_\uptok^{-1/2}X_\uptok^\top X_\uptok \Sigma_\uptok^{-1/2}\right)^2}\\
+& \mu_1(A_k^{-1})^2\tr(X_\ktoinf\Sigma_\ktoinf X_\ktoinf^\top)\label{eq::upper_bound_last_term},
\end{align}
where the first four terms correspond to the bias and the last two to the variance.
By inspecting that expression one can notice that it consists of some products of simple quantities that could be assessed individually. Namely, those quantities are:
\begin{enumerate}
\item $\mu_1(A_k^{-1})$ and $\mu_n(A_k^{-1})$ --- smallest and largest singular values of $A_k$. 
In this theorem we assume that those quantities are known or there is some oracle control over them.

\item $\mu_1\left(\Sigma_\uptok^{-1/2}X_\uptok^\top X_\uptok \Sigma_\uptok^{-1/2}\right)$ and $\mu_k\left(\Sigma_\uptok^{-1/2}X_\uptok^\top X_\uptok \Sigma_\uptok^{-1/2}\right)$. 

The matrix $X_\uptok \Sigma_\uptok^{-1/2} \in \R^{k\times n}$ has $n$ i.i.d. columns with isotropic sub-Gaussian distribution in $\R^k$. The matrix $\Sigma_\uptok^{-1/2}X_\uptok^\top X_\uptok \Sigma_\uptok^{-1/2}/n$ is the sample covariance matrix of those columns, so when $k \ll n$ it concentrates around its expectation, which is $I_k$. More precisely, by Theorem 5.39 in \citep{Vershynin_2012}, for some constants $c_x',C_x'$ (which only depend on $\sigma_x$ ) for every $t > 0$ s.t. $\sqrt{n} - C_x'\sqrt{k} - \sqrt{t} > 0$, with probability $1 - 2\exp(-c_x' t)$,
\begin{align}
\label{eq::eigvals_k_by_k}
\mu_k\left(\Sigma_\uptok^{-1/2}X_\uptok^\top X_\uptok \Sigma_\uptok^{-1/2}\right) \geq&\left(\sqrt{n} - C_x'\sqrt{k} - \sqrt{t}\right)^2,  \\
\mu_1\left(\Sigma_\uptok^{-1/2}X_\uptok^\top X_\uptok \Sigma_\uptok^{-1/2}\right) \leq& \left(\sqrt{n} + C_x'\sqrt{k} + \sqrt{t}\right)^2.
\end{align}

\item $\tr\left(X_\uptok\Sigma_\uptok^{-1} X_\uptok^\top\right)$  and $\tr\left(X_\ktoinf \Sigma_\ktoinf X_\ktoinf^\top\right)$. 

$\tr\left(X_\uptok\Sigma_\uptok^{-1} X_\uptok^\top\right)$ is the sum of squared norms of columns of $\Sigma_\uptok^{-1/2}X_\uptok^\top$, which are $n$ i.i.d. isotropic vectors in $\R^k$.  Analogously, $\tr\left(X_\ktoinf \Sigma_\ktoinf X_\ktoinf^\top\right)$ is the sum of squared norms of n i.i.d. sub-Gaussian vectors with covariance $\Sigma_\ktoinf^2$. Therefore, they concentrate around their expectations by the law of large numbers. More precisely, by Lemma \ref{lm::sum of norms} with probability at least $1 - 4e^{-c_2t}$,
\begin{align*}
\tr\left(X_\uptok\Sigma_\uptok^{-1} X_\uptok^\top\right)\leq (n + \sqrt{tn}\sigma_x^2)k,\\
\tr\left(X_\ktoinf \Sigma_\ktoinf X_\ktoinf^\top\right)\leq (n + \sqrt{tn}\sigma_x^2)\sum_{i > k} \lambda_i^2.
\end{align*}

\item $\|X_\ktoinf\theta^*_\ktoinf\|^2$.

Once again, this quantity concentrates by the law of large numbers. The vector ${X_\ktoinf\theta^*_\ktoinf}/{\|\theta^*_\ktoinf\|_{\Sigma_\ktoinf}}$ has $n$ i.i.d.\ centered components with unit variances  and sub-Gaussian norms at most $\sigma_x$.  Treating those components as sub-Gaussian vectors in $\R^1$, we can apply Lemma \ref{lm::sum of norms} to get that for any $t \in (0,n)$, with probability at least $1 - 2e^{-c_2t}$,
\[
\|X_\ktoinf\theta^*_\ktoinf\|^2\leq (n + \sqrt{tn}\sigma_x^2)\|\theta^*_\ktoinf\|_{\Sigma_\ktoinf}^2.
\]
\end{enumerate}

Now take constant $c_4$ to be large enough depending on $\sigma_x$ and set $t = n/c_4$. For some constant $c_5$ which only depends on $\sigma_x$ we get that with probability at least $1 - c_5 e^{-n/c_5}$, all the following inequalities hold at the same time:
\begin{align*}
\mu_k\left(\Sigma_\uptok^{-1/2}X_\uptok^\top X_\uptok \Sigma_\uptok^{-1/2}\right) \geq& n/c_5,\\
\mu_1\left(\Sigma_\uptok^{-1/2}X_\uptok^\top X_\uptok \Sigma_\uptok^{-1/2}\right) \leq& c_5 n,\\
\|X_\ktoinf\theta^*_\ktoinf\|^2 \leq& c_5 n \|\theta^*_\ktoinf\|_{\Sigma_\ktoinf}^2,\\
\|X_\ktoinf \Sigma_\ktoinf X_\ktoinf^\top\| \leq& c_5\left(\lambda_{k+1}^2n + \sum_{i > k} \lambda_i^2\right),\\
\tr\left(X_\uptok\Sigma_\uptok^{-1} X_\uptok^\top\right)\leq& c_5nk,\\
\tr\left(X_\ktoinf \Sigma_\ktoinf X_\ktoinf^\top\right)\leq& c_5n\sum_{i > k} \lambda_i^2.
\end{align*}

Next, plug these bounds into \eqref{eq::upper_bound_first_term}--\eqref{eq::upper_bound_last_term}:
\begin{equation*}
 \frac{\mu_1(A_k^{-1})^2}{\mu_n(A_k^{-1})^2}\frac{\mu_1\left(\Sigma^{-1/2}_\uptok X_\uptok^\top X_\uptok\Sigma_\uptok^{-1/2} \right)}{\mu_k\left(\Sigma^{-1/2}_\uptok X_\uptok^\top X_\uptok\Sigma_\uptok^{-1/2} \right)^2}\|X_\ktoinf \theta^*_\ktoinf\|^2 \leq c_5^3\frac{\mu_1(A_k^{-1})^2}{\mu_n(A_k^{-1})^2} \|\theta^*_\ktoinf\|^2_{\Sigma_\ktoinf},
 \end{equation*}
 \begin{equation*}
\frac{\|\theta_\uptok^*\|_{\Sigma_\uptok^{-1}}^2}{\mu_n(A_k^{-1})^2\mu_k\left(\Sigma^{-1/2}_\uptok X_\uptok^\top X_\uptok\Sigma_\uptok^{-1/2} \right)^2}   \leq c_5^2 \frac{\|\theta_\uptok^*\|_{\Sigma_\uptok^{-1}}^2}{\mu_n(A_k^{-1})^2n^2}, 
 \end{equation*}
 \begin{multline*}
\lambda_{k+1} \bigl(1+ \max(0, -\lambda)\mu_1(A_k^{-1})\bigr) \mu_1(A^{-1})\|X_\ktoinf\theta^*_\ktoinf\|^2 \leq\\ 
\leq c_5^2\lambda_{k+1}\left(1 +\max(0, -\lambda)\mu_1(A_k^{-1})\right) \mu_1(A_k^{-1})n\|\theta^*_\ktoinf\|^2_{\Sigma_\ktoinf},
 \end{multline*}
 \begin{multline*}
\lambda_{k+1} \bigl(1+ \max(0, -\lambda)\mu_1(A_k^{-1})\bigr)\frac{\mu_1(A_k^{-1})}{\mu_n(A_k^{-1})^2} \frac{\mu_1(\Sigma_{\uptok}^{-1/2}X_\uptok^\top X_\uptok\Sigma_{\uptok}^{-1/2})}{\mu_k(\Sigma_{\uptok}^{-1/2}X_\uptok^\top X_\uptok\Sigma_{\uptok}^{-1/2})^2}\|\Sigma_{\uptok}^{-1/2}\theta^*_\uptok\|^2\leq\\
\leq c_5^4\lambda_{k+1}\left(1 +\max(0, -\lambda)\mu_1(A_k^{-1})\right) \frac{\mu_1(A_k^{-1})}{\mu_n(A_k^{-1})^2}\frac{1}{n}\|\theta^*_\uptok\|_{\Sigma_\uptok^{-1}}^2,
 \end{multline*}
 \begin{equation*}
\frac{\mu_1(A_k^{-1})^2\tr(X_\uptok\Sigma_\uptok^{-1} X_\uptok^\top)}{\mu_n(A_k^{-1})^2 \mu_k\left(\Sigma_\uptok^{-1/2}X_\uptok^\top X_\uptok \Sigma_\uptok^{-1/2}\right)^2} \leq c_5^3\frac{\mu_1(A_k^{-1})^2}{\mu_n(A_k^{-1})^2}\frac{k}{n},
 \end{equation*}
 \begin{equation*}
 \mu_1(A_k^{-1})^2\tr(X_\ktoinf\Sigma_\ktoinf X_\ktoinf^\top) \leq c_5  \mu_1(A_k^{-1})^2 n\sum_{i > k}\lambda_i^2.
\end{equation*}

Putting all the terms together gives the result.
\end{proof}

\upperboundgivenL*
\begin{proof}
Almost all the work was already done in Lemma \ref{lm::cond number for k}. It says that for some absolute constant $c_1$ and for any $t \in (0,n)$ with probability  at least $1 - \delta - 2e^{-c_1t}$,
\[
\frac{1}{L}\left(1 - \frac{\sqrt{t}\sigma_x^2}{\sqrt{n}(1-\gamma)}\right)\left(\lambda + \sum_{i} \lambda_i\right) \leq \mu_n(A_k) \leq \mu_1(A_k) \leq  L\left(1 - \frac{\sqrt{t}\sigma_x^2}{\sqrt{n}(1-\gamma)}\right)\left(\lambda + \sum_{i } \lambda_i\right).
\]

Moreover, if   $\delta < 1-4e^{-c_1t}$, then 
\[
\rho_k \geq \frac{1 - \sigma^2\sqrt{t/n}}{L+\frac{\gamma}{1-\gamma} + \frac{\sqrt{t}\sigma^2L }{\sqrt{n}(1-\gamma)}}.
\]

We just need to choose $t$, plug these bounds into the result of Theorem \ref{th::main upper main body} and evaluate the result up to multiplicative constants.

First, choose constant $c_2$ large enough depending on $L$, $\gamma$, $\sigma_x$ , and put $t = n/c_2$. Statements above imply that if $\delta < 1 - 4e^{-n/(c_1c_2)}$, then for some constant $c_3$ which only depends on $L$, $\gamma$, $\sigma_x$, with probability at least $1 -  \delta - c_2 e^{-n/(c_1c_2)}$,
\begin{align*}
\mu_n(A_k^{-1}) =& \mu_1(A_k)^{-1} \geq \frac{1}{c_3} \left(\lambda + \sum_{i} \lambda_i\right)^{-1},\\
\mu_1(A_k^{-1}) =& \mu_n(A_k)^{-1} \leq \frac{1}{c_3}\left(\lambda + \sum_{i} \lambda_i\right)^{-1},\\
\rho_k \geq& \frac{1}{c_3}.
\end{align*}

These three inequalities allow us to evaluate the result of Theorem \ref{th::main upper main body}: let's plug them term-by-term:
\begin{itemize}
\item Since $\lambda > -\gamma \sum_{i > k} \lambda_i$, 
\[
\max(0, -\lambda) \leq \frac{\gamma}{1-\gamma} \left(\lambda + \sum_{i} \lambda_i\right).
\]

Thus,
\[
1 + \max(0, -\lambda)\mu_1(A_k^{-1})\leq 1 + \frac{\gamma}{1-\gamma} c_3,
\]
so this term is just a constant.

\item 
\[
n\lambda_{k+1} \mu_1(A_k^{-1}) \leq c_3 n\lambda_{k+1} \left(\lambda + \sum_{i} \lambda_i\right) = c_3 / \rho_k \leq c_3^2,
\]
so this term is also just a constant. 

\item 
\[
\frac{1}{n^2\mu_n(A_k^{-1})^2} \leq \frac{c_3^2}{n}\left(\lambda + \sum_{i} \lambda_i\right)^2.
\]
\item 
\begin{align*}
\frac{\lambda_{k+1}}{n}\frac{\mu_1(A_k^{-1})}{\mu_n(A_k^{-1})^2} \leq& \frac{c_3^3}{n^2}\cdot n\lambda_{k+1}\left(\lambda + \sum_{i} \lambda_i\right) \\
=& \frac{c_3^3}{n^2}\cdot \rho_k^{-1}\left(\lambda + \sum_{i} \lambda_i\right)^2\\
\leq& \frac{c_3^4}{n^2}\left(\lambda + \sum_{i} \lambda_i\right)^2.
\end{align*}
\item $\frac{\mu_1(A_k^{-1})^2}{\mu_n(A_k^{-1})^2} \leq L^2$ --- also just a constant.
\item 
\[
 n \mu_1(A_k^{-1})^2 \leq c_3^2 n \left(\lambda + \sum_{i} \lambda_i\right)^{-2}.
 \]
\end{itemize}

Plugging all these bounds in the statement of Theorem \ref{th::main upper main body} gives the result for a large enough $c$.
\end{proof}

\subsection{Upper bound matches the lower bound}
\label{subseq::upper is lower}

In the next theorem we show that the upper bound given in Theorem \ref{th::main upper main body} matches the lower bounds from Lemmas \ref{lm::var lower main body} and \ref{lm::bias lower main body} if we choose suitable $k$. Note that by Lemmas \ref{lm::cond number for k} and \ref{lm::k star must give eigvals}, being able to control the condition number of $A_{k'}$ for some $k' < n$ implies that we can choose a suitable $k$.
(Note that there are choices of $\theta^*$ and $\Sigma$ for which the lower bound $\overline{B}$ is larger than the upper bound of Lemma~5.4 in~\citep{pmlr-v119-negrea20a}; this seems to be because the proof of Lemma~B.1 in that paper applies Lemma~B.2 to a nonsymmetric matrix. This error was removed in the newer version of the same paper, which uses the results of \cite{benign_overfitting} instead.)

\uppersameaslower*
\begin{proof}
First of all, we represent
\begin{align*}
\|\theta^*_\ktoinf\|_{\Sigma_\ktoinf}^2 +  \|\theta_\uptok^*\|_{\Sigma_\uptok^{-1}}^2\left(\frac{\lambda + \sum_{i > k} \lambda_i}{n}\right)^2 =& \sum_i \left(\mathbbm{1}\{i \leq k\}\frac{|\theta_i^*|^2\rho_k^2\lambda_{k+1}^2}{\lambda_i} + \mathbbm{1}\{i > k\}\lambda_i|\theta_i^*|^2\right)\\
\frac{k}{n} + \frac{n\sum_{i > k} \lambda_i^2}{\left(\lambda + \sum_{i > k} \lambda_i\right)^2}=& \sum_i \left(\mathbbm{1}\{i \leq k\}\frac1n + \mathbbm{1}\{i > k\}\frac{\lambda_i^2}{n\lambda_{k+1}^2\rho_k^2}\right)
\end{align*}

In the following we will bound the ratio of the sums from the statement of the theorem by bounding the ratios of the corresponding terms.

\begin{itemize}
\item First case:  $\rho_k \in (a, b)$.
\begin{itemize}
\item Bias term:
\begin{itemize}
\item $i \leq k$:
\begin{align*}
&\frac{\lambda_i|\theta_i^*|^2}{\left(1 + \frac{\lambda_i}{\lambda_{k+1}\rho_k}\right)^2} : \frac{|\theta_i^*|^2\rho_k^2\lambda_{k+1}^2}{\lambda_i}\\
=& \frac{\lambda_i^2}{\rho_k^2\lambda_{k+1}^2\left(1 + \frac{\lambda_i}{\lambda_{k+1}\rho_k}\right)^2}\\
=&\left(1 + \frac{\lambda_{k+1}\rho_k}{\lambda_i}\right)^{-2}\\
\in& \left((1+b)^{-2},1\right)
\end{align*}

\item $i > k$:
\begin{align*}
&\frac{\lambda_i|\theta_i^*|^2}{\left(1 + \frac{\lambda_i}{\lambda_{k+1}\rho_k}\right)^2} : \lambda_i|\theta_i^*|^2\\
=&\left(1 + \frac{\lambda_i}{\lambda_{k+1}\rho_k}\right)^{-2}\\
\in&\left((1+a^{-1})^{-2},1\right)
\end{align*}
\end{itemize}
\item Variance term:
\begin{itemize}
\item $i \leq k$:
\begin{align*}
&\frac1n\min\left\{1, \frac{\lambda_i^2}{ \lambda_{k+1}^2(\rho_k + 1)^2}\right\} : \frac1n\\
\in& \left((1+b)^{-2},1\right]
\end{align*}
\item $i > k$:
\begin{align*}
&\frac1n\min\left\{1, \frac{\lambda_i^2}{ \lambda_{k+1}^2(\rho_k + 1)^2}\right\} : \frac{\lambda_i^2}{n\lambda_{k+1}^2\rho_k^2}\\
=& \frac{\lambda_i^2}{ \lambda_{k+1}^2(\rho_k + 1)^2}:\frac{\lambda_i^2}{\lambda_{k+1}^2\rho_k^2}\\
=& \frac{\rho_k^2}{ (\rho_k + 1)^2}\\
\in& \left((1 +  a^{-1})^{-2},1\right)
\end{align*}
\end{itemize}
\end{itemize}
\item Second case: $k = \min\{l: \rho_l > b\}$. In this case we have
\begin{gather*}
\rho_{k} \geq b,\\
\frac{\lambda_k + n\lambda_{k+1} \rho_k}{n\lambda_k} = \frac{\lambda + \lambda_k + \sum_{i > k} \lambda_i}{n\lambda_k} = \rho_{k-1} < b,\\
\forall i \leq k:\quad \lambda_i \geq \lambda_k \geq \frac{n\lambda_{k+1}\rho_k}{nb-1} = \frac{\lambda_{k+1}\rho_k}{b} \geq \frac{\lambda_{k+1}\rho_k}{b}.
\end{gather*}

The rest of the computation is analogous to the previous case:
\begin{itemize}
\item Bias term:
\begin{itemize}
\item $i \leq k$:
\begin{align*}
&\frac{\lambda_i|\theta_i^*|^2}{\left(1 + \frac{\lambda_i}{\lambda_{k+1}\rho_k}\right)^2} : \frac{|\theta_i^*|^2\rho_k^2\lambda_{k+1}^2}{\lambda_i}\\
=& \frac{\lambda_i^2}{\rho_k^2\lambda_{k+1}^2\left(1 + \frac{\lambda_i}{\lambda_{k+1}\rho_k}\right)^2}\\
=&\left(1 + \frac{\lambda_{k+1}\rho_k}{\lambda_i}\right)^{-2}\\
\in& \left[(1+b)^{-2},1\right)
\end{align*}

\item $i > k$:
\begin{align*}
&\frac{\lambda_i|\theta_i^*|^2}{\left(1 + \frac{\lambda_i}{\lambda_{k+1}\rho_k}\right)^2} : \lambda_i|\theta_i^*|^2\\
=&\left(1 + \frac{\lambda_i}{\lambda_{k+1}\rho_k}\right)^{-2}\\
\in&\left[(1+b^{-1})^{-2},1\right)
\end{align*}
\end{itemize}
\item Variance term:
\begin{itemize}
\item $i \leq k$:
\begin{align*}
&\frac1n\min\left\{1, \frac{\lambda_i^2}{ \lambda_{k+1}^2(\rho_k + 1)^2}\right\} : \frac1n\\
\in& \left[\frac{\lambda_{k+1}^2\rho_k^2/b^2}{ \lambda_{k+1}^2(\rho_k + 1)^2},1\right]\\
\subseteq& \left[\frac{b^2}{(b+1)^2b^2}, 1\right]\\
=& \left[(b + 1)^{-2},1\right]
\end{align*}
\item $i > k$:
\begin{align*}
&\frac1n\min\left\{1, \frac{\lambda_i^2}{ \lambda_{k+1}^2(\rho_k + 1)^2}\right\} : \frac{\lambda_i^2}{n\lambda_{k+1}^2\rho_k^2}\\
=& \frac{\lambda_i^2}{ \lambda_{k+1}^2(\rho_k + 1)^2}:\frac{\lambda_i^2}{\lambda_{k+1}^2\rho_k^2}\\
=& \frac{\rho_k^2}{ (\rho_k + 1)^2}\\
\in& \left[(1 + b^{-1})^{-2},1\right]
\end{align*}
\end{itemize}
\end{itemize}
\end{itemize}
\end{proof}
\subsection{Alternative form of the main bound}
\label{sec::alternative form appendix}

\nokinaltform*
\begin{proof}
\begin{align*}
\lambda_{k^* + 1} \rho_{k*} =& \lambda_{k+1}\rho_{k} + \frac{1}{n}\sum_{i=k^* + 1}^k \lambda_i\\
\leq& \lambda_{k+1}\rho_{k} + \frac{k - k^*}{n}\lambda_{k^* + 1}\\
=& \lambda_{k+1}\rho_{k} + \frac{k - k^*}{n}\frac{\lambda_{k^* + 1} \rho_{k*}}{\rho_{k*}}\\
\leq& \lambda_{k+1}\rho_{k} + \frac{\lambda_{k^* + 1} \rho_{k*}}{bc},\\
\end{align*}
where we used $k - k^* < n/c$ and $\rho_{k^*} > b$ in the last transition. Moving $\frac{\lambda_{k^* + 1} \rho_{k*}}{bc}$ to the left-hand side and dividing both sides by $(1 - b^{-1}c^{-1})$ gives the result.
\end{proof}

\lambdadominatestail*

\begin{proof}
Set $\gamma = 0$ and denote $c_1$ to be the constant $c$ from Lemma \ref{lm::controlling A_k iff small ball and rank}. Take $L = 2c_1$ and $b = L^2$, $a = b/2$.  For such choice of $\gamma, L, a, b$ denote $c_2$ to be the constant from Theorem \ref{th::showcase upper} and take any $\tilde{k} < n/c_2$.

Take any $\lambda$ s.t. 
\[
\lambda \geq 2\sum_{i > \tilde{k}}\lambda_i \quad\text{ and }\quad \rho_{\tilde{k}} \geq L^2,
\]
i.e.,
\[
\lambda \geq \max\left(2\sum_{i > \tilde{k}}\lambda_i,\; L^2n\lambda_{\tilde{k}+1} - \sum_{i > \tilde{k}}\lambda_i\right).
\]

Then the conditions of the first part of Lemma \ref{lm::controlling A_k iff small ball and rank} are satisfied with $\delta = 0$, which means that with probability $1 - c_1e^{-n/c_1}$, $\mu_n(A_{\tilde{k}})\geq L^{-1}\mu_1(A_{\tilde{k}})$, so the assumptions of the first part of Theorem \ref{th::showcase upper} are satisfied with $\delta = c_1e^{-n/c_1}$ and $\bar{k} = \tilde{k}$. Note also that since $\rho_{\tilde{k}} \geq L^2 = b$, then $k^* \leq \tilde{k}$. This means that with probability at least $1 - c_1e^{-n/c_1} - c_2e^{-n/c_2}$, for $k = k^*$,
\begin{align*}
B/c_2 \leq& \|\theta^*_\ktoinf\|_{\Sigma_\ktoinf}^2 +  \|\theta_\uptok^*\|_{\Sigma_\uptok^{-1}}^2\left(\frac{\lambda + \sum_{i > k} \lambda_i}{n}\right)^2,\\
V/c_2 \leq& \frac{k}{n} + \frac{n\sum_{i > k} \lambda_i^2}{\left(\lambda + \sum_{i > k} \lambda_i\right)^2}.  \\
\end{align*}

Now since $k = k^*$, by Theorem \ref{th:upper same as lower main body} there exists a large constant $c_3$ (that depends on $b$ and $c_2$) such that on the same event,
 \begin{align*}
B/c_3 \leq& \sum_i\lambda_i|\theta_i^*|^2 \frac{\rho_k^2\lambda_{k^*+1}^2}{\left(\rho_{k^*}\lambda_{k^*+1} + \lambda_i\right)^2},\\
V/c_3\leq&  \frac1n \sum_i\frac{\lambda_i^2}{\left(\rho_{k^*}\lambda_{k^*+1} + \lambda_i\right)^2},
\end{align*}
where $\tilde{B}$ and $\tilde{V}$ are defined in Equations \eqref{eq::bias through weighted combination}--\eqref{eq::variance through weighted combination}. 

We've just cast  the bounds to the alternative form, which allows us to transition from $k^*$ to the initial value $\tilde{k}$. By Lemma \ref{lm::change of k is negligible in alternative form} since $n/c_2 \geq \tilde{k}\geq k^*$ there exists a constant $c_4$ that depends on $c_2, c_3, b$ such that on the same event
 \begin{align*}
B/c_4 \leq& \sum_i\lambda_i|\theta_i^*|^2 \frac{\rho_k^2\lambda_{k+1}^2}{\left(\rho_k\lambda_{k+1} + \lambda_i\right)^2},\\
V/c_4\leq& \frac1n \sum_i\frac{\lambda_i^2}{\left(\rho_k\lambda_{k+1} + \lambda_i\right)^2}.
\end{align*}

Finally, since $\lambda > 2\sum_{i > k}\lambda_i$, we have
\[
\lambda/n \leq \rho_k\lambda_{k+1} = \frac{1}{n}\left(\lambda + \sum_{i > k}\lambda_i\right) \leq 1.5\lambda/n.
\]
Thus, on the same event
 \begin{align*}
B/(2.25c_4) \leq& \sum_i\lambda_i|\theta_i^*|^2 \frac{(\lambda/n)^2}{\left(\lambda/n + \lambda_i\right)^2},\\
V/(2.25c_4) \leq&  \frac1n \sum_i\frac{\lambda_i^2}{\left(\lambda/n + \lambda_i\right)^2}.
\end{align*}

To finish the proof take $c = \max(2.25c_4, c_1 + c_2, L^2)$ and $\tilde{k} = \lfloor n/c\rfloor$.

\end{proof}

\invertstiltjes*
\begin{proof}
\[
d(\lambda/n) = \sum_i \frac{\lambda_i}{\lambda_i + c\lambda_{\lfloor n/c\rfloor} + \frac{2}{n}\sum_{i > \lfloor n/c\rfloor}\lambda_i}.
\]
Consider two cases: \begin{enumerate}
\item[{\bf Case 1:}] $(1 + c)\lambda_{\lfloor n/c\rfloor} \geq \frac{2}{n}\sum_{i > \lfloor n/c\rfloor}\lambda_i$. Then
\begin{align*}
&\sum_i \frac{\lambda_i}{\lambda_i + c\lambda_{\lfloor n/c\rfloor} + \frac{2}{n}\sum_{i > \lfloor n/c\rfloor}\lambda_i}\\
\geq& \sum_i \frac{\lambda_i}{\lambda_i + (1+2c)\lambda_{\lfloor n/c\rfloor}}\\
\geq& \sum_{i\leq \lfloor n/c\rfloor} \frac{\lambda_i}{\lambda_i +(1 +  2c)\lambda_{\lfloor n/c\rfloor}}\\
\geq& \sum_{i\leq \lfloor n/c\rfloor} \frac{\lambda_i}{\lambda_i(2 + 2c)}\\
=& \frac{\lfloor n/c\rfloor}{2 + 2c}.
\end{align*}
\item[{\bf Case 2:}]  $(1 + c)\lambda_{\lfloor n/c\rfloor} < \frac{2}{n}\sum_{i > \lfloor n/c\rfloor}\lambda_i$. Then
\begin{align*}
&\sum_i \frac{\lambda_i}{\lambda_i + c\lambda_{\lfloor n/c\rfloor} + \frac{2}{n}\sum_{i > \lfloor n/c\rfloor}\lambda_i}\\
\geq& \sum_{i > \lfloor n/c\rfloor} \frac{\lambda_i}{\lambda_i + c\lambda_{\lfloor n/c\rfloor} + \frac{2}{n}\sum_{i > \lfloor n/c\rfloor}\lambda_i}\\
\geq& \sum_{i > \lfloor n/c\rfloor} \frac{\lambda_i}{(1 + c)\lambda_{\lfloor n/c\rfloor} + \frac{2}{n}\sum_{i > \lfloor n/c\rfloor}\lambda_i}\\
\geq&\sum_{i > \lfloor n/c\rfloor} \frac{\lambda_i}{ \frac{4}{n}\sum_{i > \lfloor n/c\rfloor}\lambda_i}\\
=& \frac{n}{4}.
\end{align*}
A straightforward computation shows that if $n \geq c^2 + c$ then $n/c - 1 \geq n/(c+1)$, so
\[
 \frac{\lfloor n/c\rfloor}{2 + 2c} \geq \frac{n}{2(c+1)^2},
\]
 which finishes the proof.
\end{enumerate}
\end{proof}

\section{Negative regularization}
\label{sec::negative regularization appendix}

\uniformlower*
\begin{proof}
We start exactly as in the proof of Lemma \ref{lm::bias lower main body}, where it was shown that if $A_{-i}$ is PSD for every $i$ (which is satisfied almost surely when $\lambda \geq 0$ ) then
\begin{equation}
\label{eq::algebraic lower bound bias}
\E_{\theta^*} B \geq \sum_i  \frac{\lambda_i\bar{\theta}^2_i}{(1 + \lambda_iz_i^\top A_{-i}^{-1}z_i)^2} \geq \sum_i  \frac{\lambda_i\bar{\theta}^2_i}{(1 + \lambda_i\mu_n( A_{-i}^{-1})\|z_i\|^2)^2}.
\end{equation}

Note that have $\mu_n( A_{-i}^{-1})$ is a decreasing function of $\lambda$ with probability 1. Thus, the right-hand side of \eqref{eq::algebraic lower bound bias} is a non-decreasing function of $\lambda$ with probability 1, and any lower bound for it when $\lambda = 0$ will also hold uniformly for all $\lambda \geq 0$. Thus, for the remainder of the proof, fix $\lambda = 0$.

We are going to use Lemma \ref{lm::tight control of eigenvalues for independent coordinates} to lower bound $\mu_n(A_{-i})$ for each $i$ separately (we are {\em not} looking for a uniform bound over all $i$ simultaneously). If $i \leq k$, then $A_{-i} \succeq A_{k}$ with probability 1, so we can just use  Lemma \ref{lm::tight control of eigenvalues for independent coordinates} directly. If $i > k$, consider the following matrix:
\[
X_\ktoinf^{(i)} := [\sqrt{\lambda_{k+1}}z_{k+1}, \dots, \sqrt{\lambda_{i-1}}z_{i-1} , \sqrt{\lambda_{i}}z_{1}, \sqrt{\lambda_{i+1}}z_{i+1}, \dots, \sqrt{\lambda_{p}}z_{p}].
\]
In words, we took matrix $X$, multiplied the first column by $\sqrt{\lambda_{i}/\lambda_1}$ (to make the variances equal to $\lambda_i$), swapped the first column with the $i$-th column and dropped the first $k$ columns. The purpose of this matrix is to write the following:
\[
A_{-i} = \sum_{j \neq i} \lambda_j z_j z_j^\top  \succeq \lambda_{i} z_1 z_1^\top +  \sum_{j > k, j \neq i} \lambda_j z_j z_j^\top =  X_\ktoinf^{(i)}(X_\ktoinf^{(i)})^\top.
\]
Thus, to lower bound $\mu_n(A_{-i})$ one can just lower bound $\mu_n(X_\ktoinf^{(i)}(X_\ktoinf^{(i)})^\top)$. This can be done by using Lemma \ref{lm::tight control of eigenvalues for independent coordinates} with matrix $X_\ktoinf^{(i)}$ instead of $X_\ktoinf$, which is valid because matrix $X_\ktoinf^{(i)}$ satisfies exactly the same assumptions, namely the matrix $X_\ktoinf^{(i)}\Sigma^{-1/2}_\ktoinf$ has independent centered $\sigma_x$-sub-Gaussian elements with unit variances.

Therefore, by Lemma  \ref{lm::tight control of eigenvalues for independent coordinates}  for some constant $c_1$ that only depends on $\sigma_x$ for any $i$ with probability at least $1 - c_1e^{-n/c_1}$,
\begin{align*}
\mu_n(A_{-i}) \geq& \sum_{i > k} \lambda_i - c_1\left(n\lambda_{k+1} + \sqrt{n\sum_{i > k}\lambda_i^2}\right)\\
\geq& \left(1 - c_1\rho_k(0)^{-1} - c_1\rho_k(0)^{-1/2}\right)\sum_{i > k} \lambda_i\\
=& n\lambda_{k+1}(\rho_k - c_1 - c_1\sqrt{\rho_k}),
\end{align*}
where we used Equations \eqref{eq::eigenvalue deviation via rho_k first} and \eqref{eq::eigenvalue deviation via rho_k second}. Choose a constant $b$ large enough depending on $c_1$, so that $\rho_k - c_1 - c_1\sqrt{\rho_k} \geq \rho_k/c_2$ for some constant $c_2$ that only depends on $\sigma_x$. 

By Lemma \ref{lm::sum of norms}, for some absolute constant $c_3$ for any $t \in (0,n)$, w.p.\ at least $1-2e^{-t/c_3}$, we have $\|z_i\|^2 \leq n - \sqrt{tn}\sigma_x^2 \leq n/2$, provided $t \leq n/(4\sigma_x^4).$ Combining it with the previous results and taking constant $c_4$ large enough depending on $\sigma_x$ and $c_2$ we get that if $\rho_k > c_4$ then for any $i$ with probability at least $1 - c_4 e^{-n/c_4}$,
 \[
 \frac{\lambda_i\bar{\theta}^2_i}{(1 + \lambda_i\mu_n( A_{-i}^{-1})\|z_i\|^2)^2} \geq \frac{1}{c_4} \frac{\lambda_i\bar{\theta}^2_i}{(1 + \frac{n\lambda_i}{n\lambda_{k+1}\rho_k})^2} = \frac{1}{c_4} \frac{\lambda_i\bar{\theta}^2_i}{(1 + \frac{\lambda_i}{\lambda_{k+1}\rho_k})^2}.
 \]
 Now we convert the high-probability lower bound for each term into the high-probability lower bound for the whole sum. Using Lemma \ref{lm::sum of non-neg lower bound whp} gives that with probability at least $1 - 2c_4e^{-n/c_4}$,
 \[
 \E_{\theta^*} B \geq \frac{1}{2c_4}\sum_i \frac{\lambda_i\bar{\theta}^2_i}{(1 + \frac{\lambda_i}{\lambda_{k+1}\rho_k})^2}.
 \]

Finally, by Theorem \ref{th:upper same as lower main body} there exists a constant $c_5$ that only depends on $b$ s.t. 
\[
\sum_{i} \frac{\lambda_i\bar{\theta}^2_i}{(1 + \frac{\lambda_i}{\lambda_{k+1}\rho_k})^2} \geq \frac{1}{c_5}\|\theta_\uptok\|_{\Sigma_\uptok^{-1}}^2\left(\frac{\sum_{i > k}\lambda_i}{n}\right)^2.
\]

Therefore, setting the constant $c$ large enough (depending on $b$ and $\sigma_x$) gives the result.
\end{proof}

\negreggeneralization*
\begin{proof}
In the following $c_1, c_2, \dots$ are constants that only depend on $\sigma_x$.

Let's introduce a new variable $\Diamond$ such that $\lambda = -\sum_{i > k} \lambda_i + \Diamond$. 

By Lemma \ref{lm::tight control of eigenvalues for independent coordinates} with probability at least $1 - c_1e^{-n/c_1}$,
\begin{align*}
\mu_1(A_k) = \lambda + \mu_1(X_\ktoinf X_\ktoinf^\top) \leq& \Diamond + c_1\left(n\lambda_{k+1} + \sqrt{n\sum_{i > k} \lambda_i^2}\right),\\
\mu_n(A_k) = \lambda + \mu_n(X_\ktoinf X_\ktoinf^\top) \geq& \Diamond - c_1\left(n\lambda_{k+1} + \sqrt{n\sum_{i > k} \lambda_i^2}\right).
\end{align*}

Let's put 
\begin{equation}
\label{eq::diamond range}
\sum_{i > k} \lambda_i > \Diamond > 2 c_1\left(n\lambda_{k+1} + \sqrt{n\sum_{i > k} \lambda_i^2}\right).
\end{equation}
Note that the range for  $\Diamond$ is non-empty if $\rho_k$ is large enough according to Equations \eqref{eq::eigenvalue deviation via rho_k first} and \eqref{eq::eigenvalue deviation via rho_k second}.
On the same event we get
\begin{align*}
\mu_n(A_k^{-1})^{-1} = \mu_1(A_k) \leq& \frac32 \Diamond,&\mu_n(A_k^{-1}) \geq \frac23\Diamond^{-1},\\
\mu_1(A_k^{-1})^{-1} = \mu_n(A_k)  \geq& \frac12\Diamond,&\mu_1(A_k^{-1}) \leq 2\Diamond^{-1}.
\end{align*}

Now we are in a position to use Theorem \ref{th::main upper main body}. Recall that $0 < \Diamond < \sum_{i > k}\lambda_i$. Thus
\[
\max(0, -\lambda) = -\lambda =  \sum_{i > k}\lambda_i - \Diamond \leq \sum_{i > k}\lambda_i .
\]

Note that results of Theorem \ref{th::main upper main body} still apply for the case when the expectation of the bias term is taken over the prior from assumption \ref{as::prior on theta^*}$(\bar{\theta})$. Indeed, as explained in the sketch of its proof, it decomposes very clearly into an algebraic and a stochastic part, where concentration results are applied. One can see that the only stochastic quantity that changes when the expectation over $\theta^*$ is taken is $\|X_\uptok\theta^*_\uptok\|^2.$ To obtain the result of the theorem one needs to show that $\E_\theta^*\|X_\ktoinf\theta^*_\ktoinf\|^2 \leq \tilde{c}\|\bar{\theta}_\ktoinf\|^2_{\Sigma_\ktoinf}$ with probability $1 - \tilde{c}e^{-n/\tilde{c}}$ for some $\tilde{c}$ that only depends on $\sigma_x$. This is indeed the case because expectations over $\theta^*$ of the squared components of $X_\ktoinf\theta^*_\ktoinf$ are i.i.d. sub-exponential random variables with expectation $\|\bar{\theta}_\ktoinf\|^2_{\Sigma_\ktoinf}$ and sub-exponential norm bounded by $\bar{c}\|\bar{\theta}_\ktoinf\|^2_{\Sigma_\ktoinf}$ for a constant $\bar{c}$ that only depends on $\sigma_x$. Thus, the desired concentration result holds by the same application of Bernstein's inequality as in Lemma \ref{lm::sum of norms}.

Thus, we can plug our bounds on eigenvalues into Theorem \ref{th::main upper main body} to get  that if $k < n/c_2$ then with probability at least $1 - c_1 e^{-n/c_1} - c_2 e^{-n/c_2}$,
\begin{align*}
\E_{\theta}B/c_2 \leq& \|\bar{\theta}_\ktoinf\|_{\Sigma_\ktoinf}^2\left(1 +  \frac{(2\Diamond^{-1})^2}{\left(\frac23\Diamond^{-1}\right)^2} + n\lambda_{k+1}(2\Diamond^{-1})\left(1 +(2\Diamond^{-1}) \sum_{i > k}\lambda_i\right) \right)\\
+& \|\bar{\theta}_\uptok\|_{\Sigma_\uptok^{-1}}^2\left(\frac{1}{n^2\left(\frac23\Diamond^{-1}\right)^2} +  \frac{\lambda_{k+1}}{n}\frac{(2\Diamond^{-1})}{\left(\frac23\Diamond^{-1}\right)^2}\left(1 +(2\Diamond^{-1}) \sum_{i > k}\lambda_i\right) \right), \\
V/c_2 \leq& \frac{(2\Diamond^{-1})^2}{\left(\frac23\Diamond^{-1}\right)^2}\frac{k}{n} + n (2\Diamond^{-1})^2 \sum_{i > k} \lambda_i^2.  \\
\end{align*}

Recall that $\Diamond < \sum_{i > k}\lambda_i$, so $1 + (2\Diamond^{-1})\sum_{i > k}\lambda_i$ is the same as $\Diamond^{-1} \sum_{i > k}\lambda_i$ up to a constant multiplier. That is, on the same event,
\begin{align}
B/c_3 \leq& \|\bar{\theta}_\ktoinf\|_{\Sigma_\ktoinf}^2\left(1 + \frac{n\lambda_{k+1}\sum_{i > k}\lambda_i}{\Diamond^2}\right)\label{eq::bias via Diamond first}\\
+& \|\bar{\theta}_\uptok\|_{\Sigma_\uptok^{-1}}^2\left(\frac{\Diamond^2}{n^2} + \frac{\lambda_{k+1}\sum_{i > k}\lambda_i}{n} \right), \label{eq::bias via Diamond second}\\
V/c_3 \leq& \frac{k}{n} +  \frac{n\sum_{i > k} \lambda_i^2}{\Diamond^2}. \label{eq::variance via Diamond}
\end{align}

One can see that $\Diamond$ balances the bias in the first $k$ components against two things: the bias in the tail and the variance. The value of $\Diamond$ that is optimal to balance the bias in the first $k$ components and the bias in the tail is $\sqrt{n\lambda_{k+1}\sum_{i > k}\lambda_i}$. As we will check further, up to a constant factor, $\Diamond$  will be in the range that we set in Equation \eqref{eq::diamond range}. There are two cases then: the first case is when this choice of $\Diamond$ is optimal because the variance is not larger than the bias. The second case is when $\Diamond$ needs to be chosen larger than $\sqrt{n\lambda_{k+1}\sum_{i > k}\lambda_i}$ to decrease the variance. So, consider two cases:
\begin{enumerate}
\item If the noise is small, meaning that
\[
v_\eps^2\frac{n\sum_{i > k} \lambda_i^2}{n\lambda_{k+1}\sum_{i > k}\lambda_i} \leq \|\bar{\theta}_\ktoinf\|_{\Sigma_\ktoinf}^2 +  \|\bar{\theta}_\uptok\|_{\Sigma_\uptok^{-1}}^2 \frac{\lambda_{k+1}\sum_{i > k}\lambda_i}{n},
\]
 then set
 \[
 \Diamond = a\sqrt{n\lambda_{k+1}\sum_{i > k}\lambda_i}
 \]
 for a constant $a$ that only depends on $\sigma_x$ that we will choose next.
 This $a$ must be such that Equation~\eqref{eq::diamond range} is satisfied, which means
 \begin{align*}
 a\sqrt{n\lambda_{k+1}\sum_{i > k}\lambda_i} \leq& \sum_{i > k}\lambda_i,\\
 a\sqrt{n\lambda_{k+1}\sum_{i > k}\lambda_i} \geq& 2 c_1\left(n\lambda_{k+1} + \sqrt{n\sum_{i > k} \lambda_i^2}\right).
 \end{align*}
 Using  $\sqrt{n\sum_{i > k} \lambda_i^2} \leq \sqrt{n\lambda_{k+1}\sum_{i > k}\lambda_i}$ we obtain that it is enough for $a$ to satisfy
 \begin{align*}
 a \leq& \rho_k(0)^{1/2},\\
 a \geq& 2 c_1\left(\rho_k(0)^{-1/2} + 1\right).
 \end{align*} 
One can see that $a = 4c_1$ satisfies this condition when $c > \max(1, 16c_1^2)$ since  $\rho_k(0) > c$. Taking such an $a$, plugging $\Diamond$ into Equations \eqref{eq::bias via Diamond first}--\eqref{eq::variance via Diamond}, and choosing $c_4$ big enough depending on $a, c_1, c_2, c_3$, we get that with probability at least $1 - c_4e^{-n/c_4}$,
\begin{align*}
B + v_\eps^2V \leq& c_4\left(\frac{k}{n} v_\eps^2 + v_\eps^2\frac{n\sum_{i > k} \lambda_i^2}{n\lambda_{k+1}\sum_{i > k}\lambda_i} + \|\bar{\theta}_\ktoinf\|_{\Sigma_\ktoinf}^2 +  \|\bar{\theta}_\uptok\|_{\Sigma_\uptok^{-1}}^2 \frac{\lambda_{k+1}\sum_{i > k}\lambda_i}{n}\right)\\
\leq& 2c_4\left(\frac{k}{n} v_\eps^2 +  \|\bar{\theta}_\ktoinf\|_{\Sigma_\ktoinf}^2 +  \|\bar{\theta}_\uptok\|_{\Sigma_\uptok^{-1}}^2 \frac{\lambda_{k+1}\sum_{i > k}\lambda_i}{n}\right),
\end{align*}
which implies the desired bound for any $c > 2c_4$.

\item 
If the noise is large, meaning that
\begin{equation}
\label{eq::negative regularization large noise}
v_\eps^2\frac{n\sum_{i > k} \lambda_i^2}{n\lambda_{k+1}\sum_{i > k}\lambda_i} > \|\bar{\theta}_\ktoinf\|_{\Sigma_\ktoinf}^2 +  \|\bar{\theta}_\uptok\|_{\Sigma_\uptok^{-1}}^2 \frac{\lambda_{k+1}\sum_{i > k}\lambda_i}{n},
\end{equation}
then set 
\[
\Diamond = a\sqrt{\frac{v_\eps}{\|\bar{\theta}_\uptok\|_{\Sigma_\uptok^{-1}}}n\sqrt{n\sum_{i > k} \lambda_i^2}}.
\]
for a constant $a$ that only depends on $\sigma_x$ that we choose next. As in the previous case, $a$ must be such that Equation \eqref{eq::diamond range} is satisfied, which means
 \begin{align*}
 a\sqrt{\frac{v_\eps}{\|\bar{\theta}_\uptok\|_{\Sigma_\uptok^{-1}}}n\sqrt{n\sum_{i > k} \lambda_i^2}} \leq& \sum_{i > k}\lambda_i,\\
 a\sqrt{\frac{v_\eps}{\|\bar{\theta}_\uptok\|_{\Sigma_\uptok^{-1}}}n\sqrt{n\sum_{i > k} \lambda_i^2}} \geq& 2 c_1\left(n\lambda_{k+1} + \sqrt{n\sum_{i > k} \lambda_i^2}\right).
 \end{align*}
 
 The first condition is satisfied whenever $a < \sqrt{c}$ due to Equation \eqref{eq::low noise condition for negative lambda bound}.
 Now consider the second condition. Because of Equation \eqref{eq::negative regularization large noise}, we have
 \begin{align}
 v_\eps\sqrt{n\sum_{i > k} \lambda_i^2} \geq& \|\bar{\theta}_\uptok\|_{\Sigma_\uptok^{-1}}\lambda_{k+1}\sum_{i > k}\lambda_i,\\
 \frac{\Diamond}{a}=\sqrt{\frac{v_\eps}{\|\bar{\theta}_\uptok\|_{\Sigma_\uptok^{-1}}}n\sqrt{n\sum_{i > k} \lambda_i^2}} \geq& \sqrt{n\lambda_{k+1}\sum_{i > k}\lambda_i}\label{eq::large noise increases optimal regularization}.
 \end{align}
 Thus, it is enough to satisfy
 \[
 a\sqrt{n\lambda_{k+1}\sum_{i > k}\lambda_i} \geq 2 c_1\left(n\lambda_{k+1} + \sqrt{n\sum_{i > k} \lambda_i^2}\right).
 \]
 This is exactly the same condition as in the previous case, so it can be reduced to
 \[
 a \geq 2c_1(\rho_k(0)^{-1/2} + 1).
 \]
 
Thus, just as in the small variance case, we see that since $c > \max(1, 16c_1^2)$ then $a = 4c_1$ satisfies both conditions. 

 Take such an $a$. Before plugging $\Diamond$ into Equations \eqref{eq::bias via Diamond first}--\eqref{eq::variance via Diamond}, note the following. Because of Equation \eqref{eq::large noise increases optimal regularization}, we have
 \begin{align*}
 \frac{n\lambda_{k+1}\sum_{i > k}\lambda_i}{\Diamond^2} \leq& \frac{1}{a^2},\\
 \frac{\Diamond^2}{n^2} \geq& a^2\frac{\lambda_{k+1}\sum_{i > k}\lambda_i}{n},
 \end{align*}
 which means that if we take $c_5$ large enough depending on $a$ and $c_3$, then Equations \eqref{eq::bias via Diamond first}--\eqref{eq::variance via Diamond} imply
 \begin{align*}
B/c_5 \leq& \|\bar{\theta}_\ktoinf\|_{\Sigma_\ktoinf}^2
+ \|\bar{\theta}_\uptok\|_{\Sigma_\uptok^{-1}}^2\frac{\Diamond^2}{n^2}, \\
V/c_5 \leq& \frac{k}{n} +  \frac{n\sum_{i > k} \lambda_i^2}{\Diamond^2}. 
 \end{align*}
Now plugging in the expression for $\Diamond$ gives that with probability at least $1 - c_1e^{-n/c_1}$,
 \begin{align*}
 B + v_\eps^2V \leq& c_5\left(\frac{k}{n}v_\eps^2 + (a^{-2} + a^2)v_\eps \|\bar{\theta}_\uptok\|_{\Sigma_\uptok^{-1}}\sqrt{\frac{\sum_{i > k} \lambda_i^2}{n}}  + \|\bar{\theta}_\ktoinf\|_{\Sigma_\ktoinf}^2\right),
 \end{align*}
 which implies the result for $c > \max((a^{-2} + a^2)c_5, c_1)$.
 \end{enumerate}
\end{proof}

\bibliography{benign_overfitting_in_ridge_regression}
\end{document}